\documentclass[preprint,12pt]{elsarticle}
\usepackage{mathrsfs}
\usepackage{amssymb}
\usepackage{amsmath,amscd}
\usepackage{amsfonts,amssymb}
\usepackage{latexsym}
\usepackage{amsbsy}
\usepackage{graphicx}
\usepackage{overpic}
\usepackage{mathdots}
\usepackage[all]{xy}
\usepackage{multirow}
\usepackage{tikz-cd}
\usepackage{lmodern}
\usepackage{blkarray, bigstrut}
\usepackage{booktabs}
\usepackage{diagbox}
\usepackage{arydshln}
\usepackage{changepage}
\usepackage{mathtools}
\usepackage{xcolor}
\newtheorem{thm}{Theorem}[section]
\newtheorem{theo}[thm]{Theorem}
\newtheorem{defn}[thm]{Definition}
\newtheorem{lem}[thm]{Lemma}
\newtheorem{prop}[thm]{Proposition}
\newtheorem{coro}[thm]{Corollary}
\newtheorem{claim}[thm]{Claim}
\newtheorem{exam}{Example}[section]
\newtheorem{examp}[exam]{Example}

\DeclareMathOperator{\rank}{rank}
\DeclareMathOperator{\ind}{ind}
\DeclareMathOperator{\initial}{in}
\DeclareMathOperator{\codim}{codim}
\DeclareMathOperator{\sing}{Sing}
\newenvironment{proof}{\noindent\textbf{proof.}}{\hfill$\square$\par}

\numberwithin{equation}{section}

\journal{-}

\begin{document}
\begin{frontmatter}
\title{Every type A quiver locus is a Kazhdan-Lusztig variety}
\author[1]{Jiajun Xu}
\author[1]{Guanglian Zhang\corref{cor1}}
\affiliation[1]{organization={School of Mathematical Sciences, Shanghai Jiao Tong University},
            city={Shanghai},
            postcode={200240},
            country={China}}

\cortext[cor1]{Corresponding author\\Email: s1gh1995@sjtu.edu.cn (Jiajun Xu), g.l.zhang@sjtu.edu.cn (Guanglian Zhang)\\
The work was supported by National Natural Science Foundation of China  (Grant No.12271347).}


\begin{abstract}
The Zariski closures of the orbits for representations of type A Dynkin quivers under the action of general linear groups (i.e. quiver loci) exhibit a profound connection with Schubert varieties. In this paper, we present a scheme-theoretical isomorphism between a type A quiver locus and the intersection of an opposite Schubert cell and a Schubert variety, also known as a Kazhdan-Lusztig variety in geometric representation theory. Our results generalize and unify the Zelevinsky maps for equioriented type A quiver loci and bipartite type A quiver loci, as presented respectively by A. V. Zelevinsky in 1985 and by R. Kinser and J. Rajchgot in 2015. Through this isomorphism, we establish a direct and natural connection between type A quiver loci and Schubert varieties. We compute an explicit relationship between Zelevinsky permutations and the indecomposable factors of the corresponding representations. Additionally, we present three applications of our isomorphism with examples in order to justify its further potentials.

\end{abstract}
\begin{keyword}
Quiver locus \sep partial flag variety \sep Schubert variety \sep Kazhdan-Lusztig variety
\MSC 14M15 \sep 16G20 \sep 05E10

\end{keyword}

\end{frontmatter}

\section{Introduction}

Quiver loci are the orbit closures of quiver representations under the action of general linear groups. In recent decades, significant connections have been revealed between the geometry of the quiver loci for Dynkin quivers and that of Schubert varieties. In 1985, A. V. Zelevinsky established a bijection, known as the Zelevinsky map, between the quiver locus of equioriented type A quiver and an open dense subvariety in Schubert variety \cite{zelevinskii_two_1985}, in order to address the cohomology problem. Lusztig also presented a similar viewpoint for cyclic quivers in the 1980s to prove the existence of canonical bases for quantum groups \cite{lusztig1990cir}. Lakshmibai and Peter Magyar subsequently conducted a detailed study of the Zelevinsky map, emphasizing that it is a scheme-theoretical isomorphism\cite{lakdegesche1998}. After that, the equivalence between the singularity properties of quiver loci and Schubert varieties was further explored in many aspects \cite{bobinski_normality_2001}, \cite{bobinski_schubert_2002}. In 2015, Ryan Kinser and Jenna Rajchgot extended the Zelevinsky map from equioriented type A quivers to bipartite type A quivers, providing a richer pattern for studying the relationship between type A quiver loci and Schubert/Kazhdan-Lusztig varieties. They showed that every quiver locus of bipartite type A quiver is scheme-theoretically isomorphic to a Kazhdan-Lusztig variety, which is an intersection of a Schubert variety and an opposite Schubert cell. Consequently, it follows that the type A quiver locus of arbitrary orientation must be isomorphic to an open subvariety of some Kazhdan-Lusztig variety, up to a factor of product of general linear groups: in order to facilitate the comparison of our new results with this approach in the rest of this paper, we refer to it as the K-R's approach. Ryan Kinser and Jenna Rajchgot also generalized their results to type D quivers in 2021\cite{kinser_type_2021}. In our paper, we aim to construct a scheme-theoretical isomorphism from the type A quiver locus to a Kazhdan-Lusztig variety directly. Moreover, this isomorphism is expected to unify the Zelevsinky maps for equioriented and bipartite type A quivers.
$$\begin{tikzcd}
  Q=1 \arrow[r, "\beta_1"'] & 2 \arrow[r, "\beta_2"'] & 3 & 4 \arrow[l,"\alpha_1"] & 5 \arrow[l,"\alpha_2"] \arrow[r,"\beta_3"'] & 6& 7 \arrow[l,"\alpha_3"]\end{tikzcd}$$

Let $Q=(Q_0,Q_1,s,t)$ be a Dynkin type A quiver, where $Q_0$ is the set of vertices (let $Q_0=\{1,2,\cdots,n\}$ if $Q$ has $n$ vertices), $Q_1$ is the set of arrows, and $s,t:Q_1\to Q_0$ are maps from arrows to their start and target points. Given the dimension vector $\mathbf{d}=(d_x)_{x\in Q_0}$ of $Q$, the representation space $rep_Q(\mathbf{d})$ is an affine space (over a algebraically closed field $k$, the same below). The algebraic group $\mathrm{GL}(\mathbf{d})=\prod_{x\in Q_0}\mathrm{GL}(d_x)$ acts on $rep_Q(\mathbf{d})$ and for any $V\in rep_Q(\mathbf{d})$ we denote the $\mathrm{GL}(\mathbf{d})$-orbit of $V$ by $\mathcal{O}_V$. Its Zariski closure $\overline{\mathcal{O}_V}$ in $rep_Q(\mathbf{d})$ is the quiver locus. Let $G=\mathrm{GL}(N)$, where $N=\sum_{x\in Q_0} d_x$, and take $B$ and $B^-$ to be the upper and lower triangular Borel subgroup of $G$ respectively. Our main theorem is as follows.

\begin{theo}[Theorem \ref{MAIN}]\label{introtheo1}
  For any type A quiver $Q$ with dimension vector $\mathbf{d}$, there is a Zelevinsky map $$\zeta_Q:rep_Q(\mathbf{d})\to B^-v_QP_Q/P_Q,$$ where $v_Q$ is a permutation in the Weyl group $W(G)\cong S_N$ of $G$, and $P_Q$ is a block upper triangular parabolic subgroup of $G$ containing $B$. Both of $v_Q$ and $P_Q$ depend only on $Q$ and $\mathbf{d}$, and the restriction of $\zeta_Q$ on $\overline{\mathcal{O}_V}$ is a scheme-theoretical isomorphism onto a Kazhdan-Lusztig variety $\overline{BwP_Q/P_Q}\cap B^-v_QP_Q/P_Q$, where $w$ depends on $V$ (that is, $w_Q(\mathbf{r})$ in Theorem \ref{MAIN}, called the Zelevinsky permutation).
\end{theo}

This result provides a straightforward interpretation of the close connection between quiver loci and Schubert/Kazhdan-Lusztig varieties, and in our paper we also point out several important concepts among the immediate corollaries. Firstly, the normality, Cohen-Macaulay property, and rational singularity property of type A quiver loci can be deduced, recovering the works of Bobiński and Zwara in \cite{bobinski_normality_2001}, as well as those of Kinser and Rajchgot in \cite{kinser_type_2015}. Furthermore, our construction may also enhance existing proofs in problems related to the theory of representations of type A quiver. For example, the vanishing of intersection cohomology of type A quiver loci in odd degrees and existence of canonical bases of quantum groups, etc.(cf.\cite{zelevinskii_two_1985},\cite{lusztigcanbase},\cite{lusztig1990cir}). Our result revisits these proofs with a more natural interpretation and avoids some lengthy steps.

Moreover, in future research, this construction allows us to parallel the research methods of both sides, and quickly utilize certain mature techniques from one side to deal with difficult problems on the other side. To achieve this, we need more precise information about Zelevinsky permutations to address the requirements in detailed calculations. We provide a direct and clear connection between the combinatorial data of Zelevinsky permutations and the indecomposable multiplicities of quiver representations, as follows. In fact, this connection also serves as a preparation (see discussion in Section \ref{singloc}) for our future studies.

\begin{theo}[Theorem \ref{m-r}]\label{introtheo2}
  Given type A quiver $Q$ with $n$ vertices and $V\in rep_Q(\mathbf{d})$, we assume $V=\oplus m_{pq}I_{pq}$, where $I_{pq},1\leq p\leq q\leq n$ are the indecomposable representations of $Q$. Let $w$ be the Zelevinsky permutation corresponding to $V$, i.e. $\overline{\mathcal{O}_V}\cong \overline{BwP_Q/P_Q}\cap B^-v_QP_Q/P_Q$. For permutation matrix $w$, there exists a suitable (see Definition \ref{defBM} and Section \ref{constructionofthemapsection}) block partitioning of $w$ that depends only on $Q$, and each block row and block column is labeled by some vertex in $Q_0=\{1,2,\cdots,n\}$, such that the number of non-zero elements of $w$ in the block at the row labeled by vertex $x$ and the column labeled by vertex $y$ is $$\left\{\begin{array}{ll}
                                      m_{yx} &,\mbox{ if }y\leq x  \\
                                      n_{xy}&,\mbox{ if }y=x+1 \\
                                      0 &,\mbox{ else}
                                    \end{array}\right.$$where $n_{xy}=\sum_{p\leq x<y\leq q} m_{pq}$.
\end{theo}

Of course, we should also emphasize the explanation of how this technique expands upon the existing literature. K-R's approach has also produced some results for arbitrary oriented A-type quivers in the past (see \cite{kinser_type_2015}, \cite{threeformulas}). There is a great deal of potential research that can stem from establishing direct isomorphisms and having explicit local equations instead of just an open immersion. Therefore, exploring these applications will further justify our approach.

To illustrate the application of our results, we first focus on the verification of the geometric vertex decomposability of the defining ideal of arbitrary oriented type A quiver loci. This application heavily relies on the selection of local equations, giving our results a significant advantage in this regard. Klein and Rajchgot (\cite{GVDandliaison}, 2016) have previously proven the geometric vertex decomposability of the defining ideal of equioriented type A quiver loci. Using our direct identification of this ideal of arbitrary oriented quivers to a Kazhdan-Lusztig ideal (Section \ref{G/BKL}), we can provide a simpler proof for the geometric vertex decomposability of arbitrary oriented quivers (Section \ref{GVD}).

Moreover, our approach also has advantages in solving problems that require specific details in calculations, such as checking the smoothness of type A quiver loci at specific points or computing the singular loci. Due to space limitations and the coherence of our logical presentation, we will only provide a weak criterion in this paper and illustrate the applications of our results with some examples (Section \ref{singloc}). In Section \ref{G/BKL}, we highlight a perspective of achieving isomorphisms using full flags. Then combining this perspective with the work of \cite{maxsingSL/B}, the calculations of singular loci of type A quiver loci can also be completed. However, this rough combination may overlook some integral information, such as the information provided in Theorem \ref{introtheo2}. Therefore, better results require further research and summarization, which we will present in a new article.

Moreover, our approach has advantages in solving problems that require specific details in calculations, for example checking the smoothness of type A quiver loci at specific points or computing the singular loci. Due to space limitations and coherence of logic, we will only provide a weak criterion in this paper and illustrate the power of our results with some examples in Section \ref{singloc}. In Section \ref{G/BKL}, we point out a perspective of achieving isomorphisms using full flags, and combining with the work of \cite{maxsingSL/B}, calculations of singular loci of type A quiver loci can also be completed. However, this rough combination will overlook some information that is preserved integrally, such as the information in Theorem \ref{introtheo2}. Therefore, better results require further research and summarization, which we will present in a new article.

Another aspect that can demonstrate the application potential of our results is its generalization to type D quivers, as discussed in Section \ref{gentoD}. It provides a novel non-trivial isomorphism method that exhibits higher consistency with type A isomorphisms and holds promise for generalization to other Dynkin types.

The rest of this paper is organized as follows.

In Section 2, we introduce some basic notations and preliminaries for the type A quiver loci, particularly the characterization of its degeneration in rank parameters.

In Section 3, we discuss the Kazhdan-Lusztig variety in an opposite Schubert cell and cite the important defining ideal of such variety in affine neighborhood. Section 2 and 3 serves as preparations and do not present any new results.

In Section 4, we present the construction of our Zelevinsky map for arbitrary oriented type A quiver loci. The main theorem is Theorem \ref{MAIN}, with Theorem \ref{main} as a central step. We prove that our construction provides a scheme-theoretical isomorphism between any type A quiver loci and certain Kazhdan-Lusztig varieties. Moreover, some immediate corollaries in geometric representation theory are deduced.

In Section 5, we compute the explicit relationship between Zelevinsky permutations and the indecomposable factors of the corresponding representations. This result provides important combinatorial data on Zelevinsky permutations and serves as preparation for future works.

In Section 6, we first discuss the view of achieving the Zelevinsky isomorphism through full flag varieties (Section \ref{G/BKL}). Then, in Section \ref{GVD}, we check the geometric vertex decomposability of the defining ideal of type A quiver loci as an application of our new isomorphism. In Section \ref{singloc}, we study the singularities of type A quiver loci and propose a weak criterion for the smoothness of type A quiver loci at specific points. We also claim a stronger criterion without proof due to space limitations. Finally, in Section \ref{gentoD}, we discuss the generalization of our isomorphism to type D quivers. Sections \ref{singloc} and \ref{gentoD} mainly serve to demonstrate the application direction of our main Theorem \ref{introtheo1}, and their complete content will be developed in future articles. However, the examples provided in these parts are sufficient to indicate the prospect of further research.

\section{Type A quiver loci}
For a quiver $Q=(Q_0,Q_1,s,t)$, its representation space $rep_Q(\mathbf{d})$ for a fixed dimension vector $\mathbf{d}$ can be identified as an affine matrix space $$rep_Q(\mathbf{d})\cong \prod_{\alpha\in Q_1} \mathbb{A}^{d_{s(\alpha)}\times d_{t(\alpha)}},$$
any representation in $rep_Q(\mathbf{d})$ can be denoted by $V=(V_\alpha)_{\alpha\in Q_1}$. Here, the matrix $V_\alpha$ corresponds to the arrow $\alpha$ in the quiver and has dimensions $d_{s(\alpha)}\times d_{t(\alpha)}$. Without loss of generality, we always assume $d_x\neq 0$ for all $x\in Q_0$.

Let $\mathrm{GL}(\mathbf{d})=\prod_{x\in Q_0} \mathrm{GL}(d_x)$, it acts on $rep_Q(\mathbf{d})$ by
$g.V=(g_{t(\alpha )}V_\alpha g^{-1}_{s(\alpha)})_{\alpha\in Q_1}$

The Zariski closure of a $\mathrm{GL}(\mathbf{d})$-orbit in $rep_Q(\mathbf{d})$ is called a quiver locus. If $Q$ is a Dynkin type A quiver, such closure variety is called type A quiver locus.

For a quiver $Q$ of Dynkin type A, we index the vertices as $1,2,...,n$, where vertex 1 represents the leftmost vertex, and vertex $n$ represents the rightmost vertex (i.e. we set the left and right directions artificially). The arrows in $Q$ can be classified as left arrows or right arrows, and we denote the left arrows by $\alpha_i$ and the right arrows by $\beta_i$ (where the value of $i$ always increases from left to right). A representation of $Q$ is denoted by $V=(A_\cdot ,B_\cdot)$, where the maps on the left arrows $\alpha_i$ (resp. right arrows $\beta_i$) are represented by $A_i$ (resp. $B_i$). Please see Example \ref{ex1} and Example \ref{Zex} for further clarification.

The source and sink vertices are called critical points in $Q$. We also denote these critical points by $s_1,s_2,\cdots$ from left to right, see also Example \ref{Zex}.
\subsection{Degeneration of type A quiver loci}

Given a quiver $Q$, consider any integers $1\leq a<b\leq n$, we say it determines an interval $[a,b]$ of $Q$ and write $[a,b]\subset Q$.  We define $Q_{[a,b]}$ to be the subquiver obtained by cutting $Q$ from $a$ to $b$. Now for $[a,b]$ an interval of $Q$, let $Sou_{[a,b]}$ be the set of source vertices in $Q_{[a,b]}$ and $Sin_{[a,b]}$ be the set of sinks in $Q_{[a,b]}$, note that these two sets contains not only the critical points of $Q$ but also the boundary vertices $a$ or $b$.

Given any representation $V$ with dimension vector $\mathbf{d}$, we can define $$\varphi^V_{[a,b]}:\mathop{\oplus}\limits_{x\in Sou_{[a,b]}}V_x\to\mathop{\oplus}\limits_{y\in Sin_{[a,b]}}V_y, $$
whose components are given by$$\begin{array}{ccc}V_{x_1}\oplus V_{x_2}&\to& V_y\\(v_1,v_2)&\mapsto &V_{P_1}v_1+V_{P_2}v_2\end{array}$$
where $x_1,x_2$ are two sources adjacent to sink $y$, $P_i (i=1,2)$ denotes the path from $x_i$ to $y$ , and $V_{P_i}$ is the composition of all the matrices placed on the path $P_i$, i.e., the composition of all the maps $V_{\alpha}$ going from the source $x_i$ to the sink $y$ (in the obvious order).

Let $r_{[a,b]}(V)=\rank \varphi_{[a,b]}^V$. There is a classic result for the degenerations of type A quiver loci.

\begin{theo}[cf.\cite{degenerationofA}]
  $V,W$ are in the same $\mathrm{GL}(\mathbf{d})$-orbit if and only if $r_{[a,b]}(V)=r_{[a,b]}(W)$ for all interval $[a,b]$ of $Q$. Additionally, $V$ lies in the closure of orbit of $W$ if and only if $r_{[a,b]}(V)\leq r_{[a,b]}(W)$ for all interval $[a,b]$ of $Q$.
\end{theo}
Based on this theorem, we can characterize the orbits and orbit closures using \textbf{rank parameters}. Given a vector of integers $\mathbf{r}=(r_{[a,b]})_{[a,b]\subset Q}$, if there exists a representation $V\in rep_Q(\mathbf{d})$ such that $r_{[a,b]}=(r_{[a,b]}(V)$ for all intervals $[a,b]$ of $Q$, we say $\mathbf{r}$ is a \textbf{rank parameter}, and denote $\mathcal{O}_\mathbf{r}=\mathcal{O}_V$ as the orbit of $V$. We say two rank parameters $\mathbf{r'}\leq \mathbf{r}$ if and only if $\mathcal{O}_{\mathbf{r'}}\subset \overline{\mathcal{O}_{\mathbf{r}}}$.
\begin{examp}\label{ex1}
  Consider the quiver and representation of quiver given by the following data($1,2,\cdots$ are indexes of vertices, not dimension vector):
  $$\begin{tikzcd}
  V=V_1 \arrow[r, "B_1","\beta_1"'] & V_2 \arrow[r, "B_2","\beta_2"'] & V_3 & V_4 \arrow[l,"\alpha_1","A_1"'] & V_5 \arrow[l,"\alpha_2", "A_2"'] \arrow[r,"\beta_3"', "B_3"] & V_6& V_7 \arrow[l,"\alpha_3","A_3"']
\end{tikzcd}$$
Its orbit is characterized by the rank of following matrices, which are matrices of $\varphi^V_{[a,b]}$ under standard basis.
$$B_1,B_2,A_1,A_2,B_3,A_3,B_2B_1,[A_1,B_2],A_1A_2,\left[\begin{matrix}
                                                      B_3 \\
                                                      A_2
                                                    \end{matrix}\right],[A_3,B_3],$$ $$[A_1,B_2B_1],[A_1A_2,B_2],\left[\begin{matrix}
          B_3 \\
          A_1A_2
        \end{matrix}\right],\left[\begin{matrix}
          A_3& B_3 \\
          0 &A_2
        \end{matrix}\right],$$ $$[A_1A_2,B_2B_1],\left[\begin{matrix}
          B_3& 0 \\
          A_1A_2 & B_2
        \end{matrix}\right],\left[\begin{matrix}
          A_3& B_3 \\
          0 & A_1A_2
        \end{matrix}\right],\left[\begin{matrix}
          B_3& 0 \\
          A_1A_2 & B_2B_1
        \end{matrix}\right],\left[\begin{matrix}
          A_3& B_3 & 0\\
          0 & A_1A_2&B_2
        \end{matrix}\right],$$$$\left[\begin{matrix}
          A_3& B_3 & 0\\
          0 & A_1A_2&B_2B_1
        \end{matrix}\right]$$
\end{examp}

To clarify the defining ideal of an orbit closure $\overline{\mathcal{O}_\mathbf{r}}$, we require some polynomial functions on $rep_Q(\mathbf{d})$ to describe the rank conditions on matrices. The affine coordinate ring $k[rep_Q(\mathbf{d})]$ is a polynomial ring generated by affine coordinates, which are functions that represent matrix entries. Let $f_{\alpha_i},f_{\beta_i},i=1,2,\cdots$ be matrices of affine coordinates such that $f_{\alpha_i}(V)=A_i$ and $f_{\beta_i}(V)=B_i$ for $V=(A_\cdot,B_\cdot)$.

In this paper, we refer to a subset $S\subset Q_0$ equipped with some total order as a \textbf{sequence} of vertices. Throughout this paper, we use the symbol $\prec$, such as ``$x\prec y$ in $S$", to denote the order relation in totally ordered set (sequence) of vertices to distinguish it from the symbol $``<"$ which typlically indicates the order of integers in conventional usage. We present a sequence as $$S=x\prec y\prec z\prec \cdots,$$ and denote the index of an element $x\in S$ in increasing order by $\ind_S(x)$, such as $\ind_S(x)=1,\ind_S(z)=3$ in the sequence above.

\begin{defn}\label{defBM}
  Given two sequences $S',S$ of vertices, i.e. equipping two sets $S',S\subset Q_0$ with some total orders respectively, a \textbf{$(S',S)$-blocked} matrix is a block matrix such that:
  \begin{itemize}
    \item[(BM1)] its block rows are labeled from north to south (top to bottom) by $S'$ in ascending order;
    \item[(BM2)] its block columns are labeled from west to east (left to right) by $S$ in ascending order;
    \item[(BM3)] the block row (resp. column) labeled by vertex $x$ has height (resp. width) $d_x$.
  \end{itemize}
\end{defn}

In such a matrix, we denote the row/column labeled by $x$ as row/column $[x]$, and we refer to the submatrix (block) placed at the intersection of row $[x]$ and column $[y]$ as its $([x],[y])$-block. It is important to note that labeling a matrix as $(S',S)$-blocked is simply a convenient way to describe it block by block, and does not imply any special properties. Any matrix of the appropriate size can be considered a $(S',S)$-blocked matrix.

Now, for an interval $[a,b]$ of $Q$, we assign total orders to the sets $Sou_{[a,b]}$ and $Sin_{[a,b]}$ by using the descending order of the integer indices of the vertices. We define a $(Sin_{[a,b]},Sou_{[a,b]})$-blocked matrix $M^Q_{[a,b]}$, where its $([x],[y])$-block is given by
$$\left\{\begin{array}{ll}
          f_{\alpha_i}f_{\alpha_{i+1}}\cdots f_{\alpha_j}, & \mbox{if $x<y$ and there is path}\begin{tikzcd}
x & x+1 \arrow[l, "\alpha_i"] & \cdots \arrow[l, "\alpha_{i+1}"] & y \arrow[l, "\alpha_j"]
\end{tikzcd} \\
           f_{\beta_j}\cdots f_{\beta_{i+1}} f_{\beta_i} , & \mbox{if $x>y$ and there is path}\begin{tikzcd}
y \arrow[r, "\beta_i"'] & y+1 \arrow[r, "\beta_{i+1}"'] & \cdots \arrow[r, "\beta_j"'] & x
\end{tikzcd} \\
          0,& \mbox{if there is no path from $y$ to $x$.}
         \end{array}\right. $$

Every entry in such a matrix is a function on $rep_Q(\mathbf{d})$. Therefore, $M^Q_{[a,b]}$ can be evaluated at $V$, and $M^Q_{[a,b]}(V)$ is exactly the matrix of $\varphi_{[a,b]}(V)$ under the standard basis. Every minor of $M^Q_{[a,b]}$ is a polynomial of affine coordinates of $rep_Q(\mathbf{d})$. According to the work of Riedtmann and Zwara\cite{definingidealofA} we have the defining ideal $I_\mathbf{r}$ of the orbit closure $\mathcal{O}_{\mathrm{r}}$ is generated by
$$<\mbox{minors of size $r_{[a,b]}+1$ in $M^Q_{[a,b]}$}|[a,b]\subset Q>.$$

\section{Kazhdan-Lusztig varieties}\label{KL}
In this paper, our focus is on a specific type of Kazhdan-Lusztig variety. To gain a more comprehensive understanding of both the Kazhdan-Lusztig variety and the Schubert variety, we recommend referring to \cite[Chap.8]{springer_linear_1998}, \cite{reviewsingofschu}, \cite{singofsch}, and \cite{kazhdan_representations_1979}.

Let us consider $N$ to be the sum of $d_x$ for every vertex $x$ in $Q_0$, i.e., $N=\sum_{x\in Q_0} d_x$. Additionally, let $G$ be the group $\mathrm{GL}(N)$. We define $S$ as a sequence of vertices obtained by applying some certain total order to $Q_0$. We fix the upper triangular subgroup of $G$ as the standard Borel subgroup $B$, and let $P$ to be the parabolic subgroup consisting of block upper triangular $(S,S)$-blocked matrices. The quotient variety $G/P$ is known as a flag variety.

We define $N_i$ to be the sum of $d_x$ over the first $i$ elements in the totally ordered set $S$, where $i = 1, 2, \cdots, n$, and $N_0=0$. For example, $$S=2\prec 5\prec 4\prec \cdots, N_2=d_2+d_5, N_3=d_2+d_5+d_4.$$Then, associated with the maximal torus of diagonal matrices, $W_G\cong S_N$ (the symmetric group) and $W_P\cong \prod_{i=1}^{n}S_{N_i-N_{i-1}}$ as the Weyl groups of $G$ and $P$ respectively. The orbits of the natural $B$-action on $G/P$ are indexed by $W_G/W_P$, or equivalently, by the set of shortest representatives within the cosets of $W_G/W_P$. These shortest representatives can be collected as $W^P:=$
$$\{w=(w_1,w_2,\cdots,w_N)\in S_N|w_{N_{i-1}+1}<w_{N_{i-1}+2}<\cdots<w_{N_i},\forall 1\leq i\leq n\}.$$
where $w=(w_1,w_2,\cdots,w_N)$ is the one-line notation of permutation $w$ that maps $j$ to $w_j$ for $j=1,2,\cdots,N$. Additionally, We consider any permutation $w$ as a permutation matrix with the $(i,j)$-elements given by
$$\left\{\begin{array}{ll}
           1,& \mbox{if }i=w_j \\
           0, & \mbox{else}
         \end{array}\right.$$

For $w\in W^P$, the $B$-orbit $BwP/P$ is called a Schubert cell, denoted by $X_w^\circ$. The closure of $X_w^\circ$ is the Schubert variety $X_w$. Similarly, for the opposite Borel $B^-$ of lower triangular matrices, we have the opposite Schubert cell $X^v_\circ=B^-vP/P$ where $v\in W^P$. The intersection of a Schubert variety and an opposite cell, i.e., $X_w\cap X^v_\circ$ is called a Kazhdan-Lusztig variety, denoted by $Y_w^v$. In this paper, we focus on a specific type of opposite cell $X^v_\circ$, where the index $v$ has a special form, as described below (also see Example \ref{Zex} for a specific instance).

Consider the permutation $v=(v_1,v_2,\cdots,v_N)\in W^P$ such that $v$ can be viewed as a $(S',S)$-blocked matrix, where $S'$ is a sequence of vertices formed by equipping $Q_0$ with some new order, and the $([x],[y])$-blocks of $v$ are given by
\begin{equation}\label{vQ}
\left\{\begin{array}{ll}
           I_{d_x}, & \mbox{if }x=y \\
           0, & \mbox{if }x\neq y
         \end{array}\right.\end{equation}
 We identify the opposite Schubert cell $X^v_\circ$ with the space (denoted by $\mathbb{A}^v$) of $(S',S)$-blocked matrices whose $([x],[y])$-blocks are determined by
 \begin{equation}\label{AQ}\left\{\begin{array}{ll}
            I_{d_x}, &\mbox{if }x=y \\
            0, & \mbox{if }x\prec y \mbox{ in } S\mbox{ or }S'   \\
            \mbox{undetermined}, &\mbox{ else}
          \end{array}\right. \end{equation}

Thus, the opposite cell $X^v_\circ$ is isomorphic to the affine space $\mathbb{A}^{\dim G/P-l(v)}$. In this paper we always identify the opposite cell with this space of matrices. Therefore, the coordinate ring $k[X^v_\circ]$ is the polynomial ring of the affine coordinates corresponding to the ``undetermined" area of matrices. We replace the entries in the ``undetermined" area of a matrix in such matrix space with the corresponding affine coordinate functions and denote the result by $Z$ (called the generic matrix).

Given any matrix $A$, we denote by $A_{i\times j}$ the southwest submatrix of $A$ consisting of entries in the southernmost $i$ rows and westernmost $j$ columns. At this time, if the southernmost $i$-th row (resp. westernmost $j$-th column) is exactly the closed north boundary of block row $[x]$ (resp. closed east boundary of block column $[y]$), i.e $i=\sum\limits_{\ind_{S'}(x')\geq \ind_{S'}(x)}d_{x'}$ (resp. $j=N_{ind_S(y)}$), we also denote $A_{i\times j}=A_{[x]\times j}$ (resp. $A_{i\times [y]}$), and similarly we define $A_{[x]\times[y]}$ to be the southwest submatrix of $A$ with row $[x]$ and column $[y]$ as the closed boundary. Note that $A_{[x]\times[y]}$ can be viewed as a $(\tilde{S'},\tilde{S})$-blocked matrix with $\tilde{S'}\subset S',\tilde{S}\subset S$.

The Bruhat order on $W^P$ is defined as $\tau\leq w$ if and only if $X_\tau\subset X_w$. This order can be described by the rank relations: $\tau\leq w$ if and only if
$$\rank \tau_{i\times [y]}\leq w_{i\times [y]},\forall i\in \{1,2,\cdots,N\},y\in Q_0.$$
The Kazhdan-Lusztig variety $Y^v_w$ is non-empty if and only if $v\leq w$ in $W^P$. The defining ideal $I_w$ of Kazhdan-Lusztig variety $Y_w^v$ in $k[X^v_\circ]$ is generated by (cf. \cite{kinser_type_2015},\cite{ESS})
$$\mbox{minors of size $\rank w_{i\times [y]}+1$ in $Z_{i\times [y]}$}, i\in \{1,2,\cdots,N\},y\in Q_0.$$

Furthermore, using the technique of Fulton's essential box \cite{ESS}, we can simplify the generator set. Here, we recall the row labels $S'$ of $v$. We denote by $P'$ the parabolic subgroup of block upper triangular $(S',S')$-blocked matrices and let $W_{P'}$ be its Weyl group. Particularly, when the index $w$ is the shortest representative in some $\tilde{w}W_P$ such that $\tilde{w}$ is the unique longest element in the double coset $W_{P'}\tilde{w}W_P$, we have
$$I_w=<\mbox{minors of size $\rank w_{[x]\times [y]}+1$ in $Z_{[x]\times [y]}$}| x,y\in Q_0>.$$

To observe this, one may apply \cite[Lemma 3.10]{ESS} to $\tilde{w}$ (or see our interpretation later). We say that such $w$ is of \textbf{Z-type}, where ``Z" indicates a further definition of ``Zelevinsky permutation" in Section 4.2.

\begin{defn}\label{ztype}
  A permutation $w$ is called of Z-type (corresponding to $(S',S)$-blocking), if one of the following equivalent conditions holds:
  \begin{itemize}
    \item[(1)] $w$ is the shortest element in some coset $\tilde{w}W_P$ such that $\tilde{w}$ is the unique longest element in the double coset $W_{P'}\tilde{w}W_P$.
    \item[(2)] $w$ is the unique maximal element in $W_{P'}w\cap W^P$ under the Bruhat order on $W^P$.
    \item[(3)] the nonzero entries in any block of $w$ form an identity matrix (of size may smaller than the block), and in the same block row such identity matrices (from different blocks) are arranged from southwest to northeast, in the same block column such identity matrices are arranged from northwest to southeast.
  \end{itemize}
\end{defn}

For any minor of size $\rank w_{i\times [y]}+1$ in $Z_{i\times [y]}$, where $w$ is a Z-type permutation, it is observed that either
\begin{itemize}
  \item $\exists k\geq 0$, satisfying $\rank w_{i\times [y]}=\rank w_{(i-k)\times [y]}+k$, and there is $x\in S'$ such that the southwest ${(i-k)\times j}$ area is exactly the same as southwest ${[x]\times[y]}$ area. At this time the minor of size $\rank w_{i\times j}+1$ in $Z_{i\times j}$ must be generated by minors of size $\rank w_{[x]\times [y]}+1$ in $Z_{[x]\times [y]}$. Or,
  \item $\exists k>0$, satisfying $\rank w_{i\times j}=\rank w_{(i+k)\times k}$, and there exists $x\in S'$ such that the southwest $(i-k)\times j$ area is exactly the same as southwest ${[x]\times[y]}$ area. At this time any minor of size $\rank w_{i\times j}+1$ in $Z_{i\times j}$ is exactly a minor of size $\rank w_{[x]\times [y]}+1$ in $Z_{[x]\times [y]}$.
\end{itemize}
That explains the spirit of Fulton's essential box in our specific situation.
\section{Zelevinsky map}
\subsection{Construction of the map}\label{constructionofthemapsection}
For two sequences of vertices $S=a_1\prec a_2\prec \cdots\prec a_s$ and $S'=b_1\prec b_2\prec \cdots\prec b_{s'}$ with empty intersection, a new sequence $S\triangleleft S'$ can be created by ordering the union of $S$ and $S'$ as $a_1 \prec a_2 \prec \cdots \prec a_s \prec b_1 \prec b_2 \prec \cdots \prec b_{s'}$.

For a quiver $Q$ of type A, any path can be naturally viewed as a sequence of vertices, where the order is determined by $s(\alpha) \prec t(\alpha)$ for every arrow $\alpha$ in the path. In quiver $Q$ we denote the maximal paths from right to left by $L_1,L_2,\cdots,L_l$ and the maximal paths from left to right by $R_1,R_2,\cdots,R_r$(the subscripts are always counted from left). The start and target vertices of a path $P$ can be obtained using the notations $s(P)$ and $t(P)$, respectively.

From now on, let

$S'_Q=$ $$R_1\backslash t(R_1)\triangleleft R_2\backslash t(R_2)\triangleleft \cdots \triangleleft R_r\backslash t(R_r)\triangleleft \{n\}\triangleleft L_l\backslash s(L_l)\triangleleft L_{l-1}\backslash s(L_{l-1})\triangleleft \cdots \triangleleft L_1\backslash s(L_1),$$

$S_Q=$ $$L_l\backslash t(L_l)\triangleleft L_{l-1}\backslash t(L_{l-1})\triangleleft \cdots \triangleleft L_1\backslash t(L_1)\triangleleft \{1\}\triangleleft R_1\backslash s(R_1)\triangleleft R_{2}\backslash s(R_{2})\triangleleft \cdots \triangleleft R_r\backslash s(R_r).$$

Let $P_Q$ be the parabolic subgroup of block upper triangular $(S_Q,S_Q)$-blocked matrices.

Given any representation $V=(A_\cdot,B_\cdot)\in rep_Q(\mathbf{d})$, we define a $(S'_Q,S_Q)$-blocked matrix with $([x],[y])$-blocks given by

$$\left\{\begin{array}{ll}
           I_{d_x}, &\mbox{if }x=y  \\
           0, & \mbox{if there is no path from $y$ to $x$} \\
          A_i,& \mbox{if $y=x+1$ and }\begin{tikzcd}
x & y \arrow[l, "\alpha_i"] \end{tikzcd}\\
           0, &  \mbox{if $y>x+1$ and there is a path from right to left}  \\
              & \begin{tikzcd}
x & x+1 \arrow[l, "\alpha_i"] & \cdots \arrow[l, "\alpha_{i+1}"] & y \arrow[l, "\alpha_j"]\end{tikzcd}\\
   B_j\cdots B_{i+1} B_i, &\mbox{if $y<x$ and there is a path from left to right}  \\
            &\begin{tikzcd}
y \arrow[r, "\beta_i"'] & y+1 \arrow[r, "\beta_{i+1}"'] & \cdots \arrow[r, "\beta_j"'] & x\end{tikzcd}\\

\end{array}\right.$$
We denote this matrix by $\zeta_Q(V)$. Let permutation $v_Q$ be the $(S'_Q,S_Q)$-blocked matrix defined as in (\ref{vQ}), its $([x],[y])$-block is
$$\left\{\begin{array}{ll}
           I_{d_x}, &\mbox{if }x=y  \\
           0, & \mbox{else}
         \end{array}\right.$$

We identify the opposite cell $B^-vP_Q/P_Q$ with the space of matrices $\mathbb{A}^{v_Q}$ as in (\ref{AQ}). Then $\zeta_Q(V)$ lies in $\mathbb{A}^{v_Q}$ for any $V\in rep_Q(\mathbf{d})$. We can define a map
$$\begin{array}{cccc}\zeta_Q:&rep_Q(\mathbf{d})&\to &\mathbb{A}^{v_Q}\\
&V=(A_\cdot,B_\cdot)&\mapsto& \zeta_Q(V)\end{array}$$

Conventionally, we also refer to this map as the Zelevinsky map for an arbitrary oriented type A quiver. This map is indeed a morphism of varieties, with a closed image. Furthermore, it naturally induces a surjective homomorphism of $k$-algebras
$$\zeta^*_Q:k[\mathbb{A}^{v_Q}]\to k[rep_Q(\mathbf{d})].$$

\begin{examp}\label{Zex}
Consider   $$\begin{tikzcd}
  V=V_1 \arrow[r, "B_1","\beta_1"'] & V_2 \arrow[r, "B_2","\beta_2"'] & V_3 & V_4 \arrow[l,"\alpha_1","A_1"'] & V_5 \arrow[l,"\alpha_2", "A_2"'] \arrow[r,"\beta_3"', "B_3"] & V_6& V_7 \arrow[l,"\alpha_3","A_3"']\\[-6mm]\ \ \ \ |\arrow[rr, "R_1"', no head] &             & {}|\arrow[rr, "L_1"', no head] & &{}|\arrow[r, "R_2"', no head] &{}|\arrow[r, "L_2"', no head] &{}|
\end{tikzcd}$$
The critical points of this quiver are $s_1=1,s_2=3,s_3=5,s_4=6$ and $s_5=7$. The paths are sequences, for example $$R_1=1\prec 2\prec 3, L_1=5\prec 4\prec 3.$$ The image of this representation under the Zelevinsky map is the matrix
$$\zeta_Q(V)=\begin{blockarray}{lccccccc}
 \begin{block}{l[ccccccc]}
                 1 &  0 & 0 & 0 & I & 0 & 0 & 0 \\
                2  &  0 & 0 & 0 & B_1 & I & 0 & 0 \\
                 5 &  0 & I & 0 & 0 & 0 & 0 & 0 \\
                 7 &  I & 0 & 0 & 0 & 0 & 0 & 0 \\
                 6 &  A_3 & B_3 & 0 & 0 & 0 & 0 & I \\
                 4 &  0 & A_2 & I & 0 & 0 & 0 & 0 \\
                  3 & 0 & 0 & A_1 & B_2B_1 & B_2 & I & 0\\
                  \end{block}
 &7 &  5 &  4 &  1 &  2 &  3 &  6\end{blockarray}$$
where the labels of rows and columns are $S'_Q$ and $S_Q$ respectively. Recall the settings discussed in Section \ref{KL}. In this example, the parabolic $P_Q$ is the group of block upper triangular matrices partitioned by $(d_7,d_5,d_4,d_1,d_2,d_3,d_6)$, where $\mathbf{d}$ is the dimension vector of $V$. Subsequently, the permutation index $v_Q$ of the opposite cell is exactly $$v_Q=\begin{blockarray}{lccccccc}
 \begin{block}{l[ccccccc]}
                 1 &  0 & 0 & 0 & I & 0 & 0 & 0 \\
                2  &  0 & 0 & 0 & 0 & I & 0 & 0 \\
                 5 &  0 & I & 0 & 0 & 0 & 0 & 0 \\
                 7 &  I & 0 & 0 & 0 & 0 & 0 & 0 \\
                 6 &  0 & 0 & 0 & 0 & 0 & 0 & I \\
                 4 &  0 & 0 & I & 0 & 0 & 0 & 0 \\
                  3 & 0 & 0 & 0 & 0 & 0 & I & 0\\
                  \end{block}
 &7 &  5 &  4 &  1 &  2 &  3 &  6\end{blockarray}$$
And the affine space $\mathbb{A}^{v_Q}=B^-v_QP_Q/P_Q$ consists of the matrices such as$$\begin{blockarray}{lccccccc}
 \begin{block}{l[ccccccc]}
                 1 &  0 & 0 & 0 & I & 0 & 0 & 0 \\
                2  &  0 & 0 & 0 & * & I & 0 & 0 \\
                 5 &  0 & I & 0 & 0 & 0 & 0 & 0 \\
                 7 &  I & 0 & 0 & 0 & 0 & 0 & 0 \\
                 6 &  * & * & 0 & * & * & 0 & I \\
                 4 &  * & * & I & 0 & 0 & 0 & 0 \\
                  3 & * & * & * & * & * & I & 0\\
                  \end{block}
 &7 &  5 &  4 &  1 &  2 &  3 &  6\end{blockarray},$$where $*$ indicates the places of affine coordinates.

The rank parameters needed to determine an orbit (closure) are listed in Example \ref{ex1}. In order to obtain an isomorphism between the quiver locus and the Kazhdan-Lusztig variety in this example, one may parallel the rank parameter in Example \ref{ex1} with the rank of the southwest submatrices of $\zeta_Q(V)$, which is determined by the defining ideals of these two varieties. Consider a few matrices that appear in the aforementioned list, for example: $$M_{[1,2]}^Q(V)=B_1,M_{[2,4]}^Q(V)=[A_1,B_2],M_{[1,7]}^Q(V)=\left[\begin{matrix}
          A_3& B_3 & 0\\
          0 & A_1A_2&B_2B_1
        \end{matrix}\right].$$By performing elementary transformations on matrices, we can calculate the rank of the southwest submatrices of $v_Q$, as illustrated below: $$\rank \zeta_Q(V)_{[2]\times[1]}=\rank \left[\begin{matrix}0 & 0 & 0 & B_1\\
                   0 & I & 0 & 0  \\
                  I & 0 & 0 & 0 \\
                 A_3 & B_3 & 0 & 0 \\
                   0 & A_2 & I & 0  \\
                 0 & 0 & A_1 & B_2B_1\end{matrix}\right]=\rank B_1+c,$$
        $$\rank \zeta_Q(V)_{[3]\times[2]}=\rank \left[\begin{matrix}0 & 0 & A_1 & B_2B_1 & B_2\end{matrix}\right]=\rank [A_1,B_2]+c',$$
        $$\rank \zeta_Q(V)_{[6]\times[1]}=\rank \left[\begin{matrix}
                                   A_3 & B_3 & 0 & 0  \\
                   0 & A_2 & I & 0 \\
                   0 & 0 & A_1 & B_2B_1
                                  \end{matrix}\right]=\rank \left[\begin{matrix}
          A_3& B_3 & 0\\
          0 & A_1A_2&B_2B_1
        \end{matrix}\right]+c'',$$ where $c,c',c''$ are constants that depend only on the form of $Q$ and $\mathbf{d}$, but not on the specific $V$. Actually, these constants are precisely the ranks of the corresponding southwest submatrices of $v_Q$. That is, we have
        $$\rank \zeta_Q(V)_{[2]\times[1]}=\rank M_{[1,2]}^Q(V)+\rank (v_Q)_{[2]\times[1]},$$
        $$\rank \zeta_Q(V)_{[3]\times[2]}=\rank M_{[2,4]}^Q(V)+\rank (v_Q)_{[3]\times[2]},$$
        $$\rank \zeta_Q(V)_{[6]\times[1]}=\rank M_{[1,7]}^Q(V)+\rank (v_Q)_{[6]\times[1]}.$$
        One may observe that there should be a correspondence from intervals to $S'_Q\times S_Q$, such as $$[1,2]\mapsto [2]\times[1],[2,4]\mapsto [3]\times[2],[1,7]\mapsto [6]\times[1].$$
\end{examp}
To specifically describe the correspondence from an interval to a southwest submatrix of $\zeta_Q(V)$ mentioned above, we make the following notations. First, we denote the element $z$ satisfying $\ind_{S'_Q}(z)=\ind_{S'_Q}(x)-1$ in $S'_Q$, i.e., the previous element of $x$ in the sequence $S'_Q$, by $N(x)$. Visually in $S'_Q$ the labels of rows, $N(x)$ is the first label north of $x$. Similarly we have $S(x)$ as the south one and $W(x),E(x)$ as the west and east labels of $x$ (in the sequence $S_Q$). The composition of the same such operations also makes sense, like $S^2(x)=S(S(x))$.
$$\begin{blockarray}{lccccc}
 \begin{block}{l[ccccc]}
                  \vdots   &   &  &  &  &   \\
                N(x) &   &  &  &  &    \\
                   x &  &  &  &  &   \\
                 S(x) &  &  & \cdots &  &    \\
                 S^2(x) &   &  &  &  &   \\
                  \vdots &  &  &  &  &   \\
                  \end{block}
 &\cdots &  W(x) &   x& E(x)  & \cdots  \end{blockarray}$$

 Recall that $s_1,s_2,\cdots$ are the critical points of $Q$.
\begin{defn}
  Given a vertex $x\in Q_0$, we define $x^L_Q$ to be
  $$\left\{\begin{array}{ll}
             s_{a-1}, &\mbox{if $x=s_a$ is a critical point}\\
             x, &\mbox{if $x$ is not a critical point}
           \end{array}\right.$$
  and define $x^R_Q$ to be
  $$\left\{\begin{array}{ll}
             s_{b+1}, &\mbox{if $x=s_b$ is a critical point}  \\
             x, &\mbox{if $x$ is not a critical point}
           \end{array}\right.$$ where $s_{a-1}=1$ if $s_a=1$ and $s_{b+1}=n$ if $s_b=n$.

  Moreover, for $x<n$ we define $\lambda_Q(x)$ to be
  $$\left\{\begin{array}{ll}
             x^L_Q, & \mbox{if $x$ is a source in $Q_{[x,n]}$} \\
             W(x^L_Q) , & \mbox{if $x$ is a sink in $Q_{[x,n]}$}
           \end{array}\right.$$
 and for $x>1$ we define $\mu_Q(x)$ to be
    $$\left\{\begin{array}{ll}
             x^R_Q, & \mbox{if $x$ is a sink in $Q_{[1,x]}$} \\
             S(x^R_Q), & \mbox{if $x$ is a source in $Q_{[1,x]}$}
           \end{array}\right.$$
\end{defn}
\begin{examp}
  In Example \ref{Zex}, consider the matrix $M_{[2,4]}^Q(V)=[A_1,B_2]$. We have $$2^L_Q=2,4^R_Q=4.$$And by the source/sink behaviors of $a=2,b=4$ in $Q_{[a,b]}$ there are $$\lambda_Q(a)=2^L_Q=2,\mu_Q(b)=S(4^R_Q)=3.$$ Then the matrix $\zeta_Q(V)_{[\mu_Q(b)]\times[\lambda_Q(a)]}=\zeta_Q(V)_{[3]\times[2]}$ exactly satisfies $$\rank \zeta_Q(V)_{[3]\times[2]}=\rank M_{[2,4]}^Q(V)+\rank (v_Q)_{[3]\times[2]}.$$
  $$\begin{blockarray}{cccccc|cc}
 \begin{block}{c[ccccc|cc]}
                 1 &  0 & 0 & 0 & I & 0 & 0 & 0 \\
                2  &  0 & 0 & 0 & B_1 & I & 0 & 0 \\
                 5 &  0 & I & 0 & 0 & 0 & 0 & 0 \\
                 7 &  I & 0 & 0 & 0 & 0 & 0 & 0 \\
                 6 &  A_3 & B_3 & 0 & 0 & 0 & 0 & I \\
                \boxed{4^R_Q=4} &  0 & A_2 & I & 0 & 0 & 0 & 0 \\
                \cdashline{1-8}[1pt/0pt]  3 & 0 & 0 & A_1 & B_2B_1 & B_2 & I & 0\\
                  \end{block}
 &7 &  5 &  4 &  1 &  \boxed{2^L_Q=2} &  3 &  6\end{blockarray}$$We can determine the values $\lambda_Q(a),\mu_Q(b)$ in a visual way: first find the vertices $a^L_Q\in S_Q, b^R_Q\in S'_Q$ in the labels of columns and rows respectively, then draw a line on the right(resp. left) of column $[a^L_Q]$ if $a$ is a source(resp. sink) in $Q_{[a,b]}$, and draw a line over(resp. under) the row $[b^L_Q]$ if $b$ is a sink(resp. source) in $Q_{[a,b]}$. Finally the southwest matrix bounded by these two lines is exactly $\zeta_Q(V)_{[\mu_Q(b)]\times[\lambda_Q(a)]}$. The following diagram indicates the relationship between the source/sink behavior and where we should draw the line. $$\begin{tikzcd}
 a \arrow[r] & \cdots           & \ a^L_Q| & , & \cdots \arrow[r] & b           & \overline{b^R_Q} \end{tikzcd}$$
  $$\begin{tikzcd}
 a & \cdots  \arrow[l]         & \ |a^L_Q & , & \cdots & b   \arrow[l]        & \underline{b^R_Q} \end{tikzcd}$$

\end{examp}
In the examples above we can also extract the relations between a rank parameter matrix and a corresponding southwest part in $\zeta_Q(V)$. That is, they can be linked through a series of special transformations of matrices. For instance, in Example \ref{Zex},
$$\rank \zeta_Q(V)_{[6]\times[1]}=\left[\begin{matrix}
                                   A_3 & B_3 & 0 & 0  \\
                   0 & A_2 & I & 0 \\
                   0 & 0 & A_1 & B_2B_1
                                  \end{matrix}\right]\Rightarrow \left[\begin{matrix}
                                   A_3 & B_3 & 0 & 0  \\
                   0 & 0 & I & 0 \\
                   0 & -A_1A_2 & A_1 & B_2B_1
                                  \end{matrix}\right]$$ $$\Rightarrow\left[\begin{matrix}
                                   A_3 & B_3 & 0 & 0  \\
                   0 & 0 & I & 0 \\
                   0 & -A_1A_2 & 0 & B_2B_1
                                  \end{matrix}\right]\Rightarrow\left[\begin{matrix}
                                   A_3 & B_3 & 0 & 0  \\
                   0 & A_1A_2 & B_2B_1 & 0 \\
                   0 & 0 & 0 & I
                                  \end{matrix}\right]=\left[\begin{matrix}
                                                              M_{[1,7]}^Q(V) & 0 \\
                                                              0 & I
                                                            \end{matrix}\right]$$
Rigorously, such relations are described by the following definition.
\begin{defn}
Given any $(S',S)$-blocked matrix $A$ with sequence of vertices $S',S$, we define two operations. Each operation is a composition of a series of special elementary block transformations of matrix.
\begin{itemize}
  \item[\textcircled{1}]in this operation, the following transformations are allowed: 
      \begin{itemize}
        \item[(i)] replacing a block row $[x]$ of $A$ by $$\mbox{ row }[x]\ -\mbox{ the $([x],[y])$-block }\times\mbox{ row }[y]\ $$with $y\prec x$ in $S'$.
        \item[(ii)] replacing a block column $[x]$ of $A$ by $$\mbox{ column } [x]\ -\mbox{ column }[y]\ \times \mbox{ the $([y],[x])$-block }$$with $x\prec y$ in $S$.
      \end{itemize}
  \item[\textcircled{2}]in this operation, the following transformations are allowed: 
  \begin{itemize}
    \item[(i)] Scaling a block row/column by $-1$.
    \item[(ii)] Interchanging the block rows/columns.
  \end{itemize}
\end{itemize}\end{defn}

Here we provide the allowed transformations in very detail in order to emphasize that this process will not involve the inverse of matrix. Therefore, we can replace the matrix in the field $k$ by matrix in some $k$-algebra later (Corollary \ref{k-alg}). Given two matrix $A,B$, we say $A\stackrel{\mbox{\rm{\tiny\textcircled{1}}}}{\Rightarrow}B$ if $A$ can be transformed into $B$ through an appropriate operation of type \rm{\textcircled{1}}. And similarly we have $A\stackrel{\mbox{\rm{\tiny\textcircled{2}}}}{\Rightarrow}B$. We say $A\stackrel{\mbox{\rm{\tiny\textcircled{1}+\textcircled{2}}}}{\Longrightarrow}B$ if $A$ can be transformed into $B$ in two steps: in the first step we use an operation \rm{\textcircled{1}}, and in the second step we use an operation \rm{\textcircled{2}}.

From now on, the proofs of our main theorems depend on the induction method. Given a quiver $Q$ with the number of maximal paths $l+r>1$, we denote the rightmost critical point except for $n$ by $n'$ and define the subquiver $Q':=Q_{[1,n']}$. Similarly, we can define $S'_{Q'},S_{Q'},M^{Q'}_{[a,b]},\zeta_{Q'},\lambda_{Q'}(x),\mu_{Q'}(x)$, and so on, using the data of $Q'$ (but note that we should replace $n$ with $n'$ in definitions about $Q'$). Given a representation $V\in rep_Q(\mathbf{d})$ we define a representation $V'\in rep_{Q'}(\mathbf{d'})$, where $d'=(d_x)_{x\in Q_0'}$, by $V'=(V_\alpha)_{\alpha\in Q_1'}$.

\begin{theo}\label{main}
  For any interval $[a,b]\subset Q$, there are labels $\mu_Q(b),\lambda_Q(a)$ such that
  $$\zeta_Q(V)_{[\mu_Q(b)]\times[\lambda_Q(a)]}\stackrel{\mbox{\rm{\tiny\textcircled{1}+\textcircled{2}}}}{\Longrightarrow}\left[\begin{matrix}
                                                                                                                            M^Q_{[a,b]}(V) &  &  \\
                                                                                                                             & I  &  \\
                                                                                                                             &  & 0
                                                                                                                          \end{matrix}\right],$$
where the identity matrix is of size $\rank (v_Q)_{[\mu_Q(b)]\times[\lambda_Q(a)]}$. Moreover, $$\rank (\zeta_Q(V)_{[\mu_Q(b)]\times[\lambda_Q(a)]})=\rank M_{[a,b]}(V)+\rank (v_Q)_{[\mu_Q(b)]\times[\lambda_Q(a)]}.$$
\end{theo}
\begin{proof}
  We use induction on the number of maximal paths $l+r$. If $l+r=1$, then the case is trivial (see also \cite{lakdegesche1998} and \cite{kinser_type_2015} for more details on this case).

  Now consider any quiver $Q$ of type A with $l+r>1$. Then we can define the subquiver $Q'$ as described above. Note that $n'$ is the rightmost vertex in $Q'$ and is a critical point in $Q$.

  There are two cases:

  \textbf{Case 1}: $n'$ is a source in $Q$.
$$
\begin{tikzcd}
\cdots & V_{n'} \arrow[l] \arrow[r, "B_s"] & V_{n'+1} \arrow[r, "B_{s+1}"] &\cdots \arrow[r, "B_t"] & V_n
\end{tikzcd}$$
The matrix $\zeta_Q(V)$ has the form
\begin{adjustwidth}{}{-2cm}
\begin{equation}\label{*}\hspace{-30mm}
\resizebox{1.2\textwidth}{!}{
$
\begin{blockarray}{llccccc|ccccc|}
 \begin{block}{ll[ccccc|ccccc|]}
 R_1\backslash t(R_1)&   &  *  &   &   &  &  * & 0  & 0  & \cdots&  &  0 \\
             R_2\backslash t(R_2)&       &    &   &   &  &   &  \vdots &   & & &\vdots\\
                           \vdots&\vdots & *   &   &   &  &  * &  0 & 0  &\cdots &  & 0\\
       \cdashline{1-12}[1pt/1pt]\multirow{5}*{$R_r\backslash t(R_r)\left.\rule{0mm}{11mm}\right\{$}
             &  n'      & I                          & 0      &  &   \cdots&  0      & 0                      & 0                     & &  0&  0\\
       \cdashline{2-12}[1pt/0pt]
             & n'+1     &  B_s                       &  0     &  &  \cdots &  0      & I                      & 0                     & \cdots  &  0  &0  \\
             & n'+2     &  B_{s+1}B_s                &  0     &  &         &  0      &  B_{s+1}               & I                     &   &  0  &0  \\
             & \vdots   &  \vdots                    &  \vdots&  &         &  \vdots &  \vdots                & \vdots                &  \ddots &    &   \\
             &n-1       &  B_{t-1}\cdots B_s         & 0      &  &  \cdots & 0       &  B_{t-1}\cdots B_{s+1} & B_{t-1}\cdots B_{s+2} &   & I   &0  \\
            \cdashline{1-12}[1pt/1pt]\{n\}
             &n         &B_tB_{t-1}\cdots B_s        & 0      &  &  \cdots &   0     &B_tB_{t-1}\cdots B_{s+1}&  B_tB_{t-1}\cdots B_{s+2}&\cdots&  B_t  & I \\
               \cdashline{2-12}[1pt/0pt]   L_l\backslash s(L_l)& \vdots & *  &   &  &  & * & 0 &  0 &  \cdots &  &0\\
            \vdots &  &   &   &  &  &  & \vdots&   &   &  &\vdots\\
                 L_1\backslash s(L_1) && *  &   &  &  & * & 0 &  0 &   &  &0\\
                  \end{block}
                         (S'_Q)  &        &n' & n'-1 &\cdots &  &   \cdots& n'+1 &n'+2 &   \cdots&  n-1&n &\\
 & &\BAmulticolumn{3}{c}{\ \ \ \underbrace{\rule{40mm}{0mm}}} &                              &        &    \BAmulticolumn{5}{c|}{\ \ \ \underbrace{\rule{67mm}{0mm}}} \\
 &(S_Q) &\BAmulticolumn{3}{c}{\ \ L_l\backslash t(L_l)}          & L_{l-1}\backslash t(L_{l-1}) & \cdots &    \BAmulticolumn{5}{c|}{\ \ R_r\backslash s(R_r)}        \\\end{blockarray}
$
}\hspace{-5mm}
\end{equation}\end{adjustwidth}
We indicate the label of each block row/column and its corresponding subsequence, next to the matrix. It is important to note that removing the block rows and columns $[n'+1],[n'+2],\cdots,[n]$ from $\zeta_Q(V)$ we will produce the matrix $\zeta_{Q'}(V')$, where the labels remain the same.

\textbf{Case 1.1}: For the interval $[a,b]$ with $b<n'$, we have $$a_{Q}^L=a_{Q'}^L\leq n',b_{Q}^R=b_{Q'}^R\leq n'.$$ By the definition of $\lambda_Q$, there must be $$\lambda_Q(a)=\lambda_{Q'}(a)\leq n'.$$ Note that $b_Q^R=n'$ implies that $b$ is a sink in $Q$ because $n'$ is a source and $b<n'$. Now, by the definition of $\mu_Q$ we have $$\mu_Q(b)=\mu_{Q'}(b)\leq n'.$$

Therefore, if $\mu_Q(b)\succ n$ in $S'_Q$, $$\zeta_Q(V)_{[\mu_Q(b)]\times[\lambda_Q(a)]}=\zeta_{Q'}(V')_{[\mu_{Q'}(b)]\times[\lambda_{Q'}(a)]}$$ if $\mu_Q(b)\prec n'+1$ in $S'_Q$,
                                                                    $$\zeta_Q(V)_{[\mu_Q(b)]\times[\lambda_Q(a)]}=\begin{blockarray}{lcccc}
                                                                    \begin{block}{l[cccc]}
                                                                     \mu_Q(b)& * &  &  & * \\
                                                                    \vdots&  \vdots &  &  & \vdots \\
                                                                      &I & 0 & \cdots & 0 \\
                                                                      \cdashline{2-5}[1pt/0pt]&B_s & 0 & \cdots & 0 \\
                                                                      &B_{s+1}B_s & 0 & \cdots & 0 \\
                                                                     & \vdots & \vdots &  & \vdots \\
                                                                      &B_t\cdots B_s & 0 & \cdots & 0 \\
                                                                      \cdashline{2-5}[1pt/0pt]&* &  &  & * \\
                                                                      &\vdots &  &  & \vdots\\
                                                                    \end{block}
                                                                    & & &\cdots & \lambda_Q(a)\end{blockarray}                                                                $$
Note that by deleting the block rows between the two solid lines in the matrix above, we obtain exactly $\zeta_{Q'}(V')_{[\mu_{Q'}(b)]\times[\lambda_{Q'}(a)]}$. Furthermore, the block rows between the two solid lines can be eliminated to zero by the row $[n']$.

Also note that $$M^Q_{[a,b]}(V)=M^{Q'}_{[a,b]}(V'),$$ $$\rank (v_{Q})_{[\mu_{Q}(b)]\times[\lambda_{Q}(a)]}=\rank (v_{Q'})_{[\mu_{Q'}(b)]\times[\lambda_{Q'}(a)]}.$$
By induction we can assume there is  $$\zeta_{Q'}(V')_{[\mu_{Q'}(b)]\times[\lambda_{Q'}(a)]}\stackrel{\mbox{\rm{\tiny\textcircled{1}+\textcircled{2}}}}{\Longrightarrow}\left[\begin{matrix}
                                                                                                                            M^{Q'}_{[a,b]}(V') &  &  \\
                                                                                                                             & I_{\rank (v_{Q'})_{[\mu_{Q'}(b)]\times[\lambda_{Q'}(a)]}}  &  \\
                                                                                                                             &  & 0
                                                                                                                          \end{matrix}\right]$$
By analyzing the forms of $\zeta_Q(V),M^{Q}_{[a,b]}$ and $v_Q$, we can immediately observe   $$\zeta_{Q}(V)_{[\mu_Q(b)]\times[\lambda_Q(a)]}\stackrel{\mbox{\rm{\tiny\textcircled{1}+\textcircled{2}}}}{\Longrightarrow}\left[\begin{matrix}
                                                                                                                            M^{Q}_{[a,b]}(V) &  &  \\
                                                                                                                             & I_{\rank (v_{Q})_{[\mu_{Q}(b)]\times[\lambda_{Q}(a)]}}  &  \\
                                                                                                                             &  & 0
                                                                                                                          \end{matrix}\right]$$with the operation $\rm{\textcircled{1}+\textcircled{2}}$ that depends only on $Q,a,b$.

We should make a remark that in this theorem the condition on the size of identity matrix naturally naturally holds because this rank is exactly obtained by the number of constant $1$ in the corresponding submatrix. Thus, in the rest of proof we do not check the size of the identity matrix anymore.

\textbf{Case 1.2}: For the interval $[a,b]$ with $b=n'$, we have that $b$ is a source in $Q$, $$a^L_Q=a^L_Q\leq a< n',$$ and
$$b_{Q'}^R=n', b_Q^R=n.$$

Then $$\lambda_Q(a)=\lambda_{Q'}(a)\leq n', \mu_Q(b)=\mu_{Q'}(b)=n'-1.$$
And we can observe in matrix (\ref{*}) that $$\zeta_Q(V)_{[n'-1,\lambda_Q(a)]}=\zeta_{Q'}(V')_{[n'-1,\lambda_{Q'}(a)]}.$$ Note that $M^{Q}_{[a,b]}(V')=M^{Q'}_{[a,b]}(V').$ Based on the induction assumption, we can immediately obtain the desired result.

\textbf{Case 1.3}: For interval $[a,b]$ with $a\leq n'$ and $b\geq n'$, we have $$a_{Q}^L=a_{Q'}^L\leq n',b_Q^R=b.$$
$$\begin{tikzcd}
\cdots & V_{n'} \arrow[l] \arrow[r, "B_s"] & V_{n'+1} \arrow[r, "B_{s+1}"] &\cdots \arrow[r, "B_{p}"]&V_b \arrow[r, "B_{p+1}"]&\cdots \arrow[r, "B_t"] &V_n
\end{tikzcd}$$
There is $$\zeta_Q(V)_{[\mu_Q(b)]\times[\lambda_Q(a)]}=\begin{blockarray}{lcccc}
 \begin{block}{l[cccc]}
              b &  B_{p}\cdots B_s &0 &\cdots &  0 \\
              b+1 &  B_{p+1}B_{p}\cdots B_s &0 & \cdots &0 \\
              \vdots &  \vdots & \vdots & & \vdots \\
              n &  B_{t}B_{t-1}\cdots B_s& 0 & \cdots& 0 \\
              \vdots & \BAmulticolumn{4}{@{\ [}c@{]}}{\zeta_{Q'}(V')_{[n'-1]\times[\lambda_{Q'}(a)]}}  \\
 \end{block}
 &n' &  \cdots &&\lambda_{Q}(a)\end{blockarray}.$$

Here we use the smaller brackets to mark out the scope of submatrix, and if $a=n'$ we take $\lambda_{Q'}(a):=t(L_l)$ to be the maximal vertex in sequence $S_{Q'}$. Obviously the rows $[b+1],\cdots,[n]$ can be eliminated by the row $[b]$ in this matrix.

And $$M^Q_{[a,b]}(V)=\begin{blockarray}{lcccc}
 \begin{block}{l[cccc]}
              b &  B_{p}\cdots B_s  &0 &\cdots &  0 \\
              \vdots & \BAmulticolumn{4}{@{\ [}c@{]}}{M_{[a,n']}^{Q'}(V')}  \\
 \end{block}
 &n' &  \cdots \end{blockarray}.$$

 Note that $\mu_{Q'}(n')=n'-1$. By induction, we can assume (if $a=n'$ then $M^{Q'}_{[a,n']}$ is an empty matrix): $$\zeta_{Q'}(V')_{[n'-1]\times[\lambda_{Q'}(a)]}\stackrel{\mbox{\rm{\tiny\textcircled{1}+\textcircled{2}}}}{\Longrightarrow}\left[\begin{matrix}
                                                                                                                            M^{Q'}_{[a,n']}(V') &  &  \\
                                                                                                                             & I &  \\
                                                                                                                             &  & 0
                                                                                                                          \end{matrix}\right]$$
 By the definition of operation $\rm{\textcircled{1}+\textcircled{2}}$, the first block column of $\zeta_{Q'}(V')_{[\mu_{Q'}(b)]\times[\lambda_{Q'}(a)]}$ remains invariant in the transformation process above. Additionally, we have $M^{Q}_{[a,n']}(V)=M^{Q'}_{[a,n']}(V')$. Thus, by eliminating the rows $[b+1],\cdots,[n]$ in $\zeta_Q(V)_{[\mu_Q(b)]\times[\lambda_Q(a)]}$, we can obtain
 $$ \zeta_{Q}(V)_{[\mu_Q(b)]\times[\lambda_Q(a)]}\stackrel{\mbox{\rm{\tiny\textcircled{1}+\textcircled{2}}}}{\Longrightarrow}
 \left[\begin{matrix}
 B_{p}\cdots B_s & 0 &\cdots & 0 & & \\
  \multicolumn{4}{@{\ [}c@{]}}{M^{Q}_{[a,n']}(V)} &  &  \\
  & & & &I  &  \\
  & & & &   & 0
    \end{matrix}\right]=\left[\begin{matrix}
                                                                                                                            M^{Q}_{[a,b]}(V) &  &  \\
                                                                                                                             & I &  \\
                                                                                                                             &  & 0
                                                                                                                          \end{matrix}\right]$$

\textbf{Case 1.4}: For the interval $[a,b]$ with $a>n'$, we have $$a^L_Q=a,b^R_Q=b.$$
$$\begin{tikzcd}
\cdots & V_{n'} \arrow[l] \arrow[r, "B_s"] & \cdots \arrow[r] &V_a \arrow[r, "B_{q}"]&\cdots \arrow[r, "B_{p}"]&V_b \arrow[r, "B_{p+1}"]&\cdots \arrow[r, "B_t"] & V_n
\end{tikzcd}$$
Moreover, in $Q_{[a,b]}$, $a$ is a source and $b$ is a sink. Then
$$\lambda_Q(a)=a, \mu_Q(b)=b.$$

We have $\zeta_Q(V)_{[b]\times[a]}=$ $$\begin{blockarray}{lccccccc}
 \begin{block}{l[ccccccc]}
              b &  B_{p}\cdots B_s &0 &\cdots &  0 & B_{p}\cdots B_{s+1} & \cdots &B_{p}\cdots B_{q} \\
              b+1 &  B_{p+1}B_{p}\cdots B_s &0 & \cdots &0& B_{p+1}B_{p}\cdots B_{s+1}& & B_{p+1}B_p\cdots B_{q} \\
              \vdots &  \vdots & \vdots & & \vdots& \vdots& & \vdots\\
              n &  B_{t}B_{t-1}\cdots B_s& 0 & \cdots& 0& B_{t}B_{t-1}\cdots B_{s+1} & \cdots& B_{t}B_{t-1}\cdots B_{q}  \\
              \vdots & \BAmulticolumn{4}{@{\ [}c@{]}}{\zeta_{Q'}(V')_{[n'-1]\times[n'-1]}}&0 &\cdots & 0 \\
 \end{block}
 &n' &  \cdots & &\cdots & n'+1& \cdots& a\end{blockarray}.$$
Here $s\leq q<p\leq t$. Therefore, with an operation \rm{\textcircled{1}}, we can eliminate the blocks to zero in the entire rows $[b],[b+1],\cdots,[n]$ and in the entire columns $[n'+1],\cdots,[a]$, except for the $([b],[a])$-block. Note that $\zeta_{Q'}(V')_{[n'-1]\times[n'-1]}$ has $N$ columns, and
$$\zeta_{Q'}(V')_{[n'-1]\times[n'-1]}\stackrel{\mbox{\rm{\tiny\textcircled{1}+\textcircled{2}}}}{\Longrightarrow}\left[\begin{matrix}
                                                                                                                         I_{\rank (v_{Q'})_{[n'-1]\times [n'-1]}} & 0
                                                                                                                       \end{matrix}\right].$$
Thus, $$\zeta_Q(V)_{[b]\times[a]}\stackrel{\mbox{\rm{\tiny\textcircled{1}+\textcircled{2}}}}{\Longrightarrow}\left[\begin{matrix}
                                                                                                                    B_{p}\cdots B_{q}  &  &  \\
                                                                                                                     & I &  \\
                                                                                                                     &  & 0
                                                                                                                  \end{matrix}\right].$$

\textbf{Case 2}: $n'$ is a sink in $Q$.

$$\begin{tikzcd}
\cdots \arrow[r] & V_{n'} & V_{n'+1} \arrow[l, "A_s"'] & \cdots \arrow[l, "A_{s+1}"'] & V_n \arrow[l, "A_t"']
\end{tikzcd}$$
Then $\zeta_Q(V)$ has the form
\begin{adjustwidth}{}{-1cm}
\begin{equation}\label{**}\hspace{-30mm}
\resizebox{1\textwidth}{!}{
$\begin{blockarray}{llccccc|ccc}
 \begin{block}{ll[|ccccc|ccc]}
R_1\backslash t(R_1)&   &  0  &   &   &  &  0                &*& &*                                     \\
\vdots&   &  \vdots  &   &   &  & \vdots        & & &                                     \\
R_r\backslash t(R_r)&\vdots & 0   &   &   &  &  0            & * & &*                                      \\
            \cdashline{2-10}[1pt/0pt] \{n\} &  n      & I          & 0      &   \cdots    & 0  &  0  &0 &  \cdots&0                          \\
           \cdashline{1-10}[1pt/1pt]\multirow{5}*{$L_l\backslash s(L_l)\left.\rule{0mm}{13mm}\right\{$}
            & n-1     &  A_t       &  I     &       & 0 &  0           &0 &  \cdots&0                                   \\
             & n-2     &  0         &A_{t-1} &    \ddots   &  0       &  0   &0 &  \cdots&0                      \\
             & \vdots  &  \vdots    &  &    \ddots   &         &  \vdots   &\vdots &  &\vdots                                   \\
             &n'+1     &  0         & 0      &       & A_{s+1} &  I      &0 &  \cdots&0                             \\
             \cdashline{2-10}[1pt/0pt] &n'       &  0        & 0       &    \cdots   & 0  &   A_s    &* &  &*                             \\
            \cdashline{1-10}[1pt/1pt] L_l\backslash s(L_l)& \vdots & 0  &   &  &  & 0          & &  &                                          \\
            \vdots &  & \vdots  &   &  &  & \vdots                 & &  &                                        \\
                 L_1\backslash s(L_1) & & 0  &   &  &  & 0         &* &  &*                                    \\
                  \end{block}
            (S'_Q)   &&n &n-1 &\cdots &n'+2&n'+1  & \cdots\cdots &\cdots &n'                                                            \\
               &&\BAmulticolumn{5}{c}{\underbrace{\rule{50mm}{0mm}}}&\cdots\cdots &\BAmulticolumn{2}{c}{\underbrace{\rule{20mm}{0mm}}}                                             \\
               &(S_Q)&\BAmulticolumn{5}{c}{L_l\backslash t(L_l)}  &\cdots\cdots &\BAmulticolumn{2}{c}{R_r\backslash s(R_r)}                                                    \\
\end{blockarray}
$
}\hspace{-15mm}
\end{equation}\end{adjustwidth}

Similarly to case 1, we can obtain $\zeta_{Q'}(V')$ by deleting the block rows and columns $[n'+1],\cdots,[n]$.

\textbf{Case 2.1}: For $[a,b]$ with $b<n'$,we have $$a^L_{Q}=a^L_{Q'}\leq n',b^R_{Q}=b^R_{Q'}\leq n'.$$ Since the minimal element in $S^{Q'}_2$ must be $s(L_{l-1})$, which is a source in $Q$, according to the definition of $\lambda_Q$ we have $$\lambda_Q(a)=\lambda_{Q'}(a)\leq n'.$$

Note that $b^R_{Q'}=n'$ implies that $b$ is a source in $Q$. Thus, $$\mu_Q(b)=\mu_{Q'}(b)\leq n'-1$$ with only one exception: $$\mu_Q(b)=n\neq \mu_{Q'}(b)=n',\mbox{ if $b=n'-1$ and $b$ is a source in $Q$.}$$

Therefore, we have $\zeta_Q(V)_{[\mu_Q(b)]\times[\lambda_Q(a)]}=[0,\zeta_{Q'}(V')_{[\mu_Q(b)]\times[\lambda_Q(a)]}]$ if $\mu_Q(b)\succ n'$  in $S'_Q$. And if $\mu_Q(b)\prec n-1$ in $S'_Q$, there is $\zeta_Q(V)_{[\mu_Q(b)]\times[\lambda_Q(a)]}=$ $$\begin{blockarray}{l|ccccc|ccc}
 \begin{block}{l[|ccccc|ccc]}
  \mu_Q(b) &  0  &   &   &  &  0                &*& &*                                     \\
   &  \vdots  &   &   &  & \vdots        & & &                                     \\
\vdots & 0   &   &   &  &  0            & * & &*                                      \\
            \cdashline{1-9}[1pt/0pt]  n      & I          & 0      &   \cdots    & 0  &  0  &0 &  \cdots&0                          \\
             n-1     &  A_t       &  I     &       & 0 &  0           &0 &  \cdots&0                                   \\
              n-2     &  0         &A_{t-1} &    \ddots   &  0       &  0   &0 &  \cdots&0                      \\
              \vdots  &  \vdots    &  &    \ddots   &         &  \vdots   &\vdots &  &\vdots                                   \\
             n'+1     &  0         & 0      &       & A_{s+1} &  I      &0 &  \cdots&0                             \\
          \cdashline{1-9}[1pt/0pt]    n'       &  0        & 0       &    \cdots   & 0  &   A_s    &* &  &*                             \\
            \vdots & 0  &   &  &  & 0          & &  &                                          \\
              & \vdots  &   &  &  & \vdots                 & &  &                                        \\
                 & 0  &   &  &  & 0         &* &  &*                                    \\
                  \end{block}
            &n &n-1 &\cdots &n'+2&n'+1  & \cdots\cdots &\cdots &\lambda_Q(a) \\
\end{blockarray}$$
(it has no block row to the north of row $[n]$ if $\mu_Q(b)=n$). By deleting the rows and columns between the solid lines in $\zeta_Q(V)$ we obtain $\zeta_{Q'}(V')$. An obvious operation \rm{\textcircled{1}} can be used to eliminate all the $A_i$ in the matrix above. Since $M^Q_{[a,b]}(V)=M^{Q'}_{[a,b]}(V')$, combining the induction assumption on $\zeta_{Q'}(V')$ we immediately have the desired result.

\textbf{Case 2.2}: For $[a,b]$ with $b=n'$, we have that $b$ is a sink in $Q$ and $$a^L_Q=a^L_{Q'}\leq n',b^R_{Q'}=n',b^R_Q=n.$$

Therefore, $$\lambda_Q(a)=\lambda_{Q'}(a),\mu_Q(b)=n.$$

Then $\zeta_Q(V)_{[\mu_Q(b)]\times[\lambda_Q(a)]}=$ $$\begin{blockarray}{l|ccccc|ccc}
\begin{block}{l[|ccccc|ccc]}
        \cdashline{1-9}[1pt/0pt]  n  & I          & 0      &   \cdots    & 0  &  0  &0 &  \cdots&0                          \\
           n-1&    A_t       &  I     &       & 0 &  0           &0 &  \cdots&0                                   \\
           n-2 &   0         &A_{t-1} &    \ddots   &  0       &  0   &0 &  \cdots&0                      \\
           \vdots &   \vdots    &  &    \ddots   &         &  \vdots   &\vdots &  &\vdots                                   \\
           n'+1 &  0         & 0      &       & A_{s+1} &  I      &0 &  \cdots&0                             \\
           \cdashline{1-9}[1pt/0pt]n' &     0        & 0       &    \cdots   & 0  &   A_s    & & &                             \\
          \vdots  & 0  &   &  &  & 0          & & \zeta_{Q'}(V')_{[n']\times[\lambda_{Q'}(a)]}&                                          \\
             &  \vdots  &   &  &  & \vdots                 & &  &                                        \\
              &    0  &   &  &  & 0         & &  &     \\
                  \end{block}
                  &n &n-1 &\cdots &n'+2&n'+1&&\cdots&\\
\end{blockarray}$$

Since $\mu_{Q'}(n')=n'$, by the induction we can assume $$\zeta_{Q'}(V')_{[n']\times[\lambda_{Q'}(a)]}\stackrel{\mbox{\rm{\tiny\textcircled{1}+\textcircled{2}}}}{\Longrightarrow}\left[\begin{matrix}
                                                                                                                                  M^{Q'}_{[a,n']}(V') &  &  \\
                                                                                                                                   & I & \\
                                                                                                                                   &  & 0
                                                                                                                                \end{matrix}\right].$$

Note that $M^{Q}_{[a,n']}(V)=M^{Q'}_{[a,n']}(V')$ and $b=n'$ at this time. Since the blocks $A_s,A_{s+1},\cdots,A_t$ in $\zeta_Q(V)_{[\mu_Q(b)]\times[\lambda_Q(a)]}$ obviously can be all eliminated, we obtain the result we desire.

\textbf{Case 2.3}: For the interval $[a,b]$ with $a\leq n', b>n'$.
$$\begin{tikzcd}
\cdots \arrow[r] & V_{n'} &  \cdots \arrow[l, "A_{s}"'] & V_b \arrow[l, "A_p"'] & \cdots \arrow[l, "A_{p+1}"'] & V_n \arrow[l, "A_t"']
\end{tikzcd}$$
We have $a^L_Q=a^L_{Q'}\leq n', b^R_Q=b$ and $b$ is a source in $Q_{[1,b]}$.
Hence, $$\lambda_Q(a)=\lambda_{Q'}(a),\mu_Q(b)=b-1.$$

We have  $\zeta_Q(V)_{[\mu_Q(b)]\times[\lambda_Q(a)]}=$ $$\resizebox{1\textwidth}{!}{
$\begin{blockarray}{l|ccccccc|ccc}
 \begin{block}{l[|ccccccc|ccc]}
             b-1     &  0 &\cdots &A_p       &  I     &       & 0 &  0           &0 &  \cdots&0                                   \\
              b-2     & 0 &\cdots& 0         &A_{p-1} &    \ddots   &  0       &  0   &0 &  \cdots&0                      \\
              \vdots  &  \vdots& & \vdots &    &   \ddots      &  \vdots  & &\vdots &  &\vdots                                   \\
             n'+1     & 0 &\cdots &0         & 0      &       & A_{s+1} &  I      &0 &  \cdots&0                             \\
        \cdashline{1-11}[1pt/0pt] n'       &  0 &\cdots&0        & 0       &    \cdots   & 0  &   A_s    & &  &                             \\
            \vdots & 0  &  & &  & & & 0         & & \zeta_{Q'}(V')_{[n']\times[\lambda_{Q'}(a)]} &                                          \\
              & \vdots  & &  &  & & & \vdots               &  & &                                        \\
                 & 0  &  & &  &\cdots & & 0        &  &    &                                \\
                  \end{block}
            &n &\cdots&b&b-1&\cdots &n'+2&n'+1  & \cdots\cdots &\cdots &\lambda_Q(a) \\
\end{blockarray}$}$$

Since $\mu_{Q'}(n')=n'$, by the induction we can assume $$\zeta_{Q'}(V')_{[n']\times[\lambda_{Q'}(a)]}\stackrel{\mbox{\rm{\tiny\textcircled{1}+\textcircled{2}}}}{\Longrightarrow}\left[\begin{matrix}
                                                                                                                                  M^{Q'}_{[a,n']}(V') &  &  \\
                                                                                                                                   & I &  \\
                                                                                                                                   &  & 0
                                                                                                                                \end{matrix}\right]$$
Note that $$M^Q_{[a,b]}(V)=\begin{blockarray}{lcccc}\begin{block}{l[cccc]}
                                  n'& A_s\cdots A_p & \BAmulticolumn{3}{c}{\overbracket[0.5pt][2pt]{\rule{12mm}{0mm}}}\\
                                    & 0 &   \BAmulticolumn{3}{c}{\multirow{2}{*}{$M^{Q'}_{[a,n']}(V')$}}\\
                             \vdots &\cdots &   \BAmulticolumn{3}{c}{} \\
                                  & 0 &  \BAmulticolumn{3}{c}{\raisebox{2mm}{$\underbracket[0.5pt][2pt]{\rule{12mm}{0mm}}$}} \\\end{block}
                                  & b &
                                 \end{blockarray}$$
where $(\pm) A_s\cdots A_p$ can be obtained in the $([n'],[b])$-block of $\zeta_Q(V)_{[\mu_Q(b)]\times[\lambda_Q(a)]}$ after performing an operation \rm{\textcircled{1}}. Therefore, we have $$\zeta_Q(V)_{[\mu_Q(b)]\times[\lambda_Q(a)]}\stackrel{\mbox{\rm{\tiny\textcircled{1}+\textcircled{2}}}}{\Longrightarrow}
\left[\begin{matrix}
 A_s\cdots A_p & \multicolumn{3}{c}{\overbracket[0.5pt][2pt]{\rule{12mm}{0mm}}}  &  &  \\
             0 & \multicolumn{3}{c}{\multirow{2}{*}{$M^{Q'}_{[a,n']}(V')$}} &  &  \\
        \vdots & \multicolumn{3}{c}{} &  &  \\
         0& \multicolumn{3}{c}{\raisebox{2mm}{$\underbracket[0.5pt][2pt]{\rule{12mm}{0mm}}$}} &  &  \\
         &  &  &  & I &  \\
         &  &  &  &  & 0
      \end{matrix}\right]=\left[\begin{matrix}
                                  M^Q_{[a,b]}(V) &  &  \\
                                   & I &  \\
                                   &  & 0
                                \end{matrix}\right]$$

\textbf{Case 2.4}: For $[a,b]$ with $a>n'$, we have $$a^L_Q=a,b^R_Q=b.$$ Moreover, $a$ is a sink in $Q_{[a,n]}$ and $b$ is a source in $Q_{[1,b]}$.
$$\begin{tikzcd}
\cdots \arrow[r] & V_{n'} &  \cdots \arrow[l, "A_{s}"'] & V_a \arrow[l] & \cdots \arrow[l, "A_{q}"'] & V_b \arrow[l, "A_p"']&\cdots \arrow[l]&V_n \arrow[l, "A_t"']
\end{tikzcd}$$
Hence, $$\lambda_Q(a)=a+1,\mu_Q(b)=b-1.$$
There is $\zeta_Q(V)_{[b-1]\times[a+1]}=$ $$\begin{blockarray}{lcccccccc}\begin{block}{l[cccccccc]}
b-1 &0&\cdots &A_p &I & \cdots & 0& 0\\
b-2 &0&\cdots &0 &A_{p-1} & \ddots &0 &0 \\
\vdots&\vdots&&& &\ddots&\ddots &\\
a+1&0&\cdots& 0&0 & \cdots&A_{s+1} &I\\
a& 0&\cdots&0 &0  & \cdots& 0 &A_s\\
\vdots &0&\cdots&0&0 &\cdots &0&0\\
     &\vdots& &\vdots &    &\vdots &\vdots\\
     & 0&\cdots&0  &0   &   \cdots &0      &0\\
\end{block}
&n&\cdots&b&b-1&\cdots&a+2&a+1\end{blockarray}$$

Obviously $M^Q_{[a,b]}(V)=A_pA_{p-1}\cdots A_s$ and $$\zeta_Q(V)_{[b-1]\times[a+1]}\stackrel{\mbox{\rm{\tiny\textcircled{1}+\textcircled{2}}}}{\Longrightarrow}\left[\begin{matrix}
                                                                                                                  A_pA_{p-1}\cdots A_s &  &  \\
                                                                                                                   & I &  \\
                                                                                                                   &  & 0
                                                                                                                \end{matrix}\right].$$

The proof is now complete, and the formula for ranks holds naturally.
\end{proof}

Now, as an consequence, we can also observe that both $\lambda_Q$ and $\mu_Q$ are injective maps. Moreover, it is worth mentioning the following fact.
\begin{lem}\label{conrank}
  $$\# \{(x,y)|\rank \zeta_Q(V)_{[x]\times[y]}=\rank (v_Q)_{[x]\times [y]},\forall V\in rep_Q(\mathbf{d})\}=\frac{n(n+1)}{2}.$$
\end{lem}
\begin{proof}
  That is easy to see that $\zeta_Q(V)_{[x]\times[y]}$ is of constant rank independent with $A_\cdot,B_\cdot$ if and only if this rank is always $\rank (v_Q)_{[x]\times [y]}$.

  In the matrices (\ref{*}) and (\ref{**}) in the proof of Theorem \ref{main}, for those $1\leq x,y\leq n'$, we can see that $\zeta_Q(V)_{[x]\times[y]}$ is of a constant rank if and $\zeta_{Q'}(V')_{[x]\times[y]}$ is of constant rank. (In (\ref{**}), note that if $x=n'$ then the only constant-rank case is $\zeta_{Q'}(V')_{[n']\times[n']}$.)

  Besides these sub-matrices, the other constant-rank sub-matrices $\zeta_Q(V)_{[x]\times[y]}$ are given by
  $$\mbox{in (\ref{*}), }\{x\leq y|1\leq x\leq n, n'<y\leq n\}$$
  or $$\mbox{in (\ref{**}), }\{x< y-1\mbox{ or }x=n| 1\leq x\leq n, n'<y\leq n\}\cup\{n'<x\leq n,y=n'\}.$$

  Denote $n-n'=k$. Counting the numbers we have $$\mbox{in (\ref{*}), }kn-\frac{k^2-k}{2}=\frac{n(n+1)}{2}-\frac{n'(n'+1)}{2}$$
  $$\mbox{in (\ref{**}), }kn-\frac{k^2+k}{2}+k=\frac{n(n+1)}{2}-\frac{n'(n'+1)}{2}.$$

  Use induction we immediately obtain the lemma.
\end{proof}

We simply state that a southwest $[x]\times[y]$ submatrix (of $\zeta_Q,Z$, etc.) is \textbf{of constant rank} if $(x,y)$ meets the requirements outlined in the lemma above. Consider that $\frac{n(n+1)}{2}+\binom{n}{2}=n^2$, in which $\binom{n}{2}$ represents the count of intervals in $Q$, we infer the subsequent corollary.
\begin{coro}
  For $V,W\in rep_Q(\mathbf{d})$, the orbits $\mathcal{O}_V=\mathcal{O}_W$ if and only if $$\rank \zeta_Q(V)_{[x]\times[y]}=\rank \zeta_{Q}(W)_{[x]\times[y]}$$for all $x,y\in Q_0$.
\end{coro}

Moreover, the proof of Theorem \ref{main} is purely matrix theoretical and it does not require the use of an inverse matrix. Therefore, we can replace the matrices $A_\cdot,B_\cdot$ by matrices of coefficients in other $k$-polynomial algebra and still obtain a similar corollary.
\begin{coro}\label{k-alg}
  Let $F_Q$ be the matrix obtained by replacing $A_i,B_i$ by $f_{\alpha_i},f_{\beta_i}$ respectively in the definition of $\zeta_Q(V)$ for $V\in rep_Q(\mathbf{d})$. Then for any interval $[a,b]\subset Q$, there is $$(F_Q)_{[\mu_Q(b)]\times[\lambda_Q(a)]}\stackrel{\mbox{\rm{\tiny\textcircled{1}+\textcircled{2}}}}{\Longrightarrow}
  \left[\begin{matrix}
          M^Q_{[a,b]} &  &  \\
           & I &  \\
           &  & 0
        \end{matrix}\right],$$where the identity matrix is of size $\rank (v_Q)_{[\mu_Q(b)]\times[\lambda_Q(a)]}$.
\end{coro}
\subsection{Zelevinsky map is a scheme-theoretical isomorphism}
\begin{defn}\label{wQ}
  We define the Zelevinsky permutation $w_Q(\mathbf{r})$, where $\mathbf{r}$ is a rank parameter, to be the unique $(S'_Q,S_Q)$-blocked permutation in $S_N$ satisfying:
  \begin{itemize}
    \item[(ZP1)] $\rank w(\mathbf{r})_{[x]\times [y]}=\rank \zeta_Q(V)_{[x]\times [y]}$ for any $V\in \mathcal{O}_\mathbf{r}, x,y\in Q_0$;
    \item[(ZP2)] $w_Q(\mathbf{r})$ is of Z-type (see the end of Section 3), corresponding to $(S'_Q,S_Q)$-blocking.
  \end{itemize}
\end{defn}

To obtain this unique $w_Q(\mathbf{r})$ for a given $\mathbf{r}$, one only needs to place an identity matrix of appropriate size (possibly smaller than the block) in each $([x],[y])$-block. These identity matrices should be arranged from southwest to northeast in the same row $[x]$, and from northwest to southeast in the same column $[y]$. Furthermore, for any $V\in \mathcal{O}_\mathbf{r}$, we have $$\rank w_Q(\mathbf{r})_{[x]\times [y]}=\rank \zeta_Q(V)_{[x]\times [y]}.$$
\begin{examp}\label{sp-example}
  Consider the quiver and representation in Example \ref{ex1}, as well as Example \ref{Zex}. Let the dimension vector be $\mathbf{d}$ with $d_i=2$ for $1\leq i\leq 4$ and $d_i=1$ for $5\leq i\leq 7$. If the representation $V$ has data $$B_1=B_2=\left[\begin{matrix}
                                                   1 & 0 \\
                                                   0 & 1
                                                 \end{matrix}\right],A_1=\left[\begin{matrix}
                                                                                 1 & 0 \\
                                                                                 0 & 0
                                                                               \end{matrix}\right],A_2=\left[\begin{matrix}
                                                                                                               0 \\
                                                                                                               1
                                                                                                             \end{matrix}\right],B_3=A_3=[1],$$
Then its rank parameter $\mathbf{r}$ is given by $$ \begin{tabular}{|c|c|c|c|c|c|c|}
                                          \hline
                                          \diagbox{$b$}{$r_{[a,b]}$}{$a$}
                                              & $1$ & $2$ & $3$ & $4$ & $5$ & $6$  \\\hline
                                          $2$ & $2$ &     &     &     &     &    \\\hline
                                          $3$ & $2$ & $2$ &     &     &     &    \\\hline
                                          $4$ & $2$ & $2$ & $1$ &     &     &    \\\hline
                                          $5$ & $2$ & $2$ & $0$ & $1$ &     &    \\\hline
                                          $6$ & $3$ & $3$ & $1$ & $1$ & $1$ &    \\\hline
                                          $7$ & $3$ & $3$ & $1$ & $2$ & $1$ & $1$  \\\hline
                                        \end{tabular}$$
And the southwest submatrices $\zeta_Q(V)_{[x]\times[y]}$ has ranks
$$\begin{tabular}{|c|c|c|c|c|c|c|c|}
                                          \hline
\diagbox{$x\in S'_Q$}{$y\in S_Q$}
                                              & $7$ & $5$ & $4$ & $1$ & $2$ & $3$  & $6$ \\\hline
                                          $1$ & $1$ & $2$ & $4$ & $6$ &$ 8$ & $10$ & $11$ \\\hline
                                          $2$ & $1$ & $2$ & $4$ &$2+4$& $6$ & $8$  &  $9$ \\\hline
                                          $5$ & $1$ & $2$ & $4$ &$2+4$&$2+4$& $6$  &$7$     \\\hline
                                          $7$ & $1$ &$1+1$&$1+3$&$3+3$&$3+3$& $1+5$& $6$    \\\hline
                                          $6$ &$1+0$&$2+0$&$1+2$&$3+2$&$3+2$& $1+4$&  $5$   \\\hline
                                          $4$ & $0$ &$1+0$&$0+2$&$2+2$&$2+2$&   $4$&   $4$  \\\hline
                                          $3$ & $0$ & $0$ &$1+0$&$2+0$&$2+0$& $2$  &   $2$  \\\hline
                                        \end{tabular}$$
where $V$ can be replaced by any representation in $\mathcal{O}_{\mathbf{r}}$ and, for example, ``$2+4$" located at the position $x=2,y=1$ means $$\rank \zeta_Q(V)_{[2]\times[1]}=\rank M_{[1,2]}^Q(V)+\rank (v_Q)_{[2]\times[1]}=2+4.$$ Hence, by comparing values from adjacent positions in the table, we obtain that the numbers of $1$'s in the $([x],[y])$-block of $w_Q(\mathbf{r})$ should be $$\begin{tabular}{|c|c|c|c|c|c|c|c|}
                                          \hline
\diagbox{$x\in S'_Q$}{$y\in S_Q$}
                                              & $7$& $5$&$4$ &$1$ &$2$ & $3$  & $6$ \\\hline
                                          $1$ &    &    &    &    &$ 2$&    &    \\\hline
                                          $2$ &    &    &    &    &    & $2$  &     \\\hline
                                          $5$ &    &    &    &    &    &     &$1$     \\\hline
                                          $7$ &    &    & $1$&    &    &     &       \\\hline
                                          $6$ & $1$&    &    &    &    &   &       \\\hline
                                          $4$ &    & $1$&    & $1$&    &     &       \\\hline
                                          $3$ &    &    &$1$ & $1$&    &     &       \\\hline
                                        \end{tabular}$$where the blanks are zero. Moreover, we require the permutation $w_Q(\mathbf{r})$ to follow the form defined in Definition \ref{wQ}. Then $w_Q(\mathbf{r})$ is uniquely determined, it is
                                        $$\left[\begin{array}{c:c:cc:cc:cc:cc:c}
                                              &   &   &   &   &   & 1  & 0  &   &   &   \\
                                              &   &   &   &   &   & 0  & 1  &   &   &   \\ \cdashline{1-11}
                                              &   &   &   &   &   &   &   &  1 & 0  &   \\
                                              &   &   &   &   &   &   &   &  0 &  1 &   \\ \cdashline{1-11}
                                              &   &   &   &   &   &   &   &   &   &  1 \\ \cdashline{1-11}
                                              &   &  1 & 0  &   &   &   &   &   &   &   \\\cdashline{1-11}
                                            1  &  &   &   &   &   &   &   &   &   &   \\ \cdashline{1-11}
                                              & 0 &   &   &  1& 0  &   &   &   &   &   \\
                                              & 1 &   &   & 0 &  0 &   &   &   &   &   \\ \cdashline{1-11}
                                              &   & 0 & 0 & 0 &  1 &   &   &   &   &   \\
                                              &   & 0 & 1 & 0 &  0 &   &   &   &   &
                                          \end{array}\right].$$
\end{examp}
Recall that $\zeta_Q:rep_Q(\mathbf{d})\to \mathbb{A}^{v_Q}$ induces the homomorphism of $k$-algebras $$\zeta_Q^*:k[\mathbb{A}^{v_Q}]\to k[rep_Q(\mathbf{d})].$$
In $\mathbb{A}^{v_Q}$, the Kazhdan-Lusztig variety $Y^{v_Q}_{w_Q(\mathbf{r})}$ is defined by the ideal
$$I_{w_Q(\mathbf{r})}=<\mbox{minors of size $\rank w_Q(\mathbf{r})_{[x]\times [y]}+1$ in $Z_{[x]\times[y]}$}|x,y\in Q_0>,$$
where $Z$ is the corresponding matrix of affine coordinates in $\mathbb{A}^{v_Q}$ as defined in the end of previous section.
In $rep_Q(\mathbf{d})$, the quiver locus $\overline{\mathcal{O}_\mathbf{r}}$ is defined by the ideal
$$I_\mathbf{r}=<\mbox{minors of size $r_{[a,b]}+1$ in $M^Q_{[a,b]}$}|[a,b]\subset Q>$$
\begin{lem}\label{1}
  $\zeta^*_Q(I_{w_Q(\mathbf{r})})\subset I_\mathbf{r}$.
\end{lem}
\begin{proof}
  For the generators of $I_{w_Q(\mathbf{r})}$, their composition with $\zeta_Q$ is the corresponding minors of $(F_Q)_{[x]\times[y]}$. Since $rank (w_Q(\mathbf{r}))_{[x]\times [y]}=\zeta_Q(V)_{[x]\times [y]}$, these minors vanish on any $V\in \overline{\mathcal{O}_\mathbf{r}}$.
\end{proof}

By Theorem \ref{main}, we can observe the following fact.
\begin{lem}\label{2}
  For any generator $$\mbox{minor of size $r_{[a,b]}+1$ in $M^Q_{[a,b]}$},[a,b]\subset Q,$$ of $I_\mathbf{r}$, it has preimage in $I_{w_Q(\mathbf{r})}$ under $\zeta_Q^*$.
\end{lem}

The following lemma in linear algebra is straightforward.
\begin{lem}\label{X.Y.XY}
  Let $r,r_1,r_2$ be any positive integers, and $X=(x_{ij}),Y=(y_{ij})$ be two matrices of size $r_1\times r$ and $r\times r_2$ respectively. Denote $XY=(z_{ij})$. Then the polynomials given by rule of matrix multiplication $$z_{ij}-\sum_{t=1}^{r}x_{it}y_{tj}$$ can be viewed as minors of size $r+1$ in the matrix
$$\left[\begin{matrix}
                                         Y & I_{r}\\
                                         XY & X
                                       \end{matrix}\right]$$
\end{lem}
\begin{lem}\label{3}
$\ker\zeta^*_Q\subset I_{w_Q(\mathbf{r})}$.
\end{lem}
\begin{proof}
This kernel is the defining ideal of the image of $\zeta_Q$, which is a closed subvariety in $\mathbb{A}^{v_Q}$. The defining ideal of the image is generated by two types of polynomials: the first are the coordinate functions on the places of constant zero entries in $F_Q$ which is not a constant zero entry in $Z$; the second are determined by the rule of matrix multiplications (see Example \ref{ker-example} below).

Now in matrix \ref{*} and \ref{**} we consider the special area bounded by the four horizon and vertical lines (i.e. the blocks with row/column label greater than $n'$).

In \ref{*}, we can see that if one zero block has neither a constant identity south nor west, then every coordinate corresponding to this block can be viewed as a minor of size $\rank w_Q(\mathbf{r})_{[n']\times[t(L_l)]}+1$ in $Z_{[n']\times[t(L_l)]}$. And by Lemma \ref{X.Y.XY}, the generators given by matrix multiplications which happen in the special area lie in the ideal $I_{w_Q(\mathbf{r})}$ as well.

In \ref{**}, we can see that if one zero block has neither a constant identity south nor west, then itself is contained in a zero southwest submatrix and can be viewed as a $1\times 1$ minors in $I_{w_Q(\mathbf{r})}$. And in this special area there is no matrix multiplication happens.

Then this lemma can be proved by induction.
\end{proof}\begin{examp}\label{ker-example}Consider the quiver
  $$\begin{tikzcd}
  Q:1 \arrow[r, "\beta_1"'] & 2 \arrow[r, "\beta_2"'] & 3 & 4 \arrow[l,"\alpha_1"]&5 \arrow[l,"\alpha_2"]\end{tikzcd},$$with dimension vector $\mathbf{d}=(d_1,d_2,d_3,d_4,d_5)=(2,2,2,2,2)$. The affine space $\mathbb{A}^{v_Q}$ is the space of matrices of the form $$\left[
  \begin{matrix}
  &  &  &  & 1 &  &  &  &  &  \\
  &  &  &  &  & 1 &  &  &  &  \\
  &  &  &  & x_{3,5} & x_{3,6} & 1 &  &  &  \\
  &  &  &  & x_{4,5} & x_{4,6} &  & 1 &  &  \\
  1 &  &  &  &  &  &  &  &  &  \\
  & 1 &  &  &  &  &  &  &  &  \\
  x_{7,1} & x_{7,2} & 1 &  &  &  &  &  &  &  \\
  x_{8,1} & x_{8,2} &  & 1 &  &  &  &  &  &  \\
  x_{9,1} & x_{9,2} & x_{9,3} &  x_{9,4} & x_{9,5} & x_{9,6} & x_{9,7} & x_{9,8} & 1 &  \\
  x_{10,1} & x_{10,2} & x_{10,3} & x_{10,4} & x_{10,5} & x_{10,6} & x_{10,7} & x_{10,8} &  & 1
                                                                                                                         \end{matrix}\right],$$where $x_{i,j}$ are the affine coordinates and blank positions are zero.
  And given any representation$$\begin{tikzcd}
  V=V_1 \arrow[r, "B_1","\beta_1"'] & V_2 \arrow[r, "B_2","\beta_2"'] & V_3 & V_4 \arrow[l,"\alpha_1","A_1"']&V_5 \arrow[l,"\alpha_2","A_2"']\end{tikzcd},$$the image of this representation under $\zeta_Q$ is (every block is of size $2\times 2$)
  $$\zeta_Q(V)=\left[\begin{matrix}
                          0 & 0 & I & 0 & 0 \\
                          0 & 0 & B_1 & I & 0 \\
                          I & 0 & 0 & 0 & 0 \\
                          A_2 & I & 0 & 0 & 0 \\
                          0 &  A_1 & B_2B_1& B_2 & I
                        \end{matrix}\right].$$
  The image variety of $\zeta_Q$, i.e. $\zeta_Q(rep_Q(\mathbf{d}))$, is defined in $\mathbb{A}^{v_Q}$ by the ideal generated by
  $$x_{9,1},x_{9,2},x_{10,1},x_{10,2}$$
  $$x_{9,5}-x_{9,7}x_{3,5}-x_{9,8}x_{4,5},x_{9,6}-x_{9,7}x_{3,6}-x_{9,8}x_{4,6}$$
  $$x_{10,5}-x_{10,7}x_{3,5}-x_{10,8}x_{4,5},x_{10,6}-x_{10,7}x_{3,6}-x_{10,8}x_{4,6}.$$
\end{examp}
Now we are ready to present the main theorem.
\begin{theo}\label{MAIN}
  $(\zeta_Q^*)^{-1}(I_\mathbf{r})=I_{w_Q(\mathbf{r})}$. Therefore, $\zeta_Q: \overline{\mathcal{O}_\mathbf{r}}\to Y^{v_Q}_{w_Q(\mathbf{r})}$ is a scheme-theoretical isomorphism.
\end{theo}
\begin{proof}
  This is a consequence of Lemma \ref{1}, Lemma \ref{2} and Lemma \ref{3}, because Lemma \ref{2} and Lemma \ref{3} provide $$\zeta^*_Q(I_{w_Q(\mathbf{r})})\supset I_\mathbf{r}.$$
\end{proof}

For instance, in Example 4.3, we have the quiver locus $\overline{\mathcal{O}_{\mathbf{r}}}=\overline{\mathcal{O}_V}$ is isomorphic to $Y^{v_Q}_{w_Q(\mathbf{r})}$, where $$\mathbf{d}=(2,2,2,2,1,1,1),$$ $$v_Q=(6,5,8,9,1,2,3,4,10,11,7),w_Q(\mathbf{r})=(7,8,6,10,9,11,1,2,3,4,5).$$

Moreover, it is important to note that applying our construction to equioriented or bipartite type A quivers yields the Zelevinsky map designed by A.V. Zelevinsky in \cite{zelevinskii_two_1985} and the map generalized by Kinser and Rajchgot in \cite{kinser_type_2015}, respectively.

This Zelevinsky isomorphism is also $\mathrm{GL}(\mathbf{d})$-equivariant, if we view $G/P$ as a $\mathrm{GL}(\mathbf{d})$-space via the block diagonal embedding $$\begin{array}{ccc}
\mathrm{GL}(\mathbf{d})=\prod_{x\in Q_0}\mathrm{GL}_{d_x}&\to &G=\mathrm{GL}_N\\
(g_x)_{x\in Q_0}&\mapsto& diag(g_x)_{x\in S'_Q}.\end{array}$$Note that the order of blocks in $diag(g_x)_{x\in S'_Q}$ follows the order of vertices in $S'_Q$. Then $\mathrm{GL}(\mathbf{d})$ acts on $G/P$ by multiplication. Thus, every $\mathrm{GL}(\mathbf{d})$-orbit in $rep_Q(\mathbf{d})$ is mapped to a $\mathrm{GL}(\mathbf{d})$-orbit in $G/P$.
\begin{prop}\label{orbitimage}
  $\zeta_Q(\mathcal{O}_\mathbf{r})=(\cup X_w^\circ)\cap X^{v_Q}_\circ=\cup Y_w^{v_Q}$, where the union is taken over all the $w\in W^P$ satisfying the (ZP1) for $\mathbf{r}$ in Definition \ref{wQ} (but not necessarily (ZP2)).
\end{prop}
\begin{proof}
  This can be captured by computing the $\mathrm{GL}(\mathbf{d})$-orbit in $G/P$ directly or noting that (recall the end of Section 3 for the notations)$$\begin{array}{l}X_w^\circ\cap X^{v_Q}_\circ=X_w^\circ\cap \mathbb{A}^{v_Q}\\=\{A\in \mathbb{A}^{v_Q}|\rank A_{i\times [y]}=\rank w_Q(\mathbf{r})_{i\times [y]},i\in \{1,\cdots,N\},y\in Q_0\}\end{array}.$$
\end{proof}
\subsection{Some corollaries}
After constructing of our isomorphism, an immediate corollary is that we can recover the results of Bobinski and Zwara in \cite{bobinski_normality_2001}, as well as those of Kinser and Rajchgot in \cite{kinser_type_2015}. These results show that type A quiver loci are normal, Cohen-Macaulay, and have rational singularities over fields of characteristic 0. Bobinski and Zwara achieved this result using the technique of Hom-controlled functors and proved that the types of singularities of type A (and type D) quiver loci are independent of the orientation of the quiver. Kinser and Rajchgot identified type A quiver locus of arbitrary orientation, up to a smooth factor, with an open dense subvariety of a Kazhdan-Lusztig variety. This result can now be reinterpreted most directly because it is known that every Kazhdan-Lusztig variety is normal, Cohen-Macaulay, and has rational singularities, as proven in \cite{Brion2005} and \cite{kazhdan_representations_1979}.
\begin{coro}
  The quiver loci are normal, Cohen-Macaulay for the type A quiver of arbitrary orientation. Moreover, if the ground field has characteristic 0, the type A quiver loci has rational singularities.
\end{coro}

Additionally, our construction for arbitrary orientation has the potential to enhance existing proofs for important problems. For example, it provides an opportunity to obtain the vanishing of intersection cohomology in odd degrees for type A quiver loci in a clearer manner. This is a crucial step in proving the existence of canonical bases from Lusztig. Specifically, let $X$ be a type A quiver loci defined over $\overline{\mathbb{F}_{p^r}}$, and $$\pi_{\mathbb{Q}_l}=\tau_{\leq n-1}R_{j_n^*}\tau_{\leq n-2}\cdots \tau_{\leq 0}R_{j_1^*}\mathbb{Q}_l$$be the sheaf from Deligne on $X$ (see \cite[(3.2), (4.2)]{SchubertPoincare} for the definitions and notations, where $p$ is prime and $l\nmid p$), then denote $\mathcal{H}^i (X)$ to be the cohomology sheaf of $\pi_{\mathbb{Q}_l}$(in degree $i$). And let $F$ be the Frobenius morphism defined on $X$. By the Theorem \cite[4.2]{SchubertPoincare} and our isomorphism, there is the following corollary.
\begin{coro}
  The sheaf $\mathcal{H}^i(X)$ is zero for odd $i$. If $i$ is even and the point $x_0\in X$ are defined over $\mathbb{F}_{p^r}$ then all eigenvalues of $(F^r)^*$ on $\mathcal{H}^i_{x_0}(X)$ are equal to $(p^r)^{\frac{i}{2}}$.
\end{coro}
This argument allows us to improve Lusztig's proof of existence of canonical bases (see \cite[Theorem 5.4]{lusztigcanbase}) regarding technique on special oriented quivers, and directly obtain the calculation results for the intersection cohomology of any oriented quiver.
\section{Zelevinsky permutation and indecomposable multiplicities}\label{multi}
The $\mathrm{GL}(\mathbf{d})$-orbit of a representation $V$ in $rep_Q(\mathbf{d})$ is also the isomorphism class of $V$ in the category of representations of $Q$. Besides the rank parameter, we can characterize this orbit by the indecomposable factorization $V=\oplus m_{pq}I_{pq}$ as well, where $I_{pq} (1\leq p\leq q\leq n)$ is the indecomposable representation of $Q$ given by
$$\begin{tikzcd}
\cdots \arrow[r, no head] & V_{p-1} \arrow[r, no head] \arrow[d, Rightarrow, no head]& V_p \arrow[r, "1", no head] \arrow[d, Rightarrow, no head]  & \cdots \arrow[r, "1", no head] & V_q \arrow[d, Rightarrow, no head] \arrow[r, no head] & V_{q+1} \arrow[r, no head] \arrow[d, Rightarrow, no head]& \cdots \\
                          &              0        & k^1                                                                                                                   &                                            & k^1                                                   &           0           &
\end{tikzcd}$$

In this section, given a representation $V\in rep_Q(\mathbf{d})$ and its rank parameter $\mathbf{r}$, we can establish the clear and useful connection between the Zelevinsky permutation $w(\mathbf{r})$ and the multiplicities $m_{ij}$ of indecomposable factors in $V=\oplus m_{pq}I_{pq}$.

\begin{defn}
  For a $(S'_Q,S_Q)$-blocked permutation $w\in S_N$, now we define a $n\times n$ matrix $M(w)$ in nonnegative integers, which is nearly a $(S'_Q,S_Q)$-blocked matrix: it satisfies (BM1) and (BM2) in Definition \ref{defBM} but every block row/column has only height/width one (so its row/columns are labeled and we can also use similar notation before). And the $([x],[y])$-block $M(w)_{[x],[y]}$ (only one entry) of $M(w)$ is $$M(w)_{[x],[y]}=\rank w_{[x]\times[y]}-\rank w_{[S(x)]\times[y]}-\rank w_{[x]\times[W(y)]}+\rank w_{[S(x)]\times[W(y)]},$$ where the undefined rank terms are $0$. We call $M(w)$ the\textbf{ multiplicity matrix }of $w$.
\end{defn}

This matrix $M(w)$ collects the numbers of $1$ entries in every block of $w$. We can see that given any $n\times n$ matrix $M$, there exists a permutation $w\in S_N$ such that $M=M(w)$ if and only if
    $$\sum_{y'}M_{[x],[y']}=d_x, \sum_{x'}M_{[x'],[y]}=d_y,\forall x,y\in Q_0.$$
Moreover, at this time there exists a unique $w$ of Z-type such that $M=M(w)$.

The relations between the rank parameter and the indecomposable multiplicity are revealed in \cite{degenerationofA}.

\begin{prop}[cf.\cite{degenerationofA}]\label{m-r}
  Given a rank parameter $\mathbf{r}$ and $V\in \mathcal{O}_\mathbf{r}$ with $V=\oplus m_{pq}I_{pq}$, the parameter $\mathbf{r}$ and multiplicities $m_{pq}$ have relations :
  \begin{itemize}
    \item[(1)]if $s_a<p<s_{a+1},s_b<q<s_{b+1},$ \item[]$m_{pq}=(-1)^{b-a}(r_{[p,q]}-r_{[p-1,q]}-r_{[p,p+1]}+r_{[p-1,q+1]});$
    \item[(2)]if $s_b<q<s_{b+1},$ \item[]$m_{s_aq}=(-1){b-a+1}(r_{[s_{a+1},q]}-r_{[s_{a-1},q]}-r_{[s_{a+1},q+1]}+r_{[s_a-1,q+1]});$
    \item[(3)]if $s_a<p<s_{a+1},$ \item[]$m_{ps_b}=(-1)^{b-a}(r_{[p,s_{b-1}]}-r_{[p-1,s_{b-1}]}-r_{[p,s_b+1]}+r_{[p-1,s_b+1]});$
    \item[(4)]if $s_a\neq s_b,$ \item[]$m_{s_a s_b}=(-1)^{b-a+1}(r_{[s_{a+1},s_{b-1}]}-r_{[s_a-1,s_{b-1}]}-r_{[s_{a+1},s_b+1]}+r_{[s_a-1,s_b+1]});$
    \item[(5)]in particular, \item[]$m_{s_a s_a}=r_{[s_a,s_a]}-r_{[s_a-1,s_a+1]}=d_{s_a}-r_{[s_a-1,s_a+1]},$
  \end{itemize}
where $r_{[x,x]}=d_x, r_{[y,x]}=0, \forall 1\leq x\leq n$ and $x<y.$
\end{prop}

The next theorem shows that the multiplicity matrix effectively records the indecomposable factors of a representation. This result is obvious in the equioriented case, and in the article \cite{fourformulae} it is used as the definition of Zelevinsky permutation for equioriented quiver. Surprisingly, it can also be applied to arbitrary orientations under our design. This strengthens the notion that our construction not only serves as a technical unification of previous Zelevinsky maps but also offers a more intrinsic approach.

Let $n_{xy}=\sum_{p\leq x<y\leq q}m_{pq}$.
\begin{theo}
  Given a rank parameter $\mathbf{r}$ and $V\in \mathcal{O}_\mathbf{r}$ with $V=\oplus m_{pq}I_{pq}$, the entries of the multiplicity matrix $M(w_Q(\mathbf{r}))$ satisfy:
  $$M(w_Q(\mathbf{r}))_{[x],[y]}=\left\{\begin{array}{ll}
                                      m_{yx} &,\mbox{ if }y\leq x  \\
                                      n_{xy}&,\mbox{ if }y=x+1 \\
                                      0 &,\mbox{ else}
                                    \end{array}\right.$$
\end{theo}
\begin{proof}
  We only prove for three cases, as the proofs for the other cases are similar.

\textbf{Case 1}: Neither of $y<x$ are critical points, and the directions of arrows are $$\begin{tikzcd}
\cdots & y \arrow[l] & \cdots \arrow[l] \arrow[r] & x \arrow[r] & \cdots
\end{tikzcd}$$
Note that $x$ lies in some $R_\cdot \backslash t(R_\cdot)$ and $y$ lies in some $L_\cdot\backslash t(L_\cdot)$. Whether $x+1$ and $y-1$ are critical points or not, their positions in the block rows and columns are
$$\begin{blockarray}{cccccc}\begin{block}{c[ccccc]}
   \vdots&  &  &  &&  \\
  x-1 &  &  &  & & \\
  x &  &  &\cdots  &  &\\
  (x+1)^R_Q &  &  &  &  &\\
  \vdots &   &   &   & & \\\end{block}
         &  \cdots  & y+1 & y & (y-1)^L_Q &\cdots
\end{blockarray}$$
Consider the source-sink behaviors of these points, we have$$\begin{array}{rcl}\mu_Q(x)=x&,&\mu_Q(x+1)=(x+1)^R_Q,\\\lambda_Q(y-1)=y&,&\lambda_Q(y)=y+1.\end{array}$$
Therefore,$$\begin{array}{l}
\rank w_Q(\mathbf{r})_{[x]\times[y]}=r_{[y-1,x]},\\
\rank w_Q(\mathbf{r})_{[(x+1)^R_Q]\times[y]}=r_{[y-1,x+1]},\\
\rank w_Q(\mathbf{r})_{[x]\times[y+1]}=r_{[y,x]}, \\
\rank w_Q(\mathbf{r})_{[(x+1)^R_Q]\times[y+1]}=r_{[y,x+1]},\end{array}$$and
    $$\begin{array}{l}
    M(w_Q(\mathbf{r}))_{[x],[y]}
    =r_{[y-1,x]}-r_{[y-1,x+1]}-r_{[y,x]}+r_{[y,x+1]}=m_{yx},
  \end{array}$$where the sign $(-1)^{a-b}$ in Proposition \ref{m-r}(1) is $-1$ at this time for $s_a<y<s_{a+1}$ and $s_b<x<s_{b+1}$.

\textbf{Case 2}: $x=y=s_a$ is a source point, and there is $$\begin{tikzcd}
\cdots & x-1 & x(=s_a) \arrow[l] \arrow[r]  & x+1 & \cdots &                                                \\
\end{tikzcd}$$
Then $x$ lies in some $R_\cdot\backslash t(R_\cdot)$ and some $L_\cdot\backslash t(L_\cdot)$. The positions of $x$ within the rows and columns are
$$\begin{blockarray}{cccccc}\begin{block}{c[ccccc]}
   \vdots&  &  &  &&  \\
  s_{a-1}-1 &  &  &  & & \\
  x &  &  & \cdots &  &\\
  (x+1)^R_Q &  &  &  &  &\\
  \vdots &   &   &   & & \\\end{block}
         &  \cdots  & s_{a+1}+1 & x & (x-1)^L_Q &\cdots
\end{blockarray}$$
Consider the source-sink behaviors of these points. We have $$\begin{array}{rcl}\mu_Q(s_{a-1})=x&,&\mu_Q(x+1)=(x+1)^R_Q,\\ \lambda_Q(s_{a+1})=s_{a+1}+1&,&\lambda_Q(x-1)=x.\end{array}$$
Because $\lambda_Q(x-1)=x$, $\mu_Q(s_{a-1})=x$ with $x-1\not< s_{a-1}$ and so on, the southwest matrices $w_Q(\mathbf{r})_{[x]\times[x]}$, $w_Q(\mathbf{r})_{[x]\times [s_{a+1}+1]}$ and $w_Q(\mathbf{r})_{[s_{a+1}+1]\times [(x+1)^R_Q]}$ have the same ranks as the corresponding southwest matrices of permutation $v_Q$, i.e., they are of constant rank (see Lemma \ref{conrank}). Therefore, $$\rank w_Q(\mathbf{r})_{[x]\times[x]}-\rank w_Q(\mathbf{r})_{[x]\times [s_{a+1}+1]}+\rank w_Q(\mathbf{r})_{[s_{a+1}+1]\times [x+1]}=d_x.$$We have $$\begin{array}{l}
    M(w_Q(\mathbf{r}))_{[x],[x]}=d_x-\rank w_Q(\mathbf{r})_{[(x+1)^R_Q]\times [x]}\\
    =d_x-r_{[x-1,x+1]}=m_{xx}.
  \end{array}$$

\textbf{Case 3}: $y=x+1$. At this time, if the arrow between $x$ and $y$ is of left direction $$\begin{tikzcd}
x & y \arrow[l, "A_i"']
\end{tikzcd}$$then by definition, the $([x],[y])$-block is exactly $A_i$. $$\begin{blockarray}{ccccc}\begin{block}{c[cccc]}
     \vdots  &   &  &     & \\
  x    &   &0 & A_i & \\
  S(x) &   &0 & 0   &\\
    \vdots   &   &  &     & \\\end{block}
       & \cdots  &W(y)& y &\cdots
\end{blockarray}$$And in the induction (i.e. matrix \ref{*} and \ref{**}) it is easy to observe that the southwest matrices $w_Q(\mathbf{r})_{[S(x)]\times[y]}$, $w_Q(\mathbf{r})_{[x]\times[W(y)]}$, and $w_Q(\mathbf{r})_{[S(x)]\times[W(y)]}$ are always of constant rank. Moreover, by considering the matrices (\ref{*}) and (\ref{**}) or checking the source-sink behaviors directly, we have $$\rank w_Q(\mathbf{r})_{[x],[y]}=\rank A_i.$$ Therefore, $$M(w_Q(\mathbf{r}))_{[x]\times[y]}=\rank w_Q(\mathbf{r})_{[x]\times[y]}=\rank A_i=\sum_{p\leq x<y\leq q}m_{pq}=n_{xy}.$$

If the arrow between $x$ and $y$ is of right direction$$\begin{tikzcd}
x \arrow[r, "B_i"] & y
\end{tikzcd}$$then $x$ belongs to some $R_\cdot\backslash (t(R_\cdot))$ and $y$ belongs to some $L_\cdot\backslash s(L_\cdot)$. Consequently, there must be $S(x)=y^R_Q$ and $W(y)=x^L_Q$. $$\begin{blockarray}{ccccc}\begin{block}{c[cccc]}
   \vdots  &  &  &    & \\
  x  &  & \cdots & & \\
  y^R_Q &  &   & \cdots &\\
  \vdots     &   &  &  & \\\end{block}
        & \cdots  &  x^L_Q& y &\cdots
\end{blockarray}$$And by definition $$\mu_Q(y)=y^R_Q,\lambda_Q(x)=x^R_Q.$$ Similarly, the matrices $w_Q(\mathbf{r})_{[x]\times[y]}$, $w_Q(\mathbf{r})_{[S(x)]\times[y]}$, and $w_Q(\mathbf{r})_{[x]\times[W(y)]}$ are always of constant rank. At this time, $$M(w_Q(\mathbf{r}))_{[x],[y]}=\rank w_Q(\mathbf{r})_{[x]\times[y]}=\rank B_i=\sum_{p\leq x<y\leq q}m_{pq}=n_{xy}.$$

Here we only prove these three cases. Once all the cases of $y\leq x+1$ have been proven, for $y>x+1$ there is naturally $M(w(\mathbf{r}))_{[x],[y]}=0$ because the other entries sum up to exactly the height/width of the block row/column.
\end{proof}
\begin{examp}
For the quiver and representation in Example \ref{Zex} before, the multiplicity matrix is given by  $$\begin{blockarray}{lccccccc}
 \begin{block}{l[ccccccc]}
                 1 &  0 & 0 & 0 & m_{11} & n_{12} & 0 & 0 \\
                2  &  0 & 0 & 0 & m_{12} & m_{22} & n_{23} & 0 \\
                 5 &  0 & m_{55} & m_{45} & m_{15} & m_{25} & m_{35} & n_{56} \\
                 7 &  m_{77} & m_{57} & m_{47} & m_{17} & m_{27} & m_{37} & m_{67} \\
                 6 &  n_{67} & m_{56} & m_{46} & m_{16} & m_{26} & m_{36} & m_{66} \\
                 4 &  0 & n_{45} & m_{44} & m_{14} & m_{24} & m_{34} & 0 \\
                  3 & 0 & 0 & n_{34} & m_{13} & m_{23} & m_{33} & 0\\
                  \end{block}
                     &7 &  5 &  4 &  1 &  2 &  3 &  6\end{blockarray}.$$If we take $V$ to be the specific representation as we have considered in Example \ref{sp-example},
                 then the indecomposable factorization of $V$ is $$V\cong I_{13}\oplus I_{14}\oplus I_{47}.$$ And its corresponding multiplicity matrix, which collects the numbers of $1$ entries in every block of $w_Q(\mathbf{r})$, is exactly shown as in Example \ref{sp-example}:  $$\begin{blockarray}{lccccccc}
 \begin{block}{l[ccccccc]}
                 1 &    &   &   &   & 2 &   &   \\
                2  &    &   &   &   &   & 2 &   \\
                 5 &    &   &   &   &   &   & 1 \\
                 7 &    &   & 1 &   &   &   &   \\
                 6 &  1 &   &   &   &   &   &   \\
                 4 &    & 1 &   & 1 &   &   &   \\
                  3 &   &   & 1 & 1 &   &   &  \\
                  \end{block}
                     &7 &  5 &  4 &  1 &  2 &  3 &  6\end{blockarray}.$$
\end{examp}
\begin{coro}\label{coroMwr}
  Given a $n\times n$ matrix $M$ in nonnegative integers, there exists a rank parameter $\mathbf{r}$ such that $M=M(w_Q(\mathbf{r}))$ if and only if $$\sum_{y'}M_{[x],[y']}=d_x, \sum_{x'}M_{[x'],[y]}=d_y,\forall x,y\in Q_0,$$and $$M_{[x],[y]}=0,\forall y>x+1.$$
\end{coro}

Therefore, firstly, this theorem establishes a fast transformation algorithm for rank parameters, Zelevinsky permutations, and indecomposable multiplicities. Moreover, this conclusion reveals the combinatorial nature of Zelevinsky permutations in a direct and transparent manner. Consequently, it provides us with a valuable means for further employing Zelevinsky maps to study the geometric problems of quiver loci. See Section \ref{singloc} for example.
\section{Further applications}
\subsection{Realization in $G/B$}\label{G/BKL}
Actually, we can also realize the type A quiver locus $\overline{\mathcal{O}_\mathbf{r}}$ as a Kazhdan-Lusztig variety in the full flag variety $G/B$ instead of in the partial flag $G/P_Q$. We only need to point out the isomorphism from $Y^{v_Q}_{w_Q(\mathbf{r})}\subset G/P_Q$ to some Kazhdan-Lusztig variety $\overline{B\gamma B/B} \cap B^-\theta B/B, \gamma,\theta \in W_G$. Let $P$ be a general $(S,S)$-blocked parabolic subgroup for a sequence $S$, and other notations are also the same as in Section \ref{KL}. For any $w\in W^P$, we define $\tilde{w}$ to be the longest element in the coset $wW_P$. Then we have that $B^-vP/P\cong B^-\tilde{v}B/B$, where $v\in W^P$ and $B^-\tilde{v}B/B$ can be identified with a matrix space $\tilde{\mathbb{A}}^{\tilde{v}}$ isomorphic to the matrix space $\mathbb{A}^v$ as in (\ref{AQ}). The following is an example of this isomorphism $\mathbb{A}^v\to \tilde{\mathbb{A}}^{\tilde{v}}$, where the dashed lines indicate the size of blocks in parabolic/Borel subgroups.
$$\left[\begin{array}{cc:ccc:cc}
     &  &1  &0  & 0 &  &  \\
     &  & 0 & 1 &  0&  &  \\
     &  & 0 &0  & 1 &  &  \\
    1 & 0 &  &  &  &  &  \\
    0 & 1 &  &  &  &  &  \\
     x_{61}& x_{62} & x_{63} & x_{64} &x_{65}  & 1 & 0 \\
    x_{71} & x_{72} & x_{73} & x_{74} & x_{75} &  0& 1
  \end{array}\right]\mapsto \left[\begin{array}{c:c:c:c:c:c:c}
     &  &0 &0  & 1 &  &  \\
     &  & 0 & 1 &  0&  &  \\
     &  & 1 &0  & 0 &  &  \\
    0 & 1 &  &  &  &  &  \\
    1 & 0 &  &  &  &  &  \\
     x_{62}& x_{61} & x_{65} & x_{64} &x_{63}  & 0 & 1 \\
    x_{72} & x_{71} & x_{75} & x_{74} & x_{73} &  1& 0
  \end{array}\right]$$
Let $Z$ be the generic matrix in the matrix space $\tilde{\mathbb{A}}^{\tilde{v}}$. Then, the Kazhdan-Lusztig variety $\overline{B\tilde{w}B/B}\cap B^-\tilde{v}B/B$ is defined by the ideal generated by $$\mbox{minors of size $\rank \tilde{w}_{i\times j}+1=\rank w_{i\times j}+1$ in $Z_{i\times j}$, $\forall i,j\in \{1,2,\cdots,N\}$}.$$ At this time, if $w\in W^P$ is of Z-type then we have that $\tilde{w}$ is the unique longest element in $W_{P'}\tilde{w}W_P$ and by using the Fulton's essential box or the similar analysis to that after Definition \ref{ztype} we can obtain that the defining ideal of $\overline{B\tilde{w}B/B}\cap B^-\tilde{v}B/B$ is generated by $$\mbox{minors of size $\rank w_{[x]\times [y]}+1$ in $Z_{[x]\times [y]}$, $\forall x,y\in Q_0$},$$ which is exactly the defining ideal $I_w$ under the isomorphism $\tilde{\mathbb{A}}^{\tilde{v}}\cong \mathbb{A}^{v}$. That follows that our previous Kazhdan-Lusztig variety $Y^{v_Q}_{w_Q(\mathbf{r})}$ is also isomorphic to $\overline{B\widetilde{w_Q(\mathbf{r})}B/B}\cap B^-\widetilde{v_Q}B/B$ in $G/B$. Taking Example \ref{Zex} as running example, the isomorphism from quiver locus $\overline{\mathcal{O}_{\mathbf{r}}}$ to $\overline{B\widetilde{w_Q(\mathbf{r})}B/B}\cap B^-\widetilde{v_Q}B/B$ is precisely
$$V=(A_\cdot,B_\cdot)\mapsto          \left[\begin{matrix}
                                          0 & 0 & 0 & J & 0 & 0 & 0 \\
                  0 & 0 & 0 & B_1J & J & 0 & 0 \\
                  0 & J & 0 & 0 & 0 & 0 & 0 \\
                  J & 0 & 0 & 0 & 0 & 0 & 0 \\
                   A_3J & B_3J & 0 & 0 & 0 & 0 & J \\
                  0 & A_2J & J & 0 & 0 & 0 & 0 \\
                   0 & 0 & A_1J & B_2B_1J & B_2J & J & 0\\
                                      \end{matrix}\right],$$ where $J$ is the anti-diagonal identity matrix.

                                      Realizing the quiver locus as a Kazhdan-Lusztig variety in the full flag $G/B$ has obviously numerous benefits. For instance, the study of Schubert varieties and Kazhdan-Lusztig varieties in the full flag variety is more advanced than in the partial flag variety. Therefore, we can apply more mature direct results to address the problems concerning type A quiver loci. However, as discussed in Section \ref{Zex}, most of the information in Zelevinsky permutations is stored in blocks. It is not favorable to apply the research method to full flag varieties roughly without considering the special structure of Zelevinsky permutations and these overall information. And when we are considering this information block-wisely, essentially we are addressing the problems in partial flag variety. Also, we can see that the construction for partial flag variety is more natural and logically comprehensible. And later a prospect to the generalization for type D quiver loci further reveals the practicality of partial flags.

                                      Therefore, we believe that realizing the Kazhdan-Lusztig variety in either partial flags or full flags has different advantages, and it is worthwhile to consider both of them here.
\subsection{Geometric vertex decomposability of defining ideal}\label{GVD}
Compared to the previous K-R's approach, a significant advantage of our new map is that we directly identify every type A quiver locus to a Kazhdan-Lusztig variety. Having explicit local equations, rather than just an open immersion, can lead to a great deal of research and can facilitate the handling of various problems. For example, it allows for checking the geometric vertex decomposability of the defining ideal of a type A quiver locus. Checking geometric vertex decomposability is highly dependent on the coordinates used, and having explicit equations has proven invaluable in the context of toric ideals of graphs, determinantal ladder ideals, and Hessenberg varieties.

Klein and Rajchgot have recently showed that geometric vertex decompositions correspond to elementary G-biliaisons, which in turn implies the glicci property. They showed that the defining ideal of equioriented type A quiver loci is geometric vertex decomposable. However the cases for arbitrary orientation are still open. Now we can complete this result for arbitrary orientation.

See also \cite{GVDandliaison} for the following definitions. Let $R=k[x_1,x_2,\cdots,x_t]$ be a polynomial ring and $y=x_j$ for some $1\leq j\leq t$. Define the initial $y$-form $\initial_y f$ of a polynomial $f\in R$ to be the sum of all terms of $f$ having the highest power of $y$. Given an ideal $I\subset R$, define $\initial_y I$ to be ideal generated by the initial $y$-forms of elements of $J$, i.e. $\initial_y I=<\initial_y f|f\in I>$. For an ideal $I$ let $\{y^{s_i}q_i+r_i|i\}$ be a Gr$\ddot{\mbox{o}}$bner basis of $I$ with respect to a $y$-compatible monomial order $<$ (i.e. $\initial_< f=\initial_< (\initial_y f),\forall f\in R$), where $y\not|q_i$ and $\initial_y (y^{s_i}q_i+r_i)=y^{s_i}q_i$ for all $i$. Independent with the choice of order $<$ and Gro$\ddot{\mbox{o}}$bner basis, we define $C_{y,I}=\left<q_i|i\right>$ and $N_{y,I}=\left<q_i|s_i=0\right>$. If $\initial_y I=C_{y,I}\cap (N_{y,I}+\left<y\right>)$ in $R$, this decomposition is called a geometric vertex decomposition of $I$ with respect to $y$.
\begin{defn}
 An ideal $I\subset R$ is geometrically vertex decomposable if $I$ is unmixed (i.e. $\dim R/I'=\dim R/I,\forall I'\in Ass(I)$ an associated prime) and
  \begin{itemize}
    \item[(1)] $I=R$ or $I$ is generated by indeterminates in $R$ or
    \item[(2)] for some variable $y=x_j$ of $R$, $\initial_y I=C_{y,I}\cap (N_{y,I}+\left<y\right>)$ is a geometric vertex decomposition and the contractions of $N_{y,I}$ and $C_{y,I}$ to $k[x_1,x_2,\cdots,\hat{y},\cdots,x_t]$ ($y$ removed) are geometrically vertex decomposable.
  \end{itemize}
\end{defn}

The defining ideal of a Kazhdan-Lusztig variety $\overline{B\gamma B/B}\cap B^-\theta B/B, \gamma,\theta\in W_G$ in affine space $B^-v B/B$ is called a Kazhdan-Lusztig ideal. Klein and Rajchgot showed that the defining ideal of an equioriented type A quiver locus is geometrically vertex decomposable through extend the ideal in the polynomial ring to a Kazhdan-Lusztig ideal in a larger polynomial ring, which is known to be geometrically vertex decomposable as a Stanley-Reisner ideal of subword complex. Now, we can upgrade their results not only for arbitrary orientations, but also with a direct proof.

\begin{prop}
  The defining ideal of type A quiver locus of arbitrary orientation is geometrically vertex decomposable, and consequencely glicci (in the Gorenstein liaison class of a complex intersection, see \cite{GVDandliaison}).
\end{prop}
\begin{proof}
  Since the defining ideal of any type A quiver locus is a Kazhdan-Lusztig ideal itself by Section \ref{G/BKL}, and the Kazhdan-Lusztig ideal is geometrically vertex decomposable \cite{GbasisforKL}, we immediately complete the proof.
\end{proof}
\subsection{Singularities of type A quiver locus}\label{singloc}
Let $T_x X$ be the tangent space of a variety $X$ at point $x$. The dimension $\dim T_x W\geq \dim X$ and $x$ is a singular (resp. smooth) point in $X$ if and only if $\dim T_xX\geq \dim X$ (resp. $\dim T_x X=\dim X$. All the singular points in $X$ form a closed subvariety $\sing X$of $X$, which is called the singular locus of $X$. If $X$ is a $H$-space of some algebraic group $H$, then the singular locus $\sing X$ is also $H$-stable, and there is $\sing X=\cup \overline{H.y}$ for several $\overline{H.y}$ that does not contain each other pairwise. Such $y$ (also, $H.y$) is called the maximal singularities or generic singularities of $X$. Now our problem is to discuss the dimensions of tangent spaces and the maximal singularities of type A quiver locus.

In previous research, the dimension of the tangent space $T_W \overline{\mathcal{O}_{\mathbf{r}}}$ at a point (a representation) $W$ is estimated by the extension of representations. Now we can provide an explicit formula.

The notations are as in Section \ref{KL}. Here we mainly work on the partial flag variety $G/P_Q$ as before, and many of the following results are valid for the general $(S',S)$-blocked situation (i.e. for general parabolic $P$ and arbitrary sequence of indexes $S'$) as well.
\begin{lem}\label{dataslemma}
For $\tau\leq w$ in $W^P$ there is there is isomorphism $$Y^v_w\times X_v^\circ\to X_w\cap vX^{id}_\circ,$$ where $id\in W_G$ is the identity and the isomorphism sends $(X_w^\circ\cap X^v_\circ)\times X_v^\circ$ onto $X_w^\circ\cap vX^{id}_\circ$. Here $vX^{id}$ is an affine open neighborhood of $e_v:=vP/P$ in $G/P$. The dimension of Kazhdan-Lusztig variety $Y^v_w$ is $$\dim Y^v_w= l(w)-l(v),$$ and $$\dim T_{e_\tau} X_w=\#\{1\leq i<j\leq N|s_{ij}\tau \leq w\}$$
  $$=\#\{1\leq i<j\leq N| \tau< s_{ij}.\tau \leq w\}+l(\tau),$$ where $s_{ij}\in W_G$ is the permutation transposing $i$ and $j$ and it acts on $W^P$ naturally (i.e. $s_{ij}.\tau$ is the shortest representative in $s_{ij}\tau W_P$, where $s_{ij}\tau$ is the product of permutations $s_{ij}$ and $\tau$).
\end{lem}

\begin{proof}
  See \cite[(1.3)]{SchubertPoincare}, \cite[Chap. 4]{introtoSMT}, and \cite[Chap. 13]{flagvarieties}, this lemma can be obtained easily by generalizing the results to partial flag varieties. See \cite[Lemma 4.3.3]{introtoSMT} for the dimensions of Schubert varieties. For the $(S'_Q,S_Q)$-blocked Z-type permutations $\tau$ and $w$ this lemma also can be proved by combining our Section \ref{G/BKL} with the results \cite[(1.3)]{SchubertPoincare} and \cite[Chap. 13]{flagvarieties} on full flags $G/B$, which is enough in our subsequent contents.
\end{proof}

Given any $(x,y)\in S'_Q\times S_Q$, we define the areas $$\mathcal{N}(x,y)=\{(N^i(x),y)|i>0\},\mathcal{E}(x,y)=\{(x,E^j(y))|j>0|\},$$ $$\mathcal{NE}(x,y)=\{(N^i(x),E^j(y))|i,j>0\}$$ and similarly $\mathcal{S}(x,y),\mathcal{W}(x,y),\mathcal{NW}(x,y),\mathcal{SW}(x,y),\mathcal{SE}(x,y)$. Their union is simply presented such as $(\mathcal{SE}\cup\mathcal{S})(x,y):=\mathcal{SE}(x,y)\cup\mathcal{S}(x,y)$.
$$\begin{blockarray}{llcr}
 \begin{block}{l[l|c|r]}
                 &   &    &  \\
                  &  \mathcal{NE}(x,y)  &  \mathcal{N}(x,y)  & \\\cline{2-3}
              x    &  \mathcal{E}(x,y)   & (x,y)   &  (\mathcal{NW}\cup\mathcal{W})(x,y)  \\\cline{2-4}
                 &   &     &   \\
                  &  \mathcal{SE}(x,y)   &    &   \\
                  \end{block}
 &     &  y &      \end{blockarray}$$
Because of the special structure of Zelevinsky permutations, the dimensions of the type A quiver loci and their tangent spaces can be computed block by block. Given any permutation $v\leq w\in W^P$ for both $v$ and $w$ being Z-type we have $$\dim Y^{v}_{w}=l(w)-l(v)$$
$$=\sum_{\tiny{\begin{array}{c}(x,y)\in S'_Q\times S_Q\\(x',y')\in (\mathcal{SW}\cup\mathcal{W})(x,y)\end{array}}}M(w)_{[x],[y]}M(w)_{[x'],[y']}-l(v).$$

By Proposition \ref{orbitimage}, we only need to focus on the tangent space of $Y^v_w$ at $e_\tau$ such that $\tau$ is a permutation of Z-type and $v\leq \tau\leq w$: $\mathcal{O}_\mathbf{r'}$ is singular in $\overline{\mathcal{O}_{\mathbf{r}}}$ if and only if $e_{w_Q(\mathbf{r'})}$ is singular in $X_{w_Q(\mathbf{r})}$. Consider $1\leq i<j\leq N$ and assume that the positions (of entries) $(i,\tau^{-1}(i))$ and $(j,\tau^{-1}(j))$ are contained in blocks $([x],[y])$ and $([x'],[y'])$, respectively, in a $(S'_Q,S_Q)$-blocked matrix. Then $i,j$ satisfy $\tau< s_{ij}.\tau$ if and only if $(x,y)\in\mathcal{NW}(x',y')$. And note that at this time $\rank (s_{ij}.\tau)_{[p]\times[q]}$ is given by
$$\left\{\begin{array}{ll}
    \rank \tau_{[p]\times[q]}+1, & \mbox{if $(p,q)\in (\mathcal{SE}\cup \mathcal{S})(x,y))\cap (\mathcal{NW}\cup\mathcal{W})(x',y')$} \\
    \rank \tau_{[p]\times[q]}, & \mbox{else.}
  \end{array}\right.$$
      \begin{tikzpicture}
        \draw [step=2] (0,0) grid (2,-2);
        \draw [dashed,step=2] (0,-2) grid (4,-4);
        \draw [step=2] (4,-2) grid (6,-4);
        \node at (0.5,-1) {$1$};
        \node at (5,-3.5) {$1$};
        \node at (-2.5,-1) [align=center]{the $i$-th row};
        \node at (0.5,-5) [align=center]{the $\tau^{-1}(i)$-th column};
        \draw (-1,-1) -- (0.2,-1);
        \draw (0.5,-4.5) -- (0.5,-1.3);
        \node at (-2.5,-3.5) [align=center]{the $i$-th row};
        \node at (5,-5) [align=center]{the $\tau^{-1}(i)$-th column};
        \draw (-1,-3.5) -- (4.8,-3.5);
        \draw (5,-4.5) -- (5,-3.8);
        \node at (3,1) [align=center]{the $([x],[y])$-block};
        \node at (5.5,-1.3) [align=center]{$(\mathcal{SE}\cup \mathcal{S})(x,y))\cap (\mathcal{NW}\cup\mathcal{W})(x',y')$};
        \node at (8,-3) [align=center]{the $([x'],[y'])$-block};
        \draw (2 ,0.7) parabola (1.5,-0.5);
        \draw (5.5 ,-2.5) parabola (7.5,-2.7);
        \draw (1.5,-2.7) parabola (3.5,-1.6);
        \draw (3,-2.7) parabola (3.5,-1.6);
    \end{tikzpicture}

We define the area of differences of southwest ranks for $\tau$ and $w$ by $$\mathcal{D}(\tau,w):=\{(p,q)\in S'_Q\times S_Q|\rank \tau_{[p]\times[q]}\neq \rank w_{[p]\times[q]}\}.$$Then $s_{ij}.\tau \notin \tau W_P, \tau< s_{ij}.\tau \leq w$ if and only if the corresponding $(x,y)$ and $(x',y')$ satisfy
  \begin{itemize}
    \item $(x,y)\in\mathcal{NW}(x',y')$, and
    \item $(\mathcal{SE}\cup \mathcal{S})(x,y))\cap (\mathcal{NW}\cup\mathcal{W})(x',y') \subset \mathcal{D}(\tau,w)$.
  \end{itemize}

For certain $\tau$ and $w$, we define $\mathcal{P}(\tau,w)$ to be the collection of the pairs $((x,y),(x',y'))\in (S'_Q\times S_Q)\times (S'_Q\times S_Q)$ as above. Then
$$\dim T_{e_\tau} X_w=\sum_{((x,y),(x',y'))\in \mathcal{P}(\tau,w)}M(\tau)_{[x],[y]}M(\tau)_{[x'],[y']}+l(\tau).$$
\begin{prop}
  For $\mathbf{r}$ a rank parameter of type A quiver representation in $rep_Q(\mathbf{d})$, we have
  \begin{equation}
  \dim \overline{\mathcal{O}_{\mathbf{r}}}=\sum_{\tiny{\begin{array}{c}(x,y)\in S'_Q\times S_Q\\(x',y')\in \mathcal{SW}(x,y)\cup\mathcal{W}(x,y)\end{array}}}M(w_Q(\mathbf{r}))_{[x],[y]}M(w_Q(\mathbf{r}))_{[x'],[y']}-l(v_Q),\end{equation} and for any representation $W\in\overline{\mathcal{O}_{\mathbf{r}}}$ with $W=\oplus m_{pq}I_{pq}\in \mathcal{O}_{\mathbf{r'}}$ there is
\begin{equation}\label{6.2}
\dim T_W \overline{\mathcal{O}_{\mathbf{r}}}-\dim \overline{\mathcal{O}_{\mathbf{r'}}}=\sum_{((x,y),(x',y'))\in \mathcal{P}(w_Q(\mathbf{r'}),w_Q(\mathbf{r}))}m_{yx}m_{y'x'}.
\end{equation} Here, \begin{equation}\label{6.3}
                       l(v_Q)=\sum_{s(\alpha)>s(\beta)}d_{s(\alpha)}d_{s(\beta)}+\sum_{t(\alpha)<t(\beta)}d_{t(\alpha)}d_{t(\beta)},
                     \end{equation}where $\alpha$ (resp. $\beta$) is taken over the left arrows $\alpha.$ (resp. right arrows $\beta.$) of $Q$.
\end{prop}
\begin{proof}
  To obtain equation (\ref{6.2}) we only need to note that for any $y=x+1$, neither $(S(x),y)$ nor $(x,W(y))$ belongs to $\mathcal{D}   (w_Q(\mathbf{r'}),w_Q(\mathbf{r}))$, because both of the southwest $[S(x)]\times[y]$ and $[x]\times[W(y)]$ matrices of the image of Zelevinsky map are of constant rank (see matrices (\ref{*}) and (\ref{**}), or Example \ref{Zex}). And equation (\ref{6.3}) can be checked by directly considering the intersection $\{p|x\prec p\mbox{ in $S'_Q$}\}\cap\{q|q\prec x\mbox{ in $S_Q$}\}$ for different $x\in Q_0$ (source/sink, general point in left/right path).
\end{proof}
\begin{coro}
  Given $W\in \mathcal{O}_{\mathbf{r'}}\subset \overline{\mathcal{O}_{\mathbf{r}}}$, the dimension of $T_W \overline{\mathcal{O}_{\mathbf{r}}}$ only depends on $\mathbf{r'}$ and
  $$D(\mathbf{r'},\mathbf{r}):=\{[a,b]\subset Q|r_{[a,b]}\neq r'_{[a,b]}\}.$$ Moreover, if $\mathbf{r}$ is not the maximal element in $$\{\mbox{rank parameter }\mathbf{r''}|D(\mathbf{r'},\mathbf{r''})=D(\mathbf{r'},\mathbf{r})\}$$ under the degeneration order, i.e. $\exists \mathbf{r''},D(\mathbf{r'},\mathbf{r''})=D(\mathbf{r'},\mathbf{r})$ and $\overline{\mathcal{O}_{\mathbf{r}}}\subsetneqq \overline{\mathcal{O}_{\mathbf{r''}}}$, then $W$ is a singular point of $\overline{\mathcal{O}_{\mathbf{r}}}$.
\end{coro}
\begin{proof}
  If there exists $\mathbf{r''},D(\mathbf{r'},\mathbf{r''})=D(\mathbf{r'},\mathbf{r})$ and $\overline{\mathcal{O}_{\mathbf{r}}}\subsetneqq \overline{\mathcal{O}_{\mathbf{r''}}}$, then $$\dim T_W\overline{\mathcal{O}_{\mathbf{r}}}=\dim T_W\overline{\mathcal{O}_{\mathbf{r''}}}\geq \dim \overline{\mathcal{O}_{\mathbf{r''}}}> \dim \overline{\mathcal{O}_{\mathbf{r}}},$$ which indicates that $W$ is singular point of $\overline{\mathcal{O}_{\mathbf{r}}}$.
\end{proof}

Next, we consider the smoothness criterion for points in the type A quiver locus. Let $v\leq\tau\leq w$ be the permutations of Z-type as before. Let
$$\codim X(\tau,w):=\dim X_w-\dim X_\tau=\dim Y^v_w-\dim Y^v_\tau=l(w)-l(v),$$
and $$\codim T(\tau,w):=\dim T_{e_\tau}X_w-\dim T_{e_\tau}X_\tau=\dim T_{e_\tau}Y^v_w-\dim T_{e_\tau}$$ $$=\sum_{((x,y),(x',y'))\in \mathcal{P}(\tau,w)}M(\tau)_{[x],[y]}M(\tau)_{[x'],[y']}.$$ Since $\dim T_{e_\tau}X_\tau=\dim X_\tau$, then we have $\codim T(\tau,w)\geq \codim X(\tau,w)$ and they take equal if and only if $e_\tau$ is smooth in $X_w$. Let $\delta(\tau,w)=M(\tau)-M(w)$, where its rows and columns are also labeled by $S'_Q$ and $S_Q$ like $M(\tau),M(w)$. Compute\small{
\begin{equation}\label{6.4}\begin{array}{rl}&\codim X(\tau,w)\\=&\sum\limits_{\tiny{\begin{array}{c}(x,y)\in S'_Q\times S_Q\\(x',y')\in \mathcal{SW\cup W}(x,y)\end{array}}}(M(w)_{[x],[y]}M(w)_{[x'],[y']}-M(\tau)_{[x],[y]}M(\tau)_{[x'],[y']})\\
=&\sum\limits_{\tiny{\begin{array}{c}(x,y)\in S'_Q\times S_Q\\(x',y')\in \mathcal{SW\cup W}(x,y)\end{array}}}(M(\tau)_{[x],[y]}-\delta(\tau,w)_{[x],[y]})(M(\tau)_{[x'],[y']}-\delta(\tau,w)_{[x'],[y']})\\
&-\sum\limits_{\tiny{\begin{array}{c}(x,y)\in S'_Q\times S_Q\\(x',y')\in \mathcal{SW\cup W}(x,y)\end{array}}}M(\tau)_{[x],[y]}M(\tau)_{[x'],[y']}\\
=&\sum\limits_{(x,y)\in S'_Q\times S_Q}(M(\tau)_{[x],[y]}-\frac{1}{2}\delta(\tau,w)_{[x],[y]})\sum\limits_{(x',y')\in (\mathcal{SE\cup NW})(x,y)}\delta(\tau,w)_{[x'],[y']},\end{array}\end{equation}}
the last equality holds because the sum of each row or column of $\delta(\tau,w)$ is zero. And \small{
\begin{equation}\label{6.5}\begin{array}{rl}&\codim T(\tau,w)\\=&\sum\limits_{((x,y),(x',y'))\in \mathcal{P}(\tau,w)}M(\tau)_{[x],[y]}M(\tau)_{[x'],[y']}\\
=&\sum\limits_{((x,y),(x',y'))\in \mathcal{P}(\tau,w)}(M(w)_{[x],[y]}+\delta(\tau,w)_{[x],[y]})(M(w)_{[x'],[y']}+\delta(\tau,w)_{[x'],[y']})\\
=&\sum\limits_{((x,y),(x',y'))\in \mathcal{P}(\tau,w)}M(w)_{[x],[y]}M(w)_{[x'],[y']}+\sum\limits_{(x,y)\in S'_Q\times S_Q} M(w)_{[x],[y]}\cdot\\
&(\sum\limits_{((x,y),(x',y'))\in \mathcal{P}(\tau,w)}\delta(\tau,w)_{[x'],[y']}+\sum\limits_{((x',y'),(x,y))\in \mathcal{P}(\tau,w)}\delta(\tau,w)_{[x'],[y']})\\
&+\sum\limits_{((x,y),(x',y'))\in \mathcal{P}(\tau,w)}\delta(\tau,w)_{[x],[y]}\delta(\tau,w)_{[x'],[y']}\\
=&\sum\limits_{((x,y),(x',y'))\in \mathcal{P}(\tau,w)}M(w)_{[x],[y]}M(w)_{[x'],[y']}+\sum\limits_{(x,y)\in S'_Q\times S_Q}(M(w)_{[x],[y]}+\frac{1}{2}\delta(\tau,w)_{[x],[y]})\cdot\\
&(\sum\limits_{((x,y),(x',y'))\in \mathcal{P}(\tau,w)}\delta(\tau,w)_{[x'],[y']}+\sum\limits_{((x',y'),(x,y))\in \mathcal{P}(\tau,w)}\delta(\tau,w)_{[x'],[y']}).\end{array}\end{equation}}
Note that in equations (\ref{6.4}) and (\ref{6.5}), $$M(\tau)_{[x],[y]}-\frac{1}{2}\delta(\tau,w)_{[x],[y]}=M(w)_{[x],[y]}+\frac{1}{2}\delta(\tau,w)_{[x],[y]}=\frac{1}{2}(M(w)_{[x],[y]}+M(\tau)_{[x],[y]}).$$

\begin{lem}\label{lemsp2}
  Given $\tau,w$ as above, for any $(x,y)\in S'_Q\times S_Q$ we have inequalities
  \begin{equation}\label{inequ}\begin{array}{c}
    \sum\limits_{(x',y')\in \mathcal{SE}(x,y)}\delta(\tau,w)_{[x'],[y']}\leq \sum\limits_{((x,y),(x',y'))\in \mathcal{P}(\tau,w)}\delta(\tau,w)_{[x'],[y']},\\
  \sum\limits_{(x',y')\in \mathcal{NW}(x,y)}\delta(\tau,w)_{[x'],[y']}\leq \sum\limits_{((x',y'),(x,y))\in \mathcal{P}(\tau,w)}\delta(\tau,w)_{[x'],[y']}.\end{array}
  \end{equation}
  And there exists at least one inequality (among all $(x,y)$) is strictly `$<$' if and only if (the diagram of) $\mathcal{D}(\tau,w)$ contains a pattern    $$     \begin{tikzpicture}
\draw (0,0)[step=0.6] grid (1.8,-1.20001);
\draw (0.6,-0.6) [step=0.6] grid (2.4,-1.8001);
\node at (2.1,-0.3){$\times$};
\node at (0.3,-1.5){$\times$};
\node at (4.8,-0.9)[align=left]{Instructions\\\emph{\fcolorbox{black}{white}{\rule{0pt}{5pt}\rule{5pt}{0pt}}: in $\mathcal{D}(\tau,w)$}\\\emph{$\times$: not in $\mathcal{D}(\tau,w)$}\\\emph{Others: undetermined}};
\node at (-1,-0.3){$x_2$};
\node at (2.1,-2.5){$y_2$};
\node at (-1,-1.5){$x_1$};
\node at (0.3,-2.5){$y_1$};
    \end{tikzpicture}$$
  i.e. there is $(x_1,y_1)$ and $(x_2,y_2)$ such that
  \begin{itemize}
    \item $(x_2,y_2)\in \mathcal{NE}(x_1,y_1)$ and $(x_1,y_1),(x_2,y_2)\notin \mathcal{D}(\tau,w)$;
    \item $(\mathcal{N}\cup \mathcal{NE}\cup\mathcal{E})(x_1,y_1)\cap (\mathcal{S}\cup \mathcal{SW}\cup \mathcal{W})(x_2,y_2)\in \mathcal{D}(\tau,w)$.
  \end{itemize}
\end{lem}
\begin{proof}
  We only prove the first inequality in \ref{inequ}. Given a fixed position $(x,y)$, we draw the maximal Young diagram $\mathcal{Y}$ which is contained in $\mathcal{D}(\tau,w)$ and extends southeast based on position $(S(x),y)$. Then the set $\{(x',y')|((x,y),(x',y'))\in \mathcal{P}(\tau,w)\}$ is exactly $\{(x',y')|(x',W(y'))\in\mathcal{Y}\}$.
$$\begin{tikzpicture}
\draw [dashed,step=0.6] (0,0) grid (0.6,0.6001);
\node at (-0.6,0.3) {$x$};
\node at (-0.6,-0.3) {$S(x)$};
\node at (0.3,1.2){$y$};
\draw [line width =0.6pt,step=0.6](0,0) grid (1.8,-1.8001);
\draw [line width =0.6pt,step=0.6](1.8,0) grid (2.4,-0.6001);
\draw [line width =0.6pt,step=0.6](0,-1.8) grid (0.6,-2.4001);
\node at (2.7,-0.3) {$\times$};
\node at (2.1,-0.9) {$\times$};
\node at (0.9,-2.1) {$\times$};
\node at (0.3,-2.7) {$\times$};
\draw [dashed](0,0) -- (4,0);
\draw [dashed](0,0) -- (0,-4);
\node at (8,-2) [align=left]{\emph{Instructions}\\\fcolorbox{black}{white}{\rule{0pt}{5pt}\rule{5pt}{0pt}}: the Young diagram $\mathcal{Y}$, in $\mathcal{D}(\tau,w)$\\$\times$: not in $\mathcal{D}(\tau,w)$\\\fcolorbox{gray!40}{gray!40}{\rule{0pt}{5pt}\rule{5pt}{0pt}}: in $\{(x',y')|(x',W(y'))\in\mathcal{Y}\}$\\Others: undetermined};
\draw [dashed,step=0.6](0.6,0) grid (2.4,-1.8001);
\draw [dashed,step=0.6](2.4,0) grid (3,-0.6001);
\draw [dashed,step=0.6](0.6,-1.8) grid (1.2,-2.4001);
\draw [fill=gray,opacity=0.2](0.6,-1.8) rectangle (1.2,-2.4);
\draw [fill=gray,opacity=0.2](0.6,0) rectangle (2.4,-1.8);
\draw [fill=gray,opacity=0.2](2.4,0) rectangle (3,-0.6);
\end{tikzpicture}$$

Note that $\delta(\tau,w)$ can be obtained by solely considering the difference between $\rank \tau_{[p]\times[q]}$ and $\rank w_{[p]\times[q]}$ for every single $(p,q)\in S'_Q\times S_Q$: that is, such a sole difference at $(p,q)$ influences only four positions $(p,q),(N(p),q),(p,E(q)),$ and $(N(p),E(q))$ in $\delta(\tau,w)$, and totally $\delta(\tau,w)$ is the independent composition of such influences. For example, the following diagram shows this influence on the entries of $\delta(\tau,w)$ when $\rank \tau_{[p]\times[q]}$ and $\rank w_{[p]\times[q]}$ differ in $1$. $$\begin{tikzpicture}
\draw [dashed] (0,0) grid (2,-2);
\draw  (0,-1) grid (1,-2);
\node at (-1,-0.5) {$N(p)$};
\node at (-1,-1.5) {$p$};
\node at (0.5,-2.6) {$q$};
\node at (1.5,-2.6) {$E(q)$};
\node at (0.5,-0.5) {$+1$};
\node at (1.5,-1.5) {$+1$};
\node at (0.5,-1.5) {$-1$};
\node at (1.5,-0.5) {$-1$};
\end{tikzpicture}$$

It is not hard to see that any such sole (admissible) difference of southwest ranks will not make the value $$ \sum\limits_{((x,y),(x',y'))\in \mathcal{P}(\tau,w)}\delta(\tau,w)_{[x'],[y']}-\sum\limits_{(x',y')\in \mathcal{SE}(x,y)}\delta(\tau,w)_{[x'],[y']} $$ will decrease. For example, in the following three diagrams, the difference of ranks at positions in the first diagram does not affect the above value; the positions in the second diagram could make the above value decrease, but these positions must be not in $\mathcal{D}(\tau,w)$; and the positions in the third diagram make the above value increase. $$\begin{tikzpicture}
\draw [dashed,step=0.6] (0,0) grid (0.6,0.6001);
\node at (-0.6,0.3) {$x$};
\node at (-0.6,-0.3) {$S(x)$};
\node at (0.3,1.2){$y$};
\node at (2.4,-3) {$\mathcal{SE}(x,y)$};

\draw [line width =0.6pt](1.2,0) rectangle (1.8,-0.6);
\node at (1.5,-0.3){\small{$-1$}};
\node at (2.1,-0.3){\small{$+1$}};
\node at (1.5,0.3){\small{$+1$}};
\node at (2.1,0.3){\small{$-1$}};

\draw [line width =0.6pt](1.2,-1.8) rectangle (1.8,-2.4);
\node at (1.5,-2.1){\small{$-1$}};
\node at (2.1,-2.1){\small{$+1$}};
\node at (1.5,-1.5){\small{$+1$}};
\node at (2.1,-1.5){\small{$-1$}};

\draw [dashed](0.6,0) -- (4,0);
\draw [dashed](0.6,0) -- (0.6,-4);
\draw [dashed,step=0.6](0.6,0) grid (2.4,-1.8001);
\draw [dashed,step=0.6](2.4,0) grid (3,-0.6001);
\draw [dashed,step=0.6](0.6,-1.8) grid (1.2,-2.4001);
\draw [fill=gray,opacity=0.2](0.6,-1.8) rectangle (1.2,-2.4);
\draw [fill=gray,opacity=0.2](0.6,0) rectangle (2.4,-1.8);
\draw [fill=gray,opacity=0.2](2.4,0) rectangle (3,-0.6);
\end{tikzpicture}\
\begin{tikzpicture}
\draw [dashed,step=0.6] (0,0) grid (0.6,0.6001);

\draw [line width =0.6pt](0,-2.4) rectangle (0.6,-3);
\node at (0.3,-2.7){\small{$-1$}};
\node at (0.9,-2.7){\small{$+1$}};
\node at (0.3,-2.1){\small{$+1$}};
\node at (0.9,-2.1){\small{$-1$}};

\draw [line width =0.6pt](1.8,-0.6) rectangle (2.4,-1.2);
\node at (2.1,-0.9){\small{$-1$}};
\node at (2.7,-0.9){\small{$+1$}};
\node at (2.1,-0.3){\small{$+1$}};
\node at (2.7,-0.3){\small{$-1$}};

\draw [dashed](0.6,0) -- (4,0);
\draw [dashed](0.6,0) -- (0.6,-4);
\draw [dashed,step=0.6](0.6,0) grid (2.4,-1.8001);
\draw [dashed,step=0.6](2.4,0) grid (3,-0.6001);
\draw [dashed,step=0.6](0.6,-1.8) grid (1.2,-2.4001);
\draw [fill=gray,opacity=0.2](0.6,-1.8) rectangle (1.2,-2.4);
\draw [fill=gray,opacity=0.2](0.6,0) rectangle (2.4,-1.8);
\draw [fill=gray,opacity=0.2](2.4,0) rectangle (3,-0.6);
\end{tikzpicture}\
\begin{tikzpicture}
\draw [dashed,step=0.6] (0,0) grid (0.6,0.6001);

\draw [line width =0.6pt,step=0.6](0,0) grid (1.8,-1.8001);
\draw [line width =0.6pt,step=0.6](1.8,0) grid (2.4,-0.6001);
\draw [line width =0.6pt,step=0.6](0,-1.8) grid (0.6,-2.4001);

\draw [line width =0.6pt](0.6,-2.4) rectangle (1.2,-3);
\node at (0.9,-2.7){\small{$-1$}};
\node at (1.5,-2.7){\small{$+1$}};
\node at (0.9,-2.1){\small{$+1$}};
\node at (1.5,-2.1){\small{$-1$}};

\draw [line width =0.6pt](2.4,-0.6) rectangle (3,-1.2);
\node at (2.7,-0.9){\small{$-1$}};
\node at (3.3,-0.9){\small{$+1$}};
\node at (2.7,-0.3){\small{$+1$}};
\node at (3.3,-0.3){\small{$-1$}};

\draw [dotted](-0.1,-1.7) rectangle (1.3,-3.1);
\draw [dotted](2.3,0.1) rectangle (3.7,-1.3);

\draw [dashed](0.6,0) -- (4,0);
\draw [dashed](0.6,0) -- (0.6,-4);
\draw [dashed,step=0.6](0.6,0) grid (2.4,-1.8001);
\draw [dashed,step=0.6](2.4,0) grid (3,-0.6001);
\draw [dashed,step=0.6](0.6,-1.8) grid (1.2,-2.4001);
\draw [fill=gray,opacity=0.2](0.6,-1.8) rectangle (1.2,-2.4);
\draw [fill=gray,opacity=0.2](0.6,0) rectangle (2.4,-1.8);
\draw [fill=gray,opacity=0.2](2.4,0) rectangle (3,-0.6);
\end{tikzpicture}$$The positions which make the value strictly increase obviously leads to an appearance of the special pattern in $\mathcal{D}(\tau,w)$, as shown in the dotted area in the third diagram.
\end{proof}
\begin{prop}[Weak Criterion]\label{weakcriterion}
  Given $\tau,w$ as above, then $e_\tau$ is a singular point in $X_w$ if and only if
  \begin{itemize}
    \item[(SP1)]$\exists ((x_1,y_1),(x_2,y_2))\in \mathcal{P}(\tau,w)$ such that $M(w)_{[x_1],[y_1]}>0,M(w)_{[x_2],[y_2]}>0$; or
    \item[(SP2)]$\exists (x,y),(x_1,y_1),(x_2,y_2)$ such that
  \begin{itemize}
    \item $M(w)_{[x],[y]}+M(\tau)_{[x],[y]}>0$;
    \item $(x_2,y_2)\in \mathcal{NE}(x_1,y_1)$ and $(x_1,y_1),(x_2,y_2)\notin \mathcal{D}(\tau,w)$;
    \item $(x_1,y_1),(x_2,y_2)\in (\mathcal{NW}\cup\mathcal{W})(x,y)$ and $\mathcal{SE}(x_2,y_1)\cap(\mathcal{NW}\cup\mathcal{W})(x,y)\subset \mathcal{D}(\tau,w)$; or $(x_1,y_1),(x_2,y_2)\in (\mathcal{SE}\cup\mathcal{S})(x,y)$ and $\mathcal{NW}(x_1,y_y)\cap (\mathcal{SE}\cup\mathcal{S})(x,y)\subset \mathcal{D}(\tau,w)$;
    \item $(\mathcal{N}\cup \mathcal{NE}\cup\mathcal{E})(x_1,y_1)\cap (\mathcal{S}\cup \mathcal{SW}\cup \mathcal{W})(x_2,y_2)\in \mathcal{D}(\tau,w)$.
  \end{itemize}
  \end{itemize}That is, the diagram of $\mathcal{D}(\tau,w)$ contains one of the following patterns, where `$+$' (resp. `$\triangle$') indicates that the entry of $M(w)$ (resp. $M(w)+M(\tau)$) is positive at that position, and other instructions are as before:  $$\begin{tikzpicture}
\draw [step=0.6] (0,0) grid (1.8,-1.2001);
\draw [dashed] (0,0.6) rectangle (0.6,0);
\draw [dashed] (1.8,-0.6) rectangle (2.4,-1.2);
\node at (0.3,0.3) {$+$};
\node at (2.1,-0.9) {$+$};
\node at (-0.6,0.3) {$x_1$};
\node at (-0.6,-0.9) {$x_2$};
\node at (0.3,-1.8) {$y_1$};
\node at (2.1,-1.8) {$y_2$};
\node at (0.6,-2.5) {(SP1)};
\end{tikzpicture}\hspace{1cm}
\begin{tikzpicture}
\draw [dashed] (0,0.6) rectangle (0.6,1.2);
\draw [step=0.6] (0,0.6) grid (2.4,-0.6001);
\draw [step=0.6] (1.2,-0.6) grid (3,-1.2001);
\node at (-0.6,0.9) {$x$};
\node at (0.3,-1.8) {$y$};
\node at (0.3,0.9) {$\triangle$};
\node at (-0.6,-0.3) {$x_2$};
\node at (-0.6,-0.9) {$x_1$};
\node at (0.9,-1.8) {$y_1$};
\node at (2.7,-1.8) {$y_2$};
\node at (1,-2.5) {(SP2)};
\node at (0.9,-0.9) {$\times$};
\node at (2.7,-0.3) {$\times$};\end{tikzpicture}\hspace{1cm}
\begin{tikzpicture}
\draw [step=0.6] (0,0) grid (2.4,-0.6);
\draw [step=0.6] (0.6,-0.6) grid (2.4,-1.2001);
\draw [step=0.6] (0,0) grid (0.6,0.6);
\draw [dashed] (2.4,-1.2) rectangle (3,-0.6);
\node at (2.7,-0.9) {$\triangle$};
\node at (2.7,-1.8) {$y$};
\node at (0.9,0.3) {$\times$};
\node at (0.3,-0.9) {$\times$};
\node at (1,-2.5) {(SP2)};
\node at (-0.6,-0.9) {$x,x_1$};
\node at (-0.6,0.3) {$x_2$};
\node at (0.3,-1.8) {$y_1$};
\node at (0.9,-1.8) {$y_2$};

\end{tikzpicture}$$
\end{prop}

Actually, we can obtain the stronger criterion as follows. However, due to the space limitations and consistency in writing logic, we will only present it as a claim here without providing proof. More detailed results and a systematic study will be presented in another article in the future.

\begin{claim}
  Given $\tau,w$ as above, then $e_\tau$ is a singular point in $X_w$ if and only if
  \begin{itemize}
    \item[(SP1)]$\exists ((x_1,y_1),(x_2,y_2))\in \mathcal{P}(\tau,w)$ such that $M(w)_{[x_1],[y_1]}>0,M(w)_{[x_2],[y_2]}>0$; or
    \item[(SP2')]$\exists (x_1,y_1),(x_2,y_2)$ such that
  \begin{itemize}
    \item $(x_2,y_2)\in \mathcal{NE}(x_1,y_1)$ and $(x_1,y_1),(x_2,y_2)\notin \mathcal{D}(\tau,w)$;
    \item $(\mathcal{N}\cup \mathcal{NE}\cup\mathcal{E})(x_1,y_1)\cap (\mathcal{S}\cup \mathcal{SW}\cup \mathcal{W})(x_2,y_2)\in \mathcal{D}(\tau,w)$.
  \end{itemize}
  \end{itemize}That is, the diagram of $\mathcal{D}(\tau,w)$ contains one of the following patterns:
  $$\begin{tikzpicture}
\draw [step=0.6] (0,0) grid (1.8,-1.2001);
\draw [dashed] (0,0.6) rectangle (0.6,0);
\draw [dashed] (1.8,-0.6) rectangle (2.4,-1.2);
\node at (0.3,0.3) {$+$};
\node at (2.1,-0.9) {$+$};
\node at (-0.6,0.3) {$x_1$};
\node at (-0.6,-0.9) {$x_2$};
\node at (0.3,-1.8) {$y_1$};
\node at (2.1,-1.8) {$y_2$};
\node at (0.6,-2.5) {(SP1)};
\end{tikzpicture}\hspace{1cm}\begin{tikzpicture}
\draw (0,0.6)[step=0.6] grid (1.8,-0.60001);
\draw (0.6,0) [step=0.6] grid (2.4,-1.2001);
\node at (2.1,0.3){$\times$};
\node at (0.3,-0.9){$\times$};
\node at (-0.6,0.3){$x_2$};
\node at (2.1,-1.8){$y_2$};
\node at (-0.6,-0.9){$x_1$};
\node at (0.3,-1.8){$y_1$};
\node at (0.6,-2.5) {(SP2')};
    \end{tikzpicture}$$
\end{claim}
\begin{examp}
  Consider $$\begin{tikzcd}
Q= 1 & 2 \arrow[l, "\alpha_1"] & 3 \arrow[l, "\alpha_2"] \arrow[r, "\beta_1"'] & 4 & 5 \arrow[l, "\alpha_3"]
\end{tikzcd}.$$ Let $V=3I_{14}\oplus 2I_{15}\oplus I_{55}\in rep_Q(\mathbf{d})$ with $\mathbf{d}=(3,3,4,3,5)$ and rank parameter $\mathbf{r}$. Choose $W=5I_{14}\oplus 3I_{55}\in \overline{\mathcal{O}_V}$ with rank parameter $\mathbf{r'}$. then the multiplicity matrix of $w_Q(\mathbf{r'})$ The multiplicity matrices of $w_Q(\mathbf{r})$ and $w_Q(\mathbf{r'})$ are respectively
$$M(w_Q(\mathbf{r}))=\begin{blockarray}{lccccc}
 \begin{block}{l[ccccc]}
                 3 &  0 & 0 & 0 & 0 & 4  \\
                5  &  1 & 0 & 0 & 2 & 0   \\\cline{2-4}
                 4 &  \BAmulticolumn{1}{|c|}{2} & \BAmulticolumn{1}{|c|}{0} & \BAmulticolumn{1}{|c|}{0} & 3 & 0   \\\cline{2-4}
                 2 &  0 & 3 & 0 & 0 & 0  \\
                 1 &   0& 3 & 0 & 0 & 0  \\
                  \end{block}
 &5 &  3 &  2 &  1 &  4 \end{blockarray},\ \ \ M(w_Q(\mathbf{r'}))=\begin{blockarray}{lccccc}
 \begin{block}{l[ccccc]}
                 3 &  0 & 0 & 0 & 0 & 4  \\
                5  &  3 & 0 & 0 & 0 & 0   \\
                 4 &  0 & 0 & 0 & 5 & 0    \\
                 2 &  0 & 3 & 0 & 0 & 0   \\
                 1 &   0& 3 & 0 & 0 & 0  \\
                  \end{block}
 &5 &  3 &  2 &  1 &  4 \end{blockarray}.$$ And $\mathcal{D}(w_Q(\mathbf{r'}),w_Q(\mathbf{r}))=\{(4,5),(4,3),(4,2)\}$ which is also marked out by boxes in $M(w_Q(\mathbf{r}))$. Then the pair $((5,5),(4,1))\in \mathcal{P}(w_Q(\mathbf{r'}),w_Q(\mathbf{r}))$, and $$M(w)_{[5],[5]}>0,M(w)_{[4],[1]}>0,$$ which indicates that $e_{w_Q(\mathbf{r'})}$ is a singular point in $X_{w_Q(\mathbf{r})}$ by Proposition \ref{weakcriterion} (SP1). Hence, $W$ is a singular point in $\overline{\mathcal{O}_V}$.
\end{examp}
\begin{examp}[Variety of complexes]The variety of complexes is a special type A quiver locus and it is the only quiver locus with known singular loci for a long time in the past. The quiver loci of the variety of complexes is determined in \cite{singcom1} and \cite{singcom2}, also through the Zelevinsky map for equioriented type A quiver. Here, as an example, we will see how the discussion in this section can help solve this problem quickly. Let $Q$ be the quiver $$\begin{tikzcd}
1 \arrow[r, "\beta_1"'] & 2 \arrow[r, "\beta_2"'] & 3 \arrow[r,"\beta_{3}"'] & \cdots \arrow[r, "\beta_{n-1}"'] & n
\end{tikzcd}$$ and consider the quiver locus $\overline{\mathcal{O}_{\mathbf{r}}}$ with $r_{[a,b]}=0$ for $b>a+1$ and $r_{[a,a+1]}=r_a$ for $1\leq a\leq n-1$, where $r_a$ is a family of numbers with $r_{a-1}+r_{a}\leq d_{a}$ for all $1\leq a\leq n-1$ (put $r_0=0$ and without loss of generality we assume $r_a\neq 0$ for $1\leq a\leq n-1$). The image of $V=(B_\cdot)=\oplus m_{a,a+1}I_{a,a+1}\in \overline{\mathcal{O}_{\mathbf{r}}}$ under $\zeta_Q$ is
$$\zeta_Q(V)=\begin{blockarray}{lccccc}\begin{block}{l[ccccc]}
                 1 &  I & 0 &0 &  & 0  \\
                2  &  B_1 & I & 0 &  & 0   \\
                 3 &  0 & B_2 & I & & 0    \\
                 \vdots &   &  & \ddots & \ddots &    \\
                 n &   0& 0 & \cdots & B_{n-1} & I  \\
                  \end{block}
 &1 &  2 &  3 &  \cdots &  n \end{blockarray},$$And the multiplicity matrix is $$\zeta_Q(V)=\begin{blockarray}{lccccc}\begin{block}{l[ccccc]}
                 1 &  m_{11} & m_{12} &0 &  & 0  \\
                2  &  m_{12} & m_{22} & m_{23} &  & 0   \\
                 3 &  0 & m_{23} & m_{33} & & 0    \\
                 \vdots &   &  & \ddots & \ddots &    \\
                 n &   0& 0 & \cdots & m_{n-1,n} & m_{nn}  \\
                  \end{block}
 &1 &  2 &  3 &  \cdots &  n \end{blockarray},$$

Let us try to find some $\mathbf{r'}\leq \mathbf{r}$ with permutations $w_Q(\mathbf{r'})\leq w_Q(\mathbf{r})$ satisfying the statement (SP1) or (SP2) in Proposition \ref{weakcriterion}. Note that $m_{ii}>0$ if and only if $r_{i-1}+r_{i}<d_{i}$. Therefore, for any $1\leq i\leq n-1$ satisfies $r_{i-1}+r_{i}<d_i$ and $r_i+r_{i+1}<d_{i+1}$, there must be $$M(w_Q(\mathbf{r}))_{[i],[i]}>0, M(w_Q(\mathbf{r}))_{[i+1],[i+1]}>0.$$ At this time the remaining conditions in (SP1) is $$((i,i),(i+1,i+1))\in \mathcal{P}(w_Q(\mathbf{r^{(i)}}),w_Q(\mathbf{r})),$$ and this can be ensured once there is $$r'_{[i,i+1]}<r_{[i,i+1]}.$$

So, we denote $\Omega=\{1\leq i\leq n-1|r_{i-1}+r_{i}<d_i,r_i+r_{i+1}<d_{i+1}\}$, and for any $1\leq i\leq n-1$ define
$$r^{(i)}_{[a,b]}=\left\{\begin{array}{ll} r_{[a,a+1]}-1,& a=i\\
r_{[a,b]},&\mbox{else}\end{array}\right.$$Then $\mathbf{r^{(i)}}$ is a rank parameter (Corollary \ref{coroMwr}). And by Proposition \ref{weakcriterion} (SP1), $\mathcal{O}_{\mathbf{r^{(i)}}}$ is singular in $\overline{\mathcal{O}_{\mathbf{r}}}$ if $i\in \Omega$. Obviously, such singularity must be maximal singularity of $\overline{\mathcal{O}_{\mathbf{r}}}$.

Next, for any $1\leq i< n-1$ we define $$r^{(i,i+1)}_{[a,b]}=\left\{\begin{array}{ll} r_{[a,a+1]}-1,& a=i\mbox{ or }i+1\\
r_{[a,b]},&\mbox{else}\end{array}\right.$$then $\mathbf{r^{(i,i+1)}}$ is a rank parameter as well. Note that $$M(w_Q(\mathbf{r^{(i,i+1)}}))_{[i],[i]}=M(w_Q(\mathbf{r}))_{[i],[i]}+1>0,$$and for the position $(i,i)$ the first inequality in (\ref{inequ}) is strictly `$<$', i.e. $(i,i),(i,i+2),(i+1,i+1)$ play the roles of $(x,y),(x_1,y_1),(x_2,y_2)$ in Proposition \ref{weakcriterion} (SP2), we obtain that $\mathcal{O}_{\mathbf{r^{(i,i+1)}}}$ is singular in $\overline{\mathcal{O}_{\mathbf{r}}}$ for all $1\leq i< n-1$. Note that $\mathcal{O}_{\mathbf{r^{(i,i+1)}}}\subset \overline{\mathcal{O}_{\mathbf{r^{(i)}}}}$ if $i\in \Omega$ or $i+1\in \Omega$, so we only need to consider $\mathcal{O}_{\mathbf{r^{(i,i+1)}}}$ for $i,i+1\notin \Omega$ as the candidates for maximal singularities of $\overline{\mathcal{O}_{\mathbf{r}}}$.

Actually, now we have found all the maximal singularities of $\overline{\mathcal{O}_{\mathbf{r}}}$, they are $\mathcal{O}_{\mathbf{r^{(i)}}}$ for $i\in \Omega$ and $\mathcal{O}_{\mathbf{r^{(i,i+1)}}}$ for $i,i+1\notin \Omega$. For any $\mathbf{r'}\leq \mathbf{r}$, if $\mathbf{r'}_{[i,i+1]}<\mathbf{r}_{[i,i+1]}$ for some $i\in \Omega$ then $\mathcal{O}_{\mathbf{r'}}\subset \overline{\mathcal{O}_{\mathbf{r^{(i)}}}}$; if $\mathbf{r'}_{[i,i+1]}<\mathbf{r}_{[i,i+1]}$ and $\mathbf{r'}_{[i+1,i+2]}<\mathbf{r}_{[i+1,i+2]}$ for some $1\leq i <n-1$ then $\mathcal{O}_{\mathbf{r'}}\subset \overline{\mathcal{O}_{\mathbf{r^{(i,i+1)}}}}$; in the rest cases, $\{1\leq i\leq n-1|r'_{[i,i+1]}<r_{[i,i+1]}\}$ intersects $\Omega$ in empty and contains no consecutive integers, by Proposition \ref{weakcriterion}, this $\mathcal{O}_{\mathbf{r'}}$ is smooth in $\overline{\mathcal{O}_{\mathbf{r}}}$. Therefore, all the maximal singularities are determined.
\end{examp}

In order to determine the singular loci of Type A quiver loci, we can also combine the results of Section \ref{G/BKL} with the conclusions in \cite{maxsingSL/B} on $G/B$ and search for the bad patterns of type (4231) or (3412) in Zelevinsky permutations. However, as mentioned earlier, the direct and rough application of these results may compromise some of the valuable data we prefer to obtain.  Therefore, it is more appropriate to conduct specialized research and summarization on the specific form of Zelevinsky permutations. We will present this content in a new article. In fact, the patterns (SP1) and (SP2) proposed here are respectively associated with the patterns of type (4231) and (3412) in \cite{maxsingSL/B}. The computational complexity of finding the singular loci of type A quiver loci will sharply increase with the growth of quiver vertices and the dimensions in $\mathbf{d}$. For instance, it is readily apparent from the method of searching for patterns (4231) and (3412) in permutations: the number of vertices effects the number of nesting spaces in a flag in $G/P_Q$ directly. Hence, using our Zelevinsky map for arbitrary orientations has obvious advantage than using K-R connections when we deal with these problems.

\subsection{Generalization for type D quiver}\label{gentoD}
In this section, we will provide an example for the generalization of Zelevinsky map for type D quiver. In \cite{bobinski_schubert_2002}, G. Bobiński and G. Zwara proved that the type of singularities of type D quiver loci are equivalent to the types of singularities of Schubert varieties in double Grassmannians. In \cite{kinser_type_2021}, Kinser and Rajchgot developed a generalized Zelevinsky map for type D quiver loci, showing that, up to an affine factor, every type D quiver locus is isomorphic to a linear slice of a Schubert variety in the double Grassmannian. Their approach is highly technical but, regrettably, exhibits significant inconsistencies with the type A cases in some aspects. Of course, this inconsistency is partly inherited from the choice of the double Grassmannian as the corresponding structure for the type D cases in \cite{bobinski_schubert_2002}. Why does Type A correspond to `flags' while Type D corresponds to double `Grassmannians'? This issue has created a significant divide in Kinser-Rajchgot's construction for types A and D, and it is not convenient for researchers to extend the Zelevinsky map to other quivers of Dynkin type, even when simply seeking objects to establish an equivalence of singularity. In this part, we will generalize the type A Zelevinsky map to type D, providing an example of subvarieties generalizing the Kazhdan-Lusztig varieties in a homogeneous space of special flags, which is isomorphic to type D quiver loci.

The homogeneous space we selected is closely related to the double Grassmannian as well. Compared to the isomorphism of Kinser and Rajchgot, this new isomorphism exhibits stronger consistency with type A and is a direct isomorphism without the need for smooth factors or taking open subvarieties. It is also more easily generalizable to the rest of the Dynkin types. Due to space constraints, we only provide here this example to demonstrate this method, with the main purpose still being to prove the application potential of the main content (i.e., the type A Zelevinsky map) in various aspects.

Let $Q$ be the type D quiver $$\begin{tikzcd}
1 \arrow[rd, "\beta_1"']     &                         &   &                         &                                               &                            \\
                             & 2 \arrow[r, "\beta_2"'] & 3 & 4 \arrow[l, "\alpha_1"] & 5 \arrow[l, "\alpha_2"] \arrow[r, "\beta_3"'] & 6  \\
1' \arrow[ru, "\beta_{1'}"'] &                         &   &                         &                                               &
\end{tikzcd},$$ and for any fixed dimension vector $\mathbf{d}$ the representations $V\in rep_Q(\mathbf{d})$ are presented as
$$\begin{tikzcd}
V_1 \arrow[rd, "\beta_1"', "B_1"]     &                         &   &                         &                                               &   &                         \\
                             & V_2 \arrow[r, "\beta_2"', "B_2"] & V_3 & V_4 \arrow[l, "\alpha_1", "A_1"'] & V_5 \arrow[l, "\alpha_2", "A_2"'] \arrow[r, "\beta_3"',"B_3"] & V_6  \\
V_{1'} \arrow[ru, "\beta_{1'}"', "B_{1'}"] &                         &   &                         &                                               &   &
\end{tikzcd},$$ Let $\tilde{Q}$ be the type A quiver$$\begin{tikzcd}
 1 \arrow[r ]    & 2 \arrow[r ] & 3 & 4 \arrow[l ] & 5 \arrow[l ] \arrow[r ] & 6
\end{tikzcd},$$  According to a type D quiver representation $V=(A_\cdot,B_\cdot)$ we define three type A quiver representations of $\tilde{Q}$ by $$\begin{tikzcd}
         V^{(1)}:=  V_{1} \arrow[r, "B_1"] & V_2 \arrow[r, "B_2"] & V_3 & V_4 \arrow[l, "A_1"'] & V_5 \arrow[l, "A_2"'] \arrow[r, "B_3"] & V_6
           \end{tikzcd}$$
           $$\begin{tikzcd}
         V^{(1')}:=  V_{1'} \arrow[r, "B_{1'}"] & V_2 \arrow[r, "B_2"] & V_3 & V_4 \arrow[l, "A_1"'] & V_5 \arrow[l, "A_2"'] \arrow[r, "B_3"] & V_6
           \end{tikzcd}.$$
           $$\begin{tikzcd}
         V^{(0)}:=  V_1\oplus V_{1'} \arrow[r, "{[B_1,B_{1'}]}"] & V_2 \arrow[r, "B_2"] & V_3 & V_4 \arrow[l, "A_1"'] & V_5 \arrow[l, "A_2"'] \arrow[r, "B_3"] & V_6
           \end{tikzcd}.$$Recall that the vertices of this type A quiver are $1,2,\cdots,6$, here we just put $V^{(1)}_{1}=V_1, V^{(1')}_{1}=V_{1'},V^{(0)}_{1}=V_1\oplus V_{1'}$. The $\mathrm{GL}(\mathbf{d})$-orbit (closure) of $V$ is featured by the rank parameter $\mathbf{r}$ that computed by Kinser and Rajchgot in \cite{kinser_type_2021}, consisting of ranks of the union of \begin{equation}\begin{array}{c} M_{[a,b]}(V^{(1)}), M_{[a,b]}(V^{(1')}),M_{[a,b]}(V^{(0)}), 1\leq a<b\leq n,\\\mbox{and } 
           M_{[2|3]}(V)=\left[\begin{matrix}
                                 B_2 B_1& 0\\
                                 B_1 & B_{1'}
                               \end{matrix}\right],
                               M_{[2|4]}(V)= \left[\begin{matrix}
                               A_1  &B_2B_1&0\\
                               0 &B_1& B_{1'}
                               \end{matrix}\right],\\
                               M_{[2|5]}(V)=\left[\begin{matrix}
                                 A_1A_2 & B_2B_1  &0\\
                                 0&B_1&B_{1'}\\
                               \end{matrix}\right],
                                M_{[2|6]}(V)=\left[\begin{matrix}
                                 B_3 & 0 &0 \\
                                 A_1A_2 &B_2B_1&0 \\
                                  0&B_1 & B_{1'}
                               \end{matrix}\right],\\
                               M_{[3|4]}(V)=\left[\begin{matrix}
                                 A_1 & B_2B_1 &0 \\
                                 A_1 & 0   & B_2B_{1'}
                               \end{matrix}\right],
                               M_{[3|5]}(V)=\left[\begin{matrix}
                                 A_1A_2& B_2 B_1& 0  \\
                                  A_1A_2 &0& B_2B_{1'}
                               \end{matrix}\right],\\
                               M_{[3|6]}(V)=\left[\begin{matrix}
                               B_3 & 0& 0\\
                                 A_1A_2 &B_2 B_1 & 0 \\
                                 A_1A_2 &0 &B_2B_{1'}\\
                               \end{matrix}\right],\\
                               M_{[4|5]}(V)=\left[\begin{matrix}
                                A_1A_2 &A_1 & B_2 B_1 & 0 \\
                               0 & A_1 &0 &B_2B_{1'}
                               \end{matrix}\right],\\
                               M_{[4|6]}(V)=\left[\begin{matrix}
                               B_3    &0   & 0 & 0\\
                                A_1A_2 &A_1 &B_2 B_1 & 0\\
                                    0  &A_1 &0       &B_2B_{1'}
                               \end{matrix}\right],\\M_{[5|6]}(V)=\left[\begin{matrix}
                                 B_3  &0& B_3 &0 \\
                                 A_1A_2 &B_2B_1& 0& 0 \\
                                 0 & 0 & B_2B_{1'} & A_1A_2
                               \end{matrix}\right].\end{array}\end{equation}
Here we directly present these matrices without discussing the general situations, and this result can be checked by applying \cite{kinser_type_2021} to this specific example. To sum up, the $\mathrm{GL}(\mathbf{d})$-orbit of $V$ is controlled by the ranks of\begin{equation}\label{Drankpar}\left\{\begin{array}{l}
                                                          M_{[a,b]}(V^{(1)}),1\leq a<b\leq 6\\
                                                          M_{[1,b]}(V^{(1')}),2\leq b\leq 6\\
                                                          M_{[1,b]}(V^{(0)}),2\leq b\leq 6\\
                                                          M_{[a|b]}(V),2\leq a<b\leq 6\end{array}\right.
                                                        \end{equation}

Now, let $$S'_Q=1\prec 1'\prec 2\prec 5\prec 6\prec 4\prec 3$$and $$S_Q=5\prec 4\prec 1\prec 1'\prec 2\prec 3\prec 6.$$ In the rest of this section, given a $(S'_Q,S_Q)$-blocked matrix $A$ we define $A^0{[x]\times[1']}$ to be matrix obtained by replacing column $[1]$ in $A_{[x]\times[1]}$ by the column $[1']$ of $A$. We define $A^{[x]}_{[y]}, x\prec y\in S'_Q$ by
$$A^{[x]}_{[y]}=\left[\begin{array}{c:c}
          A_{[x]\times[1]}\backslash A_{[q]\times[1]} & \multirow{2}{*}{$A^0{[x]\times[1']}$} \\\cdashline{1-1}
          0 & \\\cdashline{1-2}
          A_{[y]\times[1]}&0
        \end{array}\right],$$where $A_{[x]\times[1]}\backslash A_{[y]\times[1]}$ is the submatrix of $A_{[p]\times[1]}$ obtained by removing $A_{[q]\times[1]}$.

Let $N=\sum_{x\in Q_0} d_x$, $G=\mathrm{GL}(N)$ and $P$ be the parabolic group of block upper triangular $(S_Q,S_Q)$-blocked matrices. Let $H$ be the subgroup of $P$ consisting of $A\in P$ with $A_{[1],[1']}=0$. Then $G/H$ can be identified with the homogeneous space of special flags $$0\subset V_{(1)}\subset V_{(2)}\subset \begin{array}{c}V_{(3)}\\V'_{(3)}\end{array}\subset V_{(4)}\subset V_{(5)}\subset V_{(6)}=k^N$$ satsifying $$V_{(3)}/V_{(2)}\cap V'_{(3)}/V_{(2)}=0$$ with $$\dim V_{(1)}=d_5,\dim V_{(2)}=d_5+d_4,\dim V_{(3)}=d_5+d_4+d_1,\dim V'_{(3)}=d_5+d_4+d_{1'},$$ $$\dim V_{(4)}=d_5+d_4+d_1+d_{1'}+d_2,\dim V_{(5)}=d_5+d_4+d_1+d_{1'}+d_2+d_3.$$

The upper triangular Borel subgroup $B$ of $G$ acts on $G/H$ by natural product. By Gauss elimination, every representative $u$ of the double coset decomposition $G=\cup BuH$ can be determined by a permutation $w\in W^P$ and some pairs $\{(s_1,t_1),(s_2,t_2),\cdots\}$ satisfying that $1\leq s_p<d_1$ are pairwise distinct,  $1\leq t_p<d_{1'}\}$ are pairwise distinct and $w_{d_5+d_4+s_p}>w_{d_5+d_4+d_1+t_p}$. The $(i,j)$-entry $u_{ij}$ of $u$ is
$$\left\{\begin{array}{ll}
           1, & \mbox{if $w_{ij}=1$} \\
           1, & \mbox{if $(i,j)=(w_{d_5+d_4+s_p},d_5+d_4+d_1+t_p)$, }\\
           0, &\mbox{else}
         \end{array}\right.$$

Let $v_Q$ be the matrix $$v_Q=\begin{blockarray}{lccccccc}\begin{block}{l[ccccccc]}
                 1 & 0  & 0 &I & 0 &0 & 0&0 \\
                1' & 0  & 0 & 0 & I&0 &0 &0 \\
                 2 &  0 & 0 & 0 & 0 &I& 0& 0\\
                 5 & I & 0 & 0 & 0 &0 & 0& 0\\
                 6 & 0  & 0 & 0 & 0 & 0& 0& I\\
                 4 & 0  & I& 0 & 0 & 0& 0& 0\\
                 3 & 0  &0  & 0 & 0 & 0&I& 0\\
                  \end{block}
                   &5  &4 &1 &1'&2&3&6  \end{blockarray}$$
then the $B^-$-orbit $B^-v_QH/H\subset G/H$ can be identified with affine space of matrices like $$\begin{blockarray}{lccccccc}\begin{block}{l[ccccccc]}
                 1 & 0  & 0 &I  & 0 &0 & 0&0 \\
                1' & 0  & 0 & * & I &0 &0 &0 \\
                 2 &  0 & 0 & * & * &I & 0& 0\\
                 5 & I  & 0 & 0 & 0 &0 & 0& 0\\
                 6 & *  & 0 & * & * & *& 0& I\\
                 4 & *  & I & 0 & 0 & 0& 0& 0\\
                 3 & *  & * & * & * & *&I & 0\\
                  \end{block}
                   &5  &4   &1  &1 '& 2&3 &6  \end{blockarray}.$$ In this matrix space the intersection $\overline{BuH/H}\cap B^-v_QH/H$ has a defining ideal generated by minors of the generic matrix $Z$. If the representative $u$ satisfies certain maximal conditions (just like conditions of Z-type permutation, we do not present them here), the generators of defining ideal can be simplified to the union of
                    \begin{equation}\label{BuH}\left\{\begin{array}{l}\mbox{minors of size $u_{[x]\times[y]}+1$ in $Z_{[x]\times[y]}$},x,y\in Q_0\\
                    \mbox{minors of size $u^0_{[x]\times[1']}+1$ in $Z^0_{[x]\times[1']}$},x\in Q_0\\
                    \mbox{minors of size $u^{[x]}_{[y]}+1$ in $Z^{[x]}_{[y]}$},x\prec y\in S'_Q.\end{array}\right.\end{equation}

We define a map $\zeta_Q$ sending $V=(A_\cdot,B_\cdot)\in rep_Q(\mathbf{r})$ to matrix
 $$\zeta_Q(V)=\begin{blockarray}{lccccccc}\begin{block}{l[ccccccc]}
                 1 & 0  & 0 &I  & 0 &0 & 0&0 \\
                1' & 0  & 0 & 0 & I &0 &0 &0 \\
                 2 &  0 & 0 & B_1 & B_{1'} &I & 0& 0\\
                 5 & I  & 0 & 0 & 0 &0 & 0& 0\\
                 6 & B_3& 0 & 0 & 0 &0 & 0& I\\
                 4 & A_2 & I & 0 & 0 & 0& 0& 0\\
                 3 & 0  & A_1 & B_2B_1 & B_2B_{1'} & B_2&I & 0\\
                  \end{block}
                   &5  &4   &1  &1 '& 2&3 &6  \end{blockarray}.$$Then actually (the restriction of) $\zeta_Q$ provides an isomorphism from any $\overline{\mathcal{O}_{V}}$ to some $\overline{BuH/H}\cap B^-v_QH/H$. We define $\lambda_Q(x)=\lambda_{\tilde{Q}}(x),\mu_Q(x)=\mu_{\tilde{Q}}(x)$ for all $x$. One may observe that the conditions (\ref{Drankpar}) of rank parameter for type D quiver are parallel to the conditions (\ref{BuH}) according to the rules (in a way similar to Theorem \ref{main}): \begin{itemize}
                                   \item $M_{[a,b]}(V^{(1)})\leftrightarrow \zeta_Q(V)_{[\mu_Q(b)]\times[\lambda_Q(a)]}, 1\leq a<b\leq 6$;
                                   \item $M_{[1,b]}(V^{(1')})\leftrightarrow \zeta_Q(V)^0_{[\mu_Q(b)]\times[1']}, 2\leq b\leq 6$;
                                   \item $M_{[1,b]}(V^{(0)})\leftrightarrow \zeta_Q(V)_{[\mu_Q(b)]\times[1']},2\leq b\leq 6$;
                                   \item $M_{[a|b]}(V)\leftrightarrow \zeta_Q(V)^{[x]}_{[y]}$, where $x$ is the northern one in $\{\mu_Q(a),\mu_Q(b)\}$ and $y$ is the southern one, i.e. $\ind_{S'_Q}(x)=\min\{\ind_{S'_Q}(\mu_Q(a)),\ind_{S'_Q}(\mu_Q(b))\}$ and $\ind_{S'_Q}(y)=\max\{\ind_{S'_Q}(\mu_Q(a)),\ind_{S'_Q}(\mu_Q(b))\}$, $1\leq a<b\leq 6$.
                                 \end{itemize}
                   For example, $M_{[1,4]}(V^{(0)})\leftrightarrow \zeta_Q(V)_{[3]\times[1']}$, that is $$\zeta_Q(V)_{[3]\times[1]}=[0,A_1,B_2B_1,B_2B_{1'}]=[0,M_{[1,4]}(V^{(0)})];$$
                   $M_{[2|4]}(V)\leftrightarrow \zeta_Q(V)^{[2]}_{[3]}$, where $\mu_Q(2)=2,\mu_Q(4)=3$ and $2\prec 3$ in $S'_Q$, that is
                   $$\zeta_Q(V)^{[2]}_{[3]}=\left[\begin{matrix}
                                                    0 & 0 & B_1 & 0 & 0 & B_{1'} \\
                                                    I & 0 & 0 & I & 0 & 0 \\
                                                    B_3 & 0 & 0 & B_3 & 0 & 0 \\
                                                    A_2 & I & 0 & A_2 & I & 0 \\
                                                    0 & 0 & 0 & 0 & A_1 & B_2B_{1'}\\
                                                    0 & A_1 & B_2B_{1'} & 0 & 0 & 0 \\
                                                  \end{matrix}\right]$$$$\xrightarrow[transformation]{elementary}
                                                  \left[\begin{matrix}
                                                    0 & 0 & B_1 & 0 & 0 & B_{1'} \\
                                                    I & 0 & 0 & I & 0 & 0 \\
                                                    0 & 0 & 0 & 0 & 0 & 0 \\
                                                    0 & I & 0 & 0 & I & 0 \\
                                                    0 & -A_1 & -B_2B_{1'} & 0 & 0 & 0\\
                                                    0 & A_1 & B_2B_{1'} & 0 & 0 & 0 \\
                                                  \end{matrix}\right]\rightarrow \left[\begin{matrix}
                                                                                         M_{[2|4]}(V) &  &  \\
                                                                                          & I &  \\
                                                                                          & & 0
                                                                                       \end{matrix}\right];$$
                   $M_{[3|5]}(V)\leftrightarrow \zeta_Q(V)^{[5]}_{[4]}$, that is
                    $$\zeta_Q(V)^{[5]}_{[4]}=\left[\begin{matrix}
                                                     I & 0 & 0 & I & 0 & 0 \\
                                                     B_3 & 0 & 0 & B_3 & 0 & 0 \\
                                                     0 & 0 & 0 & A_2 & I & 0 \\
                                                     0 & 0& 0 & 0 & A_1 & B_2B_{1'} \\
                                                     A_2 & I & 0 & 0 & 0 & 0 \\
                                                     0& A_1 & B_2B_1 & 0 & 0 &
                                                   \end{matrix}\right]$$
                                                   $$\rightarrow \left[\begin{matrix}
                                                     I & 0 & 0 & I & 0 & 0 \\
                                                     0 & 0 & 0 & 0 & 0 & 0 \\
                                                     0 & 0 & 0 & 0 & I & 0 \\
                                                     A_1A_2 & 0& 0 & 0 & 0 & B_2B_{1'} \\
                                                     0 & I & 0 & 0 & 0 & 0 \\
                                                     -A_1A_2 & 0 & B_2B_1 & 0 & 0 &0
                                                   \end{matrix}\right]\rightarrow\left[\begin{matrix}
                                                                                         M_{[3|5]}(V) &  &  \\
                                                                                          & I &  \\
                                                                                         &  & 0
                                                                                       \end{matrix}\right].$$

Actually, this correspondence reveals the isomorphism between a type D quiver locus $\overline{\mathcal{O}_{V}}$ and some $\overline{BuH/H}\cap Bv_QH/H$. Therefore, it allows us to develop more geometric results on type D quiver loci through this isomorphism, similar to our discussion of type A in other parts of this section, in future.
\bibliographystyle{plain}
\bibliography{ref}

\begin{thebibliography}{10}

\bibitem{maxsingSL/B}
Maximal singular loci of {Schubert} varieties in $sl(n)/b$.
\newblock {\em Trans. Amer. Math. Soc.}, 355.

\bibitem{GbasisforKL}
A {Gröbner} basis for {Kazhdan-Lusztig} ideals.
\newblock {\em Amer. J. Math.}, 134(4):1089–1137, 2012.

\bibitem{degenerationofA}
S.~Abeasis and A.~Del Fra.
\newblock Degenerations for the representations of a quiver of type {$A_m$}.
\newblock {\em J. Algebra.}, 93(2):376–412, 1985.

\bibitem{reviewsingofschu}
S.~Billey and V.~Lakshmibai.
\newblock On the singular locus of a {Schubert} variety.
\newblock {\em J. Ramanujan Math. Soc.}, 15(3):155--223, 2000.

\bibitem{singofsch}
S.~Billey and V.~Lakshmibai.
\newblock {\em Singular loci of Schubert varieties.}
\newblock Number 182 in Progress in Mathematics. Birkhäuser Boston, Inc.,
  Boston, MA, 2000.

\bibitem{bobinski_normality_2001}
G.~Bobiński and G.~Zwara.
\newblock Normality of orbit closures for {Dynkin} quivers of type {$A_n$}.
\newblock {\em manuscripta mathematica}, 105(1):103--109, May 2001.

\bibitem{bobinski_schubert_2002}
G.~Bobiński and G.~Zwara.
\newblock Schubert varieties and representations of {Dynkin} quivers.
\newblock {\em Colloquium Mathematicum}, 94(2):285--309, 2002.

\bibitem{Brion2005}
M.~Brion.
\newblock {\em Lectures on the Geometry of Flag Varieties}, pages 33--85.
\newblock Birkh{\"a}user Basel, Basel, 2005.

\bibitem{ESS}
W.~Fulton.
\newblock Flags, schubert polynomials, degeneracy loci, and determinantal
  formulas.
\newblock {\em Duke Math. J.}, 65(3):381--420, 1992.

\bibitem{singcom1}
N.~Gonciulea.
\newblock Singular loci of varieties of complexes.
\newblock {\em J. Algebra}, 235(2):547–558, 2001.

\bibitem{kazhdan_representations_1979}
D.~Kazhdan and G.~Lusztig.
\newblock Representations of {Coxeter} groups and {Hecke} algebras.
\newblock {\em Inventiones Mathematicae}, 53(2):165--184, 1979.

\bibitem{SchubertPoincare}
D.~Kazhdan and G.~Lusztig.
\newblock {\em {Schubert} varieties and {Poincaré} duality.}, pages 185--203.
\newblock Amer. Math. Soc., Providence, R.I., 1979.

\bibitem{kinser_type_2015}
R.~Kinser and J~. Rajchgot.
\newblock Type {A} quiver loci and {Schubert} varieties.
\newblock {\em Journal of Commutative Algebra}, 7(2):265–301, 2015.

\bibitem{kinser_type_2021}
R.~Kinser and J~. Rajchgot.
\newblock Type {D} quiver representation varieties, double {Grassmannians}, and
  symmetric varieties.
\newblock {\em Advances in Mathematics}, 376(107454):44 pp., 2021.

\bibitem{GVDandliaison}
P.~Klein and J.~Rajchgot.
\newblock Geometric vertex decomposition and liaison.
\newblock {\em Forum Math. Sigma}, 9(e70):23 pp., 2021.

\bibitem{fourformulae}
A.~Knutson and E.~Miller.
\newblock Four positive formulae for type {A} quiver polynomials.
\newblock {\em Invent. math.}, 166(2):229–325, 2006.

\bibitem{singcom2}
V.~Lakshmibai.
\newblock Singular loci of varieties of complexes. {II}.
\newblock {\em J. Pure Appl. Algebra}, 152(1-3):217–230, 2000.

\bibitem{flagvarieties}
V.~Lakshmibai and J.~Brown.
\newblock {\em Flag varieties. An interplay of geometry, combinatorics, and
  representation theory.}
\newblock Number~53 in Texts and Readings in Mathematics. Hindustan Book
  Agency, Delhi, second edition edition, 2018.

\bibitem{lakdegesche1998}
V.~Lakshmibai and P.~Magyar.
\newblock Degeneracy schemes, quiver schemes, and {Schubert} varieties.
\newblock {\em Internat. Math. Res. Notices}, (12):627–640, 1998.

\bibitem{lusztig1990cir}
G.~Lusztig.
\newblock Canonical bases arising from quantized enveloping algebras.
\newblock {\em J. Amer. Math. Soc.}, 3(2):447–498, 1990.

\bibitem{lusztigcanbase}
G.~Lusztig.
\newblock Canonical bases and {Hall} algebras.
\newblock {\em NATO Adv. Sci. Inst. Ser. C: Math. Phys. Sci.}, 514:365–399,
  1998.

\bibitem{threeformulas}
A.~Knutson R.~Kinser and J.~Rajchgot.
\newblock Three combinatorial formulas for type {A} quiver polynomials and
  {K-polynomials}.
\newblock {\em Duke Math. J.}, 168(4):505--551, 2019.

\bibitem{definingidealofA}
C.~Riedtmann and G.~Zwara.
\newblock Orbit closures and rank schemes.
\newblock {\em Comment. Math. Helv.}, 88(1):55--84, 2013.

\bibitem{introtoSMT}
C.~S. Seshadri.
\newblock {\em Introduction to the theory of standard monomials.}
\newblock Number~46 in Texts and Readings in Mathematics. Hindustan Book
  Agency, New Delhi,, second edition edition.

\bibitem{springer_linear_1998}
T.~A. Springer.
\newblock {\em Linear {Algebraic} {Groups}}.
\newblock Birkhäuser Boston, Boston, MA, 1998.

\bibitem{zelevinskii_two_1985}
A.~V. Zelevinsky.
\newblock Two remarks on graded nilpotent classes.
\newblock {\em Russian Mathematical Surveys}, 40(1):249--250, 1985.

\end{thebibliography}
\end{document}